\documentclass[11.5pt,twoside, a4paper, english, reqno]{amsart} 
\usepackage{amsaddr}
\usepackage{a4wide}
\usepackage{amscd}
\usepackage{amssymb}
\usepackage{amsthm}
\usepackage{graphicx}
\usepackage{amsmath,calrsfs}
\usepackage{latexsym}
\usepackage{enumitem,dsfont}

\usepackage{upref}
\usepackage[colorlinks]{hyperref}
\usepackage{color,graphics}
\usepackage{upref,hyperref,color}

\usepackage[usenames,dvipsnames]{xcolor}
\usepackage{pdfsync}
\usepackage{ulem,cancel,pgf}

\usepackage{frcursive}

\theoremstyle{plain}
\newtheorem{theorem}{Theorem}[section]
\theoremstyle{plain}
\newtheorem{lemma}[theorem]{Lemma}

\newtheorem{corollary}[theorem]{Corollary}

\theoremstyle{definition}
\newtheorem{definition}{Definition}[section]

\newtheorem{remark}{Remark}[section]
\newtheorem{claim}{Claim}[section]

\newtheorem*{maintheorem*}{Main Theorem}
\newtheorem*{maincorollary*}{Main Corollary}

\DeclareFontFamily{U}{BOONDOX-calo}{\skewchar\font=50 }
\DeclareFontShape{U}{BOONDOX-calo}{m}{n}{
	<-> s*[1.05] BOONDOX-r-calo}{}
\DeclareFontShape{U}{BOONDOX-calo}{b}{n}{
	<-> s*[1.05] BOONDOX-b-calo}{}
\DeclareMathAlphabet{\mathcalb}{U}{BOONDOX-calo}{m}{n}
\SetMathAlphabet{\mathcalb}{bold}{U}{BOONDOX-calo}{b}{n}
\DeclareMathAlphabet{\mathbcalb}{U}{BOONDOX-calo}{b}{n}

\numberwithin{equation}{section} \allowdisplaybreaks

\title[Uniqueness Results]
{Uniqueness Results for Mixed Local and Nonlocal Equations with Singular Nonlinearities and Source Terms}
\date{\today}

\author[Abdelhamid Gouasmia]{Abdelhamid Gouasmia}
\address[Abdelhamid Gouasmia]{
Department of Mathematics, Faculty of Sciences And Technology,\\
Mohamed Cherif Messaadia University,\\
P.O.Box 1553, Souk Ahras 41000, Algeria.\\[4pt]
Laboratoire d'equations aux d\'{e}riv\'{e}es partielles non lin\'{e}aires et histoire des math\'{e}matiques,\\
Ecole Normale Sup\'{e}rieure,\\
B.P. 92, Vieux Kouba, 16050 Algiers, Algeria.}
\email[Abdelhamid Gouasmia]{gouasmia.abdelhamid@gmail.com}

\usepackage{comment}
\begin{document}
	
\begin{abstract}
	This paper considers a local and non-local problem characterized by singular nonlinearity and a source term. Specifically, we focus on the following problem:
	\begin{equation}\label{A}\tag{P}
	-\Delta_{p} u + (-\Delta)^{s}_{q} u = f(x) u^{-\alpha} + g(x) u^{\beta}, \quad u > 0 \quad \text{in } \Omega; \quad u = 0, \quad \text{in } \mathbb{R}^{N} \setminus \Omega,
	\end{equation}
	where \( \Omega \subset \mathbb{R}^N \) is an open bounded domain with a \( C^{2} \) boundary \( \partial \Omega \), and \( N > p \). We assume that \( 0 < s < 1 \) and \( 1 < p, q < \infty \), with the conditions \( q = p \) or \( q < p \), corresponding to the homogeneous and non-homogeneous cases, respectively. The parameters satisfy \( 0 < \beta < q - 1 \) and \( \alpha > 0 \). The function \( f \) is non-zero and belongs to a suitable Lebesgue space \( L^{r}(\Omega) \) for some \( r \in [1, \infty] \), or satisfies a growth condition involving negative powers of the distance function \( d(\cdot) \) near the boundary \( \partial \Omega \). Additionally, \( g \) is a nonnegative function within appropriate Lebesgue spaces. The primary objectives of this paper are twofold. First, we establish the uniqueness of infinite energy solutions to problem \eqref{A} by introducing a novel comparison principle under certain conditions. Second, we derive several existence results for weak solutions in various senses, accompanied by regularity results for problem \eqref{A}. Furthermore, we present a non-existence result when the function \( f(x) \sim d^{-\delta}(x) \) and \( x \) is near the boundary, under the condition \( \delta \geq p \). Our approach leverages the Picone identities on one hand and the interaction between the local and non-local terms on the other hand.
\end{abstract}

\maketitle
\textbf{Keywords:}	singular nonlinearity, mixed local and nonlocal operators, uniqueness results, existence results
\\[2mm]
\hspace*{0.45cm}\textbf{MSC:} 35A01,  35B65,  35J75, 35M12.  \\

\setcounter{tocdepth}{1}
\tableofcontents
\section{Introduction}
In their seminal work in the 1960s, Fulks and Maybee \cite{fulks1960singular} introduced a class of singular problems related to the steady-state temperature distribution in an electrically conducting medium. This is described by the following governing equation for \( T > 0 \):
\begin{equation}\label{Iequ1}
	\mathbf{c} u_{t}(x, t) - \kappa \Delta u(x, t) = \dfrac{E^{2}(x, t)}{\sigma(u(x, t))} \quad \text{in } \Omega \times (0, T),
\end{equation}
where \( \Omega \subset \mathbb{R}^{3} \) represents the electrically conducting medium, \( u(x, t) \) denotes the steady-state temperature distribution at time \( t \) and position \( x \in \Omega \), \( \sigma(u(x, t)) \) denotes the resistivity of the material, and \( E(x, t) \) represents the local voltage drop in \( \Omega \) as a function of position and time. The parameters \( \mathbf{c} \) and \( \kappa \) correspond to the specific heat capacity of the conductor and the thermal conductivity of \( \Omega \), respectively. As a result, the rate of heat generation at any point \( x \) in the medium and at time \( t \) is given by \( E^{2}(x, t) \sigma(u(x, t))^{-1} \), implying that the temperature distribution in the conducting medium satisfies equation \eqref{Iequ1}.\\[4pt]
It is also noteworthy that singular problems have applications in various contexts, including pseudo-plastic fluids, Chandrasekhar equations in radiative transfer, non-Newtonian fluid flows in porous media, and heterogeneous catalysts. Moreover, these problems are relevant in a wide range of fields, encompassing real-world models in gaseous dynamics within astrophysics, relativistic mechanics, nuclear physics, and the study of chemical reactions. Additionally, they relate to phenomena such as glacial advance, the transportation of coal slurries along conveyor belts, and numerous other geophysical and industrial applications, as discussed in \cite{diaz1987elliptic, fowler1930emden, nachman1980nonlinear, schowalter1960application}, among others. For more insights into the derivation of specific models and their applications, we refer readers to \cite{oliva2024singular}, which provides detailed discussions, including the Blasius model and others.\\[4pt]
On the other hand, the investigation of singular equations has increasingly attracted attention as a notable mathematical challenge. The first systematic treatment of these issues was presented in \cite{crandall1977dirichlet, stuart1976}. Specifically, \cite{crandall1977dirichlet} analyzed the singular boundary value problem defined as follows:
\begin{equation*}
-\Delta u = \textbf{F}(x, u), \quad u = 0 \quad \text{on } \partial\Omega,
\end{equation*}
where \( \textbf{F}(x, s) \) exhibits singular behavior at \( s = 0 \). The authors established results regarding the existence and uniqueness of classical solutions by employing the classical method of sub- and supersolutions applied to a non-singular approximating problem. Earlier, \cite{fulks1960singular} addressed this topic and similarly established existence and uniqueness results for problem \eqref{Iequ1}. Subsequently, a substantial body of research has focused on exploring various types of these problems. For an overview of the developments and key results, we refer to \cite{ref08, barrios2015semilinear, boccardo2010semilinear, ref51, canino2017moving, ref01, carmona2023regularizing, diaz1987elliptic, durastanti2022comparison, giacomoni2021sobolev, giacomoni2018positive,  giacomoni2012singular, oliva2024singular, ref57, zhang2004existence} and the references therein, which provide comprehensive insights into the study of both the \( p \)-Laplace operator, defined for \( p > 1 \) as  $ \Delta_{p} u = \text{div} \left( \left|\nabla u \right|^{p-2} \nabla u \right), $ and the fractional \( q \)-Laplacian, which for a fixed parameter \( s \in (0, 1) \) is defined by
\begin{align*}
( -\Delta ) ^{s}_{q} u(x) &:= 2\, \textbf{P.V.}\int_{\mathbb{R}^N}\frac{|u(x) - u(y) |^{q-2} ( u(x) - u(y)) }{| x-y|^{N + s q} }\,dy,\\[4pt]
&= 2 \lim_{\epsilon \to 0} \int_{\left\lbrace y \in \mathbb{R}^{N}\, : \, | y -x| \geq \epsilon\right\rbrace}\frac{|u(x) - u(y) |^{q-2} ( u(x) - u(y)) }{| x-y|^{N + s q} }\,dy,
\end{align*}
where \( q > 1 \). These operators incorporate the nonlinear term \( \textbf{F}(x, u) \), which encompasses three types of nonlinearities: a purely singular nonlinearity \( \textbf{F}(x, u) = \textbf{f}_{1}(x) u^{-\alpha} \), a singular nonlinearity with a source term \( \textbf{F}(x, u) = \textbf{f}_{1}(x) u^{-\alpha} + \textbf{f}_{2}(x, u) \), and a singular nonlinearity with an absorption term \( \textbf{F}(x, u) = \textbf{f}_{1}(x) u^{-\alpha} - \textbf{f}_{2}(x, u) \), where \( \alpha > 0 \) and \( \textbf{f}_{1} \) belongs to an appropriate Lebesgue space or/and behaves like \( \text{dist}^{-\delta}(x, \partial\Omega) \) with \( \delta > 0 \). It is assumed that the function \( \textbf{f}_{2} \) is a suitable non-negative continuous function.\\[4pt]
Recently, problems involving \textbf{mixed} local and non-local operators (referring to operators combining the $p$-Laplace operator and the fractional \( q \)-Laplacian, as defined above) have garnered significant attention due to their wide-ranging applications in real-world phenomena, particularly in physical processes arising from mixed \textbf{dispersal strategies}. The term \textbf{dispersal} generally refers to the movement of a biological population, with density denoted by \( u \), competing for resources within a given environment \( \Omega \). This movement can occur through various mechanisms, including both local and non-local dispersal. In \cite{kao2012evolution}, the authors investigate the influence of mixed dispersal on species invasion and the evolution of dispersal strategies in spatially periodic yet temporally constant environments. Furthermore, the study in \cite{dipierro2022non} introduces a model describing the diffusion of a biological population within an ecological niche, subject to both local and non-local dispersal. For additional insights and applications, see \cite{chen2008evolution} and \cite{dipierro2021description}. Motivated by these considerations, the analysis of elliptic problems involving mixed operators, particularly those containing singular terms, has become a focal point of research. In the singular case, we start from the work of \cite{garain2022mixed}, where the authors establish the existence and uniqueness of a weak solution for the following mixed local and nonlocal $p$-Laplace equation:
\begin{equation}\label{P5}
-\Delta_{p} u + (-\Delta)^{s}_{p} u = f(x) u^{-\alpha}, \quad u > 0 \quad \text{in}\, \Omega; \quad u = 0, \quad \text{in}\, \mathbb{R}^{N} \setminus \Omega.
\end{equation}
In fact, the existence of a distributional solution is established through an approximation approach, which relies on the regularity of the datum $f$ and the singular exponent $\alpha$, as outlined below:
\begin{itemize}
	\item[$\bullet$] If $0 < \alpha < 1$ and $f \in L^{\left( \frac{p^{*}}{1 - \alpha}\right)'}(\Omega)$, then $u \in W^{1, p}_{0}(\Omega)$.
	\item[$\bullet$] If $\alpha = 1$ and $f \in L^{1}(\Omega)$, then $u \in W^{1, p}_{0}(\Omega)$.
	\item[$\bullet$] If $\alpha > 1$ and $f \in L^{1}(\Omega)$, then $u \in W^{1, p}_{\text{loc}}(\Omega)$, and $u^{\frac{\alpha + p -1}{p}} \in W^{1, p}_{0}(\Omega)$.
\end{itemize}
In \cite{ref48}, the authors enhance the findings presented in \cite{garain2022mixed}. These enhancements include properties related to existence and non-existence, Sobolev regularity results of both power and exponential types, and the boundary behavior of the weak solution. This advancement is explored through the interplay between the summability of the datum and the power exponent in singular nonlinearities for problem \eqref{P5}, specifically with \( p = 2 \), addressing two cases for the function \( f \) (Lebesgue weights and singular weights) (see also \cite{bhowmick2024existence, huang2023lazer} for related issues). Additionally, uniqueness results for the class of singular weights are established in \cite[Theorem 2.1]{gouasmia2024singular}. The technique employed in the proof of the uniqueness result presented in this paper, as well as in \cite{garain2022mixed}, traces back to \cite{canino2016uniqueness} within the context of the local semilinear case. This method has been extensively utilized in various studies to establish uniqueness results in scenarios where the solution belongs to \( W^{s, p}_{\text{loc}}(\Omega) \), with \( p > 1 \) and \( s \in \left( 0, 1 \right] \). For further reading on this topic, we refer to \cite{ref08, arora2022large, ref01, ref65}, although this is not an exhaustive list. It is important to note that the selection of test functions in this method does not ensure the validity of the approach when addressing problems involving singular nonlinearity with a \textbf{source} term.

For this reason, and given the limited results available regarding local and non-local operator equations involving singular nonlinearity with a source term together, this paper investigates the following problem:
\begin{equation}\label{P2}\tag{P}
-\Delta_{p} u + (-\Delta)^{s}_{q} u = f(x) u^{-\alpha} + g(x) u^{\beta}, \quad u > 0 \quad \text{in } \Omega; \quad u = 0, \quad \text{in } \mathbb{R}^{N} \setminus \Omega,
\end{equation}
where \( \Omega \subset \mathbb{R}^N \), with \( N > p \) and a boundary \( \partial \Omega \) that is \( C^{2} \). Assume that \( 0 < s < 1 \), \( 1 < p, q < \infty \), such that \( q = p \) or \( q < p \), corresponding to the homogeneous and non-homogeneous cases, respectively; \( 0 < \beta < q - 1 \); \( \alpha > 0 \); and that \( g \) is an appropriate non-negative measurable function. Regarding the function \( f: \Omega \to \mathbb{R}^{+} \), we consider the following cases:
\begin{itemize}
	\item[$ \textbf{(F1)} $] \( f \) is not identically zero and belongs to a suitable Lebesgue space \( L^{r}(\Omega) \) for some \( r \in \left[ 1, \infty \right] \).
	\item[$ \textbf{(F2)} $] \( f \) satisfies the following growth condition: for any \( x \in \Omega \)  
	\begin{equation}\label{equ24} 
	\textbf{C}_{1} d(x)^{-\delta} \leq f(x) \leq \textbf{C}_{2} d(x)^{-\delta},
	\end{equation}
	for some \( \delta \in \left[ 0, p \right) \), with \( \textbf{C}_{1} \) and \( \textbf{C}_{2} \) being positive constants, where \( d(x) := \text{dist}(x, \partial \Omega) \). 
\end{itemize}
\textbf{The main purpose of our paper is two-part:}\\
$ \bullet $ Firstly, we address the uniqueness of infinite energy solutions to problem \eqref{P2} (if such solutions exist), which arises from a new comparison principle established in this paper under certain conditions (see Theorem \ref{Theorem1} below). This result is achieved by refining the method in \cite{canino2016uniqueness} and employing Picone's identities to tackle the difficulties arising from the local and non-local terms in our problem, along with selecting suitable test functions (see \cite{brezis1986remarks, durastanti2022comparison} for the case of local operators and finite energy solutions, i.e., \( u \in W^{1, p}_{0}(\Omega) \)), while also benefiting from the condition on \( \beta \) mentioned earlier. Furthermore, this proof can be extended to more general settings, as follows:
\begin{equation*}
-\Delta_{p} u + (-\Delta)^{s}_{q} u = h(x, u), \quad u > 0 \quad \text{in } \Omega; \quad u = 0 \quad \text{in } \mathbb{R}^{N} \setminus \Omega,
\end{equation*}
where \( h : \Omega \times \mathbb{R}^{+} \to \mathbb{R}^{+} \) is a Carathéodory function, and we assume that for a.e. \( x \in \Omega \), the map \( s \mapsto \frac{h(x, s)}{s^{q-1}} \) is non-increasing on \( \mathbb{R}^{+} \setminus \{0\} \). This is further explored in \textbf{Section 2}, in the context of finite energy solutions.\\[4pt]
$ \bullet  $ Secondly, we establish several existence results for weak solutions in various senses, alongside regularity results for problem \eqref{P2}. Our primary focus is on the class of weight functions \( f \) that satisfy the conditions \textbf{(F1)} and \textbf{(F2)}, in addition to functions \( g \) that meet certain regularity criteria. Furthermore, we present a non-existence result specifically for the class \textbf{(F2)} when \( \delta > p \). To demonstrate the existence results, we employ a classical regularization approach, as the energy functional associated with problem \eqref{P2} is not \( C^{1} \) due to the presence of a singular term. Consequently, minimax results from critical point theory become inapplicable. In this context, we draw upon results obtained in the paper \cite{antonini2023global}, which addresses regularity results and includes Hopf's type lemma for positive supersolutions, involving mixed local and non-local operators. This combination of operators is leveraged in certain estimates. Additionally, we generate a uniqueness result in the case where \( g = 0 \), which is not addressed in the initial part of our study.
\subsection{Notations and basics} 
First, we establish some notations that will be used throughout this paper:\\
For \( t \in \mathbb{R} \), we denote \( [t]^{p-1} := |t|^{p-2} t \) for \( p > 1 \), and we define \( t^{\pm} = \max\{\pm t, 0\} \). Additionally, for a fixed \( k > 0 \), we introduce the truncation function \( \mathbf{T}_{k} \) as follows: 
\[
\mathbf{T}_{k}(s) := \min\{s, k\} \quad \text{for} \quad s \geq 0, \quad \text{and} \quad \mathbf{T}_{k}(s) := - \mathbf{T}_{k}(-s) \quad \text{for} \quad s < 0.
\]
For \( r > 1 \), we denote by \( r' = \frac{r}{r-1} \) the conjugate exponent of \( r \). The constants \( C \) and \( C_{i} \) will represent general constants that may vary from line to line or even within the same line. If \( C \) depends on the parameters \( N, p, \ldots \), we write \( C = C(N, p, \ldots) \). Moreover, \( \textbf{B}_{1}(x) \) denotes the ball of radius \( 1 \) centered at the point \( x \).\\[4pt]
Next, we employ definitions and properties associated with specific function spaces. Before that, we assume that \( \Omega \subset \mathbb{R}^{N} \) (with \( N > p \)) is a bounded domain with \( C^{2} \) boundary \( \partial\Omega \). Now, we consider a measurable function \( u : \mathbb{R}^{N} \to \mathbb{R} \) and adopt the following:\\
Let \( p \in [1, +\infty[ \). The norm in the Lebesgue space \( L^{p}(\Omega) \) is defined as follows:
\[
\left\| u \right\|_{L^{p}(\Omega)} := \left( \int_{\Omega} \left| u \right|^{p} \, dx \right)^{\frac{1}{p}}.
\]
The space \( L^{p}_{\text{loc}}(\Omega) \) denotes the space of locally \( L^{p} \)-integrable functions, while \( L^{\infty}_{c}(\Omega) \) denotes the space of \( L^{\infty} \)-functions with compact support in \( \Omega \). \\
We recall that the Sobolev space \( W^{1, p}(\mathbb{R}^{N}) \) is defined as
\[
W^{1, p}(\mathbb{R}^{N}) := \left\{ u \in L^{p}(\mathbb{R}^{N}) \quad \vert \quad \nabla u \in L^{p}(\mathbb{R}^{N}) \right\},
\]
equipped with the norm \( \left\| u \right\|^{p}_{W^{1, p}(\mathbb{R}^{N})} = \left\| u \right\|^{p}_{L^{p}(\mathbb{R}^{N})} + \left\| \nabla u \right\|^{p}_{L^{p}(\mathbb{R}^{N})} \). The space \( W_{0}^{1, p}(\Omega) \) is defined as the set of functions under the norm 
\[
\left\| u \right\|_{W_{0}^{1, p}(\Omega)} := \left\| \nabla u \right\|_{L^{p}(\Omega)},
\]
given by
\[
W_{0}^{1, p}(\Omega) := \left\{ u \in W^{1, p}(\mathbb{R}^{N}) \quad \vert \quad u = 0 \, \text{ a.e. in } \, \mathbb{R}^{N} \setminus \Omega \right\}.
\]
Now, we have the following Sobolev embedding:
\begin{theorem}[{\cite[Theorem 3]{evans2022partial}}] \label{thm0}
	Let \( p \geq 1 \) with \( N > p \). Then, there exists a positive constant \( C = C (N, p, \Omega) \) such that for any measurable function \( u \in W^{1, p}_{0}(\Omega) \), we have
	\[
	\| u \| _{L^{p^{*}}(\Omega)} \leq C \| u \| _{W^{1, p}_{0}(\Omega)},
	\]
	where \( p^{*} \) is the Sobolev critical exponent, defined as \( \frac{Np}{N - p} \).
\end{theorem}
\begin{remark}\label{remark0}
	In light of Theorem \ref{thm0}, it is easy to see that \( \|\cdot\| _{W^{1, p}(\mathbb{R}^N)} \) and \( \|\cdot\| _{W_0^{1, p}(\Omega)} \) are equivalent norms on \( W_0^{1, p}(\Omega) \). According to the results in \cite{evans2022partial}, we have that \( W^{1,p}_0(\Omega) \) is continuously embedded in \( L^r(\Omega) \) when \( 1 \leq r \leq p^{*} \), and compactly for \( 1 \leq r < p^{*} \).
\end{remark}
\noindent  Let \( 0 < s < 1 \) and \( q > 1 \). The fractional Sobolev space \( W^{s, q}(\mathbb{R}^{N}) \) is defined by:
\[
W^{s, q}(\mathbb{R}^{N}) :=\left\lbrace u\in L^{q}(\mathbb{R}^{N}),\quad [u]^{q}_{s, q} := \displaystyle \iint_{\mathbb{R}^{2N}} \dfrac{\vert u(x)-u(y)\vert ^{q}}{\vert x-y\vert^{N+s q} }dxdy < \infty \right\rbrace,
\]
and it is endowed with the norm \( \Vert u\Vert _{W^{s, q}(\mathbb{R}^{N})}^{q} := \Vert u\Vert^{q}  _{L^{q}(\mathbb{R}^{N})} + [u]^{q}_{s, q} \). The space \( W_{0}^{s, p}(\Omega) \) is the set of functions defined as:
\[
W_{0}^{s, q}(\Omega) :=\left\lbrace u\in  W^{s, q}\left( \mathbb{R}^{N}\right) \quad \vert \quad u=0 \quad \text{a.e. in}\quad \mathbb{R}^{N}\setminus\Omega \right\rbrace.
\]
The associated Banach norm in the space \( W_{0}^{s, q}(\Omega) \) is given by the Gagliardo semi-norm:
\[
\Vert u\Vert _{W_{0}^{s, q}(\Omega )}:=[u]_{s, q}.
\]
However, the space \( W_{0}^{s, q}(\Omega) \) can be treated as the closure of \( C^{\infty}_{c} (\Omega) \) with respect to the norm \( [u]_{s, q} \) if \( \partial \Omega \) is smooth enough, where
\[
C^{\infty}_{c} (\Omega) := \left\lbrace \varphi\, : \, \mathbb{R}^{N} \to \mathbb{R}\,:\, \varphi \in C^{\infty}(\mathbb{R}^{N}) \quad \text{and} \quad \text{supp}(\varphi) \Subset \Omega \right\rbrace.
\]
In the following, we recall the fractional Poincaré inequality:
\begin{theorem}[{\cite[Theorem 6.5]{ref03}}] \label{thm3}
	Let \( s \in (0,1) \), \( q \geq 1 \) with \( N > sq \). Then, there exists a positive constant \( C = C (N, q, s, \Omega) \) such that, for any measurable and compactly supported \( u : \mathbb{R}^N \to \mathbb{R} \) function, we have
	\[
	\| u\| _{L^{q^{*}_{s}}(\mathbb{R}^N)}^{q} \leq C \iint_{\mathbb{R}^{2N}}  \dfrac{\vert u(x)-u(y)\vert ^{q}}{\vert x-y\vert^{N+s q} }dxdy,
	\]
	where \( q^{*}_{s} \) is the fractional Sobolev critical exponent, defined as \( \frac{Nq}{N - sq} \).
\end{theorem}
\begin{remark}\label{remark1}
	In light of Theorem \ref{thm3}, it is easy to see that \( \|\cdot\| _{W^{s, q}(\mathbb{R}^N)} \) and \( \|\cdot\| _{W_0^{s, q}(\Omega )} \) are equivalent norms on \( W_0^{s, q}(\Omega) \). According to the results in \cite{ref03}, we have that \( W^{s,q}_0(\Omega) \) is continuously embedded in \( L^r(\Omega) \) when \( 1 \leq r \leq q^{*}_{s} \), and compactly for \( 1 \leq r < q^{*}_{s} \).
\end{remark}
\noindent Now, we consider the space \( \mathcal{W}^{1, p}(\Omega) \), defined as
\[
\mathcal{W}^{1, p}(\Omega) = \left\{ u \in W^{1,p}(\mathbb{R}^{N}) \, : \, u|_{\Omega} \in W^{1, p}_{0}(\Omega), \, u = 0 \, \text{ a.e. in } \mathbb{R}^{N} \setminus \Omega \right\}.
\]
Using \cite[Proposition 9.18]{brezis2011functional}, we can identify \( \mathcal{W}^{1, p}(\Omega) \) with the space \( W^{1, p}_{0}(\Omega) \) if \( \Omega \) admits a \( C^{1} -\)boundary.\\
On the other hand, we denote by \( d(\cdot) \) the distance function up to the boundary \( \partial\Omega \). That means
\[
d(x)  := \text{dist}(x, \partial\Omega) = \inf_{y \in \partial\Omega}\left| x - y\right|.
\]
As stated in \cite{ref08}, we smoothly extend the distance function $ d(\cdot) $ in \( \Omega^{\text{c}} = \mathbb{R}^{N} \setminus \Omega \), for \( \rho > 0 \), as follows:
\begin{align*}
d_{e}(x) :=
\begin{cases}
d(x)  & \text{ if } x \in \Omega,\\
- d(x)  & \text{ if } x \in (\Omega^{\text{c}})_{\rho},\\
- \rho & \text{ otherwise}.
\end{cases}
\end{align*}
Here, \( (\Omega^{\text{c}})_{\rho} = \left\{ x \in \Omega^{\text{c}} : \text{dist}(x, \partial\Omega) < \rho \right\} \). Hence, for \( \gamma, \rho > 0 \) and \( \kappa \geq 0 \), we define:
\begin{align*}
\underline{w}_{\rho}(x) =
\begin{cases}
(d_{e}(x) + \kappa^{\frac{1}{\gamma}})^{\gamma}_{+} - \kappa & \text{ if } x \in \Omega \cup (\Omega^{\text{c}})_{\rho},\\
- \kappa & \text{ otherwise}.
\end{cases}
\end{align*}
\begin{theorem}[{\cite[Theorem 3.3]{ref08}}]\label{theorem2} 
	There exist \( \kappa_{*}, \varrho_{*} > 0 \) such that for all \( \kappa \in \left[0, \kappa_{*}\right) \) and \( \gamma \in (0, s) \), there exists a positive constant \( M \) such that for all \( \varrho \in (0, \varrho_{*}) \):
	\[
	(- \Delta)^{s}_{q}\underline{w}_{\rho}  \leq M (d(x) + \kappa^{\frac{1}{\gamma}})^{\gamma (q-1) - qs} \quad \text{ weakly in } \Omega_{\varrho},
	\]
	where \( \Omega_{\varrho} = \left\{ x \in \Omega : d(x) < \varrho \right\}. \) Further, for all \( \kappa > 0 \) and \( \gamma \in (0, s) \), \( \underline{w}_{\rho} \in W^{s, q}(\Omega_{\varrho}) \).
\end{theorem}
\noindent We will utilize the following lemmas in the upcoming discussions, starting with some useful inequalities:
\begin{lemma}[{\cite{ref06}}] \label{lemma1} 
	For \( p > 1 \) and for all \( \xi, \xi' \in \mathbb{R}^{N} \) with \( \left| \xi \right| + \left| \xi' \right| > 0 \), we have:
	\begin{align*}
	\left(\left| \xi \right|^{p-2} \xi - \left| \xi' \right|^{p-2} \xi'\right) \cdot \left(\xi - \xi'\right) &\geq C(p) \left( \left| \xi \right| + \left| \xi' \right| \right)^{p - 2} \left| \xi - \xi' \right|^{2},  \\[4pt]
	\left| \left| \xi \right|^{p-2} \xi - \left| \xi' \right|^{p-2} \xi' \right|  &\leq C(p) \left( \left| \xi\right| + \left| \xi' \right| \right)^{p - 2} \left| \xi - \xi' \right|,  \\[4pt]
	\left(\left| \xi \right|^{p-2} \xi - \left| \xi' \right|^{p-2}\xi'\right) \cdot \left(\xi  - \xi'\right) &\geq C(p) \left|\xi - \xi' \right|^{p}, \quad \text{ if } p \geq 2.
	\end{align*}
\end{lemma}
\noindent Additionally, we require the following lemma, which was proved in \cite{Stampacchia1965}:
\begin{lemma}\label{Lemma5}
	Let \( \Psi : \mathbb{R}^{+} \to  \mathbb{R}^{+} \) be a non-increasing function such that 
	\[
	\Psi(h) \leq \frac{C}{(h - k)^{\eta}} \Psi(k)^{\delta}, \quad \text{ for all } h > k > 1,
	\]
	where \( C > 0 \), \( \delta > 1 \), and \( \eta > 0 \). Then \( \Psi(d) = 0 \), where \( d^{\eta} = C \Psi(0)^{\delta - 1} 2^{\frac{\delta \eta}{\delta - 1}} \).
\end{lemma}
\noindent By the Discrete Picone inequality \cite[Proposition 4.1]{brasco2014convexity}, \cite{giacomoni2021existence} proved the following weak comparison principle:
\begin{lemma}\label{Lemma2}
	Let \( 1 < q < \infty \). Then, for \( 1 < r \leq q \) and for any \( u, v \), two Lebesgue measurable and positive functions in \( \Omega \), we have:
{\small 	\begin{equation*}
	\begin{aligned}
	&\left[ u(x) - u(y) \right]^{q-1} \left( \frac{u(x)^{r} - v(x)^{r}}{u(x)^{r-1}} - \frac{u(y)^{r} - v(y)^{r}}{u(y)^{r-1}} \right) + \left[ v(x) - v(y) \right]^{q-1} \left( \frac{v(x)^{r} - u(x)^{r}}{v(x)^{r-1}} - \frac{v(y)^{r} - u(y)^{r}}{v(y)^{r-1}} \right) \geq 0,
	\end{aligned}
	\end{equation*}}
for a.e. $ x, y \in \Omega. $
\end{lemma}
\noindent Similarly, we can obtain the same result in the local case, as outlined below. For the reader's convenience, we provide some details of the proof.
\begin{lemma}\label{Lemma3}
	Let \( 1 < p < \infty \). Then, for \( 1 < r \leq p \) and for any \( u, v \) two positive differentiable functions in \( \Omega \), we have:
	\begin{equation}\label{locequ0} 
	\left[\nabla u\right] ^{p-1} \nabla\left( \frac{u^{r} - v^{r}}{u^{r-1}} \right) +  \left[\nabla v\right]^{p-1} \nabla\left( \frac{v^{r} - u^{r}}{v^{r-1}} \right) \geq 0, 
	\end{equation}
for all  $ x \in \Omega. $
\end{lemma}
\begin{proof}
Let \( u \) and \( v \) be two differentiable functions satisfying \( u, v > 0 \) in \( \Omega \) and \( 1 < r \leq p \). Now, by applying \cite[Proposition 2.9]{brasco2014convexity}, we obtain
	\begin{equation} \label{pic}
	\left[ \nabla u \right]^{p-1} \nabla\left( \frac{v^r}{u^{r-1}} \right) \leq \left| \nabla v \right|^r \left| \nabla u \right|^{p - r}.
	\end{equation}
	Let us start with the case \( r = p \). By applying the above inequality, we obtain
	\begin{equation} \label{Le1}
	\left[ \nabla u \right]^{p-1} \nabla \left( \frac{u^{p} - v^{p}}{u^{p-1}} \right) \geq \left| \nabla u \right|^{p} - \left| \nabla v \right|^{p}.
	\end{equation}
	By exchanging the roles of \( u \) and \( v \), we obtain
	\begin{equation}\label{Le2}
	\left[ \nabla v \right]^{p-1} \nabla \left( \frac{v^{p} - u^{p}}{v^{p-1}} \right) \geq \left| \nabla v \right|^{p} - \left| \nabla u \right|^{p}.
	\end{equation}	
	Combining \eqref{Le1} and \eqref{Le2}, we get
	\begin{align*}
	\left[ \nabla u \right]^{p-1} \nabla \left( \frac{u^{p} - v^{p}}{u^{p-1}} \right) + \left[ \nabla v \right]^{p-1} \nabla \left( \frac{v^{p} - u^{p}}{v^{p-1}} \right) \geq 0,
	\end{align*}
	which concludes the proof of \eqref{locequ0} for \( r = p \). Now we deal with the case \( 1 < r < p \). By using Young's inequality, \eqref{pic} implies
	\begin{equation*} 
	\left[ \nabla u \right] ^{p-1} \nabla \left( \frac{u^{r} - v^{r}}{u^{r-1}} \right) \geq \frac{r}{p} \left( \left| \nabla u \right|^{p} - \left| \nabla v \right|^{p} \right).
	\end{equation*}
	Reversing the roles of \( u \) and \( v \), we have
	\begin{equation*} 
	\left[ \nabla v \right] ^{p-1} \nabla \left( \frac{v^{r} - u^{r}}{v^{r-1}} \right) \geq \frac{r}{p} \left( \left| \nabla v \right|^{p} - \left| \nabla u \right|^{p} \right).
	\end{equation*}
	Adding the above inequalities gives the desired result.
\end{proof}
\noindent  Now, we have the following result:
\begin{lemma}\label{Lemma1}
	For any \( s \in (0, 1) \), \( 1 < p < \infty \), and \( 1 < q \leq p \), there exists a constant \( \textbf{C}_{1} = \textbf{C}_{1}(N, p, q, s, \Omega) \) such that
	\[
	\left\| u \right\|_{W^{s, q}_{0}(\Omega)} \leq \textbf{C}_{1} \left\| u \right\|_{W^{1, p}_{0}(\Omega)} \quad \text{for every } u \in \mathcal{W}^{1, p}(\Omega).
	\]
	Moreover, there exists a constant \( \textbf{C}_{2} = \textbf{C}_{2}(N, p, q, s, \Omega) > 0 \) such that
	\[
	\left\| u \right\|_{W^{s, q}(\Omega)} \leq \textbf{C}_{2} \left\| u \right\|_{W^{1, p}(\Omega)} \quad \text{for every } u \in \mathcal{W}^{1, p}(\Omega).
	\]
\end{lemma}
\begin{proof}
	Since the case \( p = q \) is addressed in \cite[Lemma 2.1]{buccheri2022system}, we will focus on the case where \( 1 < q < p \). Initially, we observe that for all \( u \in \mathcal{W}^{1, p}(\Omega) \):
	\[
	\iint_{\mathbb{R}^{2N}} \dfrac{\vert u(x)-u(y)\vert^{q}}{\vert x-y\vert^{N+ s q}} \, dx \, dy = \underbrace{\iint_{\Omega \times \Omega} \dfrac{\vert u(x)-u(y)\vert^{q}}{\vert x-y\vert^{N+ s q}} \, dx \, dy}_{\textbf{I}_{1}} + 2 \underbrace{\iint_{\Omega \times \Omega^{c}} \dfrac{\vert u(x)-u(y)\vert^{q}}{\vert x-y\vert^{N+ s q}} \, dx \, dy}_{\textbf{I}_{2}}.
	\]
\textbf{Estimate of \( \textbf{I}_{1} \).} Following the approach in the proof of \cite[Proposition 2.2]{ref03}, we apply the variable change \( z = y - x \) and utilize Hölder's inequality, the Poincaré inequality (Theorem \ref{thm0}), and the convexity of \( \tau \mapsto \tau^p \) to obtain:
	\begin{equation*}
	\begin{aligned}
	\textbf{I}_{1} &= \iint_{\Omega \times (\Omega \cap \textbf{B}_{1}(x))} \dfrac{\vert u(x)-u(y)\vert^{q}}{\vert x-y\vert^{N+ s q}} \, dx \, dy + \iint_{\Omega \times (\Omega \cap \textbf{B}_{1}(x)^{c})} \dfrac{\vert u(x)-u(y)\vert^{q}}{\vert x-y\vert^{N+ s q}} \, dx \, dy \\[4pt]
	& \leq \iint_{\Omega \times \textbf{B}_{1}(0)} \dfrac{\vert u(x)-u(x + z)\vert^{q}}{\vert z\vert^{N+ s q}} \, dxdz + 2^{q} \iint_{\Omega \times (\Omega \cap \textbf{B}_{1}(x)^{c})} \dfrac{\left| u(x) \right|^{q}}{\vert x-y\vert^{N+ s q}} \, dx \, dy \\[4pt]
	& \leq \iint_{\Omega \times \textbf{B}_{1}(0) \times [0, 1]} \dfrac{\left| \nabla u(x + t z)\right|^{q}}{\vert z\vert^{N+ s q - q}} \, dx \, dz \, dt + 2^{q} C(N, p, q, \Omega) \left\| u\right\|_{L^{p}(\Omega)}^{q} \\[4pt]
	& \leq C(N, p, q, \Omega) \left\| \nabla u\right\|_{L^{p}(\Omega)}^{q}.
	\end{aligned}
	\end{equation*}	
	\textbf{Estimate of \( \textbf{I}_{2} \).} This can be estimated in exactly the same manner as before, leading to:
	\[
	\textbf{I}_{2} \leq C(N, p, q, \Omega) \left\| \nabla u\right\|_{L^{p}(\Omega)}^{q}.
	\]	
Thus, we conclude the proof of the assertion. For the second part of this lemma, we use \cite[Proposition 2.2]{ref03} and the H\"{o}lder inequality to obtain the desired result.
\end{proof}
\noindent  Finally, we present a rigorous result on the weak comparison principle:
\begin{lemma}\label{Lemma4}
	For any \( s \in (0, 1) \), \( 1 < p < \infty \), and \( 1 < q \leq p \), let \( u, v \in W^{1, p}(\Omega) \) satisfy \( u \leq v \) in \( \mathbb{R}^{N} \setminus \Omega \). Then, for all \( \varphi \in W^{1, p}_{0}(\Omega) \) with \( \varphi \geq 0 \) in \( \Omega \), the following holds:
	\begin{equation*}
	\begin{aligned}
	&\int_{\Omega} \left[ \nabla u \right]^{p-1} \nabla \varphi \, dx + \iint_{\mathbb{R}^{2N}} \frac{\left[ u(x) - u(y) \right]^{q-1} \left( \varphi(x) - \varphi(y) \right)}{|x - y|^{N + sq}} \, dx \, dy \\[4pt]
	& \qquad \leq \int_{\Omega} \left[ \nabla v \right]^{p-1} \nabla \varphi \, dx + \iint_{\mathbb{R}^{2N}} \frac{\left[ v(x) - v(y) \right]^{q-1} \left( \varphi(x) - \varphi(y) \right)}{|x - y|^{N + sq}} \, dx \, dy.
	\end{aligned}
	\end{equation*}
	Then, it follows that \( u \leq v \) in \( \Omega \).
\end{lemma}
\begin{proof}
	The proof follows along the same lines as \cite[Lemma 2.6]{da2020limiting}.
\end{proof}
\subsection{Statements of main results} 
In the first part of this work, we present a new comparison principle that holds independent interest. Prior to this, we introduce the following definitions of weak sub-solutions, super-solutions, and solutions for problem \eqref{P2}.
\begin{definition}\label{definition1}
	We say that a function \( u \in W^{1, p}_{\text{loc}}(\Omega) \) is a \textbf{weak super-solution} to \eqref{P2} if the following conditions are satisfied:
	\begin{enumerate}
		\item[(i)]  There exists a constant \( \theta \geq 1 \) such that \( u^{\theta} \in W^{1, p}_{0}(\Omega) \).
		\item[(ii)] For any $ K \Subset  \Omega, $ there exists a constant $ C(K) > 0 $ such that
		$ u \geq C(K) \text{ in }\, K.$
		\item[(iii)] For all $ \varphi \in W^{1, p}_{0}(\Omega)\cap L^{\infty}_{c}(\Omega),$ with $  \varphi \geq 0 ,$ we have
		{\small 	\begin{equation*}
			\begin{gathered}
			\begin{aligned}
			\int_{\Omega} \left[ \nabla u\right] ^{p-1} \nabla \varphi dx &+  \displaystyle \displaystyle\iint_{\mathbb{R}^{2N}} \dfrac{\left[ u(x) - u(y) \right] ^{q-1} \left(\varphi (x)- \varphi(y)\right)}{\vert x-y\vert ^{N+ s q}}dxdy\geq {\displaystyle\int_{\Omega}} \left( f(x) u^{-\alpha} + g(x) u^{\beta}\right) \varphi dx.
			\end{aligned}
			\end{gathered}
			\end{equation*}}
	\end{enumerate}
	If \( u \) satisfies the reversed inequality, it is referred to as a \textbf{weak sub-solution} of \eqref{P2}. Moreover, a function \( u \) that simultaneously satisfies the conditions of both a weak sub-solution and a weak super-solution of \eqref{P2} is defined as a \textbf{weak solution} to \eqref{P2}.
\end{definition}
\begin{remark}
	Lemma \ref{Lemma1} ensures that Definition \ref{definition1} is well defined.
\end{remark}
\begin{remark}\label{remark2}
An important observation regarding Definition \ref{definition1} is that the solution $u$ generally does not belong to the space $W^{1, p}_{0}(\Omega)$. Furthermore, it should be emphasized that no trace operator exists in $W^{1, p}_{\text{loc}}(\Omega)$. For this reason, we adopt the following definition to interpret the Dirichlet condition in a generalized sense (see \cite[Definition 1.2]{canino2016existence}): \textbf{"}We say that $u \leq 0$ on $\partial \Omega$ if $u = 0$ in $\mathbb{R}^{N}\setminus \Omega$ and $(u - \epsilon)^{+} \in W^{1, p}_{0}(\Omega)$ for every $\epsilon > 0$. Moreover, we declare $u = 0$ on $\partial \Omega$ if $u \geq 0$ and $u \leq 0$ on $\partial \Omega$. \textbf{"} 

\noindent It is crucial to note that Definition \ref{definition1} - \textbf{(i)} ensures that the solution satisfies the conditions of this definition. In particular, since $u^{\theta} \in W^{1, p}_{0}(\Omega)$ for $\theta \geq 1$, there exists a sequence of non-negative functions $\varphi_{n} \in C^{\infty}_{c}(\Omega)$ such that $\varphi_{n} \to u^{\theta}$ in $W^{1, p}_{0}(\Omega)$. Defining $\psi_{n} := (\varphi_{n}^{\frac{1}{\theta}} - \epsilon)^{+}$, we obtain
\begin{align*}
\left\| \psi_{n} \right\|^{p}_{W^{1, p}_{0}(\Omega)} = \int_{\left\lbrace \varphi_{n}^{\frac{1}{\theta}} > \epsilon \right\rbrace} \left| \nabla \varphi_{n}^{\frac{1}{\theta}} \right|^{p} dx < \epsilon^{p(1-\theta)} \theta^{-p} \int_{\left\lbrace \varphi_{n}^{\frac{1}{\theta}} > \epsilon \right\rbrace} \left| \nabla \varphi_{n} \right|^{p} dx < C,
\end{align*}
where $C > 0$ is independent of $n$. Therefore, $(\psi_{n})$ is uniformly bounded in $W^{1, p}_{0}(\Omega)$. Consequently, by the reflexivity of $W^{1, p}_{0}(\Omega)$, it follows that $(u - \epsilon)^{+} \in W^{1, p}_{0}(\Omega)$ for every $\epsilon > 0$.
\end{remark}
\begin{theorem}\label{Theorem1}
Assume that \( g \in L^{\left(\frac{p^{*} }{\beta +1}\right) '}(\Omega) \). Consider the case where the class of \textbf{(F1)} holds and one of the following conditions is satisfied:
\begin{itemize}
	\item[(H1)] \( \alpha < 1 \) and \( f \in L^{\left(\frac{p^{*}}{1 - \alpha}\right)'}(\Omega) \).
	\item[(H2)] \( \alpha = 1 \) and \( f \in L^r(\Omega) \) for some \( r > 1 \).
	\item[(H3)] \( \alpha > 1 \) and \( f \in L^1(\Omega) \).
\end{itemize}
When the class of \textbf{(F2)} holds, assume that
\begin{itemize}
	\item[(H4)] \( \delta < 1 + \frac{1}{p'} \).
\end{itemize}
Let  \( \underline{u}, \overline{u} \in W^{1, p}_{\text{loc}}(\Omega) \) be weak sub-solution and super-solution of the problem \eqref{P2}, respectively, in the sense of Definition \ref{definition1}. Then \( \underline{u} \leq \overline{u} \) a.e. in \( \Omega \).
\end{theorem} 
\noindent  As a consequence of the comparison principle, we obtain the following uniqueness result:
\begin{corollary}\label{corollary}
	Assume that the hypotheses 	(H1)-(H4) in Theorem \ref{Theorem1} are satisfied. Then, the weak solution to the problem \eqref{P2} in the sense of Definition \ref{definition1}, if it exists, is unique.
\end{corollary}
\noindent We now underscore a direct consequence of the uniqueness result:
\begin{corollary}\label{corollary2}
	Assume that the conditions (H1)-(H4) of Theorem \ref{Theorem1} are satisfied, and let $u$ be the unique solution, if it exists, to problem \eqref{P2} in the sense of Definition \ref{definition1}. Suppose that the domain $\Omega$ is symmetric with respect to the hyperplane 
	$$ \mathcal{H}^{\upsilon} _{\lambda} := \left\lbrace x \cdot \upsilon = \lambda \right\rbrace, \quad \lambda \in \mathbb{R}, \quad \upsilon \in \textbf{S}^{N -1}. $$ 
	If, in addition, both $f$ and $g$ are symmetric with respect to the hyperplane $\mathcal{H}^{\upsilon} _{\lambda}$, then the solution $u$ inherits this symmetry. In particular, when $\Omega$ is a ball or an annulus centered at the origin, and $f$ and $g$ are radially symmetric, the solution $u$ is also radially symmetric.
\end{corollary}
\noindent In the second part of this work, we investigate the existence and qualitative properties of weak solutions to \eqref{P2} for the class of weight functions \( f \) characterized by \textbf{(F1)} and \textbf{(F2)}, along with a non-existence result in the case of \textbf{(F2)}. The solutions are considered within the framework of Definition \ref{definition1}, which guarantees a uniqueness result under certain conditions (see Corollary \ref{corollary}), alongside other pertinent definitions. Specifically, we have:\\[4pt]
$ \bullet $ \textbf{For the class of weight functions satisfying} \textbf{(F1)}:
\begin{theorem}\label{Theorem3}
	Let \( 0 < \alpha < 1 \) and suppose that \( f \in L^{r}(\Omega) \) with \( 1 \leq r < r_{p, \alpha} \), where
	\begin{equation}\label{equ77}
	r_{p, \alpha} := \dfrac{pN}{N(p-1) + p + \alpha (N - p)} = \left(\frac{p^{*}}{1 - \alpha}\right)' .
	\end{equation}
	Furthermore, assume that \( g \in L^{m_{s, p, q, \alpha, \beta, r}}(\Omega) \), where
	\begin{equation}\label{equ41}
	m_{s, p, q, \alpha, \beta, r} := \dfrac{N(p(\vartheta_{r} - 1) + q)}{sqp(\vartheta_{r} - 1) + Nq - (\beta + 1)(N- sq)},
	\end{equation}
	with
	\begin{equation}\label{equ79}
	\vartheta_{r} =  \frac{r(N - sq)( \alpha + p - 1) - N(r-1)(p-q)}{p(N - rsq)}.
	\end{equation}
	Under these conditions, there exists a positive weak solution \( u \) to problem \eqref{P2} such that:
	\begin{itemize}
		\item[\textbf{(1)}] \( u \in W^{1, \varrho_{s, p, q, \alpha, r}}_{0}(\Omega) \), where
		\begin{equation}\label{equ50}
		\varrho_{s, p, q, \alpha, r} := \dfrac{N p \vartheta_{r}}{N \vartheta_{r} + (1 - \vartheta_{r})(N - p)}.
		\end{equation}
		\item[\textbf{(2)}] For every \( \omega \Subset \Omega \), there exists a constant \( C = C(\omega) > 0 \) such that \( u \geq C \) in \( \omega \).
		\item[\textbf{(3)}] For all \( \varphi \in C_{c}^{\infty}(\Omega) \), the solution satisfies the following identity:
		\begin{equation*}
		\begin{aligned}
		&\int_{\Omega} \left[ \nabla u \right] ^{p-1}   \nabla \varphi \, dx + \iint_{\mathbb{R}^{2N}} \frac{\left[ u(x) - u(y)\right]^{q-1} \left( \varphi(x) - \varphi(y) \right)}{\vert x - y \vert^{N + sq}} \, dx \, dy \\
		&= \int_{\Omega} f(x) u^{-\alpha} \varphi \, dx + \int_{\Omega} g(x) u^{\beta} \varphi \, dx.
		\end{aligned}
		\end{equation*}
	\end{itemize}
	Moreover, the solution satisfies the Sobolev regularity \( u^{\vartheta_{r}} \in W^{1, p}_{0}(\Omega) \). Additionally, \( u \) belongs to \( L^{\frac{N r (\alpha + q - 1)}{N - sqr}}(\Omega) \).
\end{theorem}
\begin{remark}
	From the specified range of \( r \) in this theorem, it follows that \( 1 < \varrho_{s, p, q, \alpha, r} < p \) and \( \vartheta_{r} < 1 \). Consequently, \eqref{P2} admits solutions with infinite energy, but not in the space \( W^{1, p}_{\text{loc}}(\Omega) \). It is also important to note that Corollary \ref{corollary} does not address the uniqueness result in this case.
\end{remark}
It is worth noting that we establish the uniqueness result specifically in the case when \( g = 0 \) for this theorem. More precisely, we have the following theorem:
\begin{theorem}\label{Theorem6}
	Assume that \( g \equiv 0 \) in problem \eqref{P2}. Let \( u_{1} \) and \( u_{2} \) be two weak solutions to problem \eqref{P2} with \( f_{1}, f_{2} \in L^{r}(\Omega) \) for \( 1 \leq r < r_{p, \alpha} \), where \( r_{p, \alpha} \) is defined as in \eqref{equ77}. Then, there exists a constant \( \textbf{C } > 0\), independent of \( u_{1} \) and \( u_{2} \), such that
	\begin{equation}\label{equ65}
	\left\| \left( (u_{1} - u_{2})^{+} \right)^{\vartheta_{r}} \right\| _{W^{1, p}_{0}(\Omega)} \leq \textbf{C }\left\| (f_{1} - f_{2})^{+} \right\|^{\frac{\vartheta_{r}}{p + \alpha - 1}} _{L^{r}(\Omega)},
	\end{equation}
where \( \vartheta_{r}  \) is defined as in \eqref{equ79}.
\end{theorem}
\begin{remark}
	The inequality in \eqref{equ65} ensures the uniqueness of the weak solution to problem \eqref{P2} when \( g \equiv 0 \), as stated in Theorem \ref{Theorem3}.
\end{remark}
\begin{theorem}\label{Theorem4}
	If either of the following conditions holds:\\[4pt]
	\textbf{(i)} \( 0 < \alpha < 1 \), and \( f \in L^{r}(\Omega) \) with \( r_{p, \alpha} \leq r \leq \infty \), where \( r_{p, \alpha} \) is defined as in \eqref{equ77}. \\[4pt]
	\textbf{(ii)} \( \alpha = 1 \), and \( f \in L^{r}(\Omega) \) with \( r \in \left[1, \infty\right] \).\\[4pt]
	Moreover, suppose that \( g \in L^{\left( \frac{p^{*}}{\beta + 1} \right)'}(\Omega) \). Then, there exists a positive weak solution \( u \) to problem \eqref{P2} in the following sense:
	\begin{itemize}
		\item[\textbf{(1)}] \( u \in W^{1, p}_{0}(\Omega) \).
		\item[\textbf{(2)}] For every \( \omega \Subset \Omega \), there exists a constant \( C = C(\omega) > 0 \) such that \( u \geq C \) in \( \omega \).
		\item[\textbf{(3)}] For all \( \varphi \in C_{c}^{\infty}(\Omega) \), the solution satisfies the following:
		\begin{equation*}
		\begin{aligned}
		&\int_{\Omega} \left[ \nabla u \right]^{p-1} \nabla \varphi \, dx + \iint_{\mathbb{R}^{2N}} \frac{\left[ u(x) - u(y)\right]^{q-1} \left( \varphi(x) - \varphi(y) \right)}{\vert x - y \vert^{N + sq}} \, dx \, dy  = \int_{\Omega} f(x) u^{-\alpha} \varphi \, dx + \int_{\Omega} g(x) u^{\beta} \varphi \, dx.
		\end{aligned}
		\end{equation*}
	\end{itemize}
Moreover, if \( \frac{N}{p} < r <  \frac{p^{*}}{\beta} \) and the function \( g \) satisfies the additional regularity condition \( g \in L^{\frac{p^{*} r}{p^{*} - \beta r}}(\Omega) \), then the solution \( u \) belongs to \( L^{\infty}(\Omega) \).
\end{theorem}
\begin{remark}
	According to Theorem \ref{Theorem3}, we observe that for any $\alpha \in (0, 1)$, there is a form of \textbf{continuity} in the summability exponent $\varrho_{s, p, q, \alpha, r}$. Specifically, as $r \to r_{p, \alpha}$, where $r_{p, \alpha}$ is defined in \eqref{equ77}, we have $\vartheta_{r} \to 1$, with $\vartheta_{r}$ defined in \eqref{equ79}, and $\varrho_{s, p, q, \alpha, r} \to p$, where $\varrho_{s, p, q, \alpha, r}$ is defined in \eqref{equ50}.
\end{remark}
\begin{theorem}\label{Theorem5}
	Let \( \alpha > 1 \) and suppose that \( f \in L^{r}(\Omega) \) with \( 1 \leq r < \frac{N}{sq} \), and \( g \in L^{\left( \frac{p^{*}}{\beta + 1} \right)'}(\Omega) \). Then, there exists a positive weak solution \( u \) to problem \eqref{P2} in the following sense:
	\begin{itemize}
		\item[\textbf{(1)}] \( u \in W^{1, p}_{\text{loc}}(\Omega) \).
		\item[\textbf{(2)}] \( u^{\vartheta_{r}} \in W^{1, p}_{0}(\Omega), \) where \( \vartheta_{r} \) is defined as in \eqref{equ79}.
		\item[\textbf{(3)}] For every \( \omega \Subset \Omega \), there exists a constant \( C = C(\omega) > 0 \) such that \( u \geq C \) in \( \omega \).
		\item[\textbf{(4)}] For all \( \varphi \in C_{c}^{\infty}(\Omega) \), the solution satisfies the following:
		\begin{equation*}
		\begin{aligned}
		&\int_{\Omega} \left[ \nabla u \right]^{p-1} \nabla \varphi \, dx + \iint_{\mathbb{R}^{2N}} \frac{\left[ u(x) - u(y) \right]^{q-1} \left( \varphi(x) - \varphi(y) \right)}{\vert x - y \vert^{N + sq}} \, dx \, dy = \int_{\Omega} f(x) u^{-\alpha} \varphi \, dx + \int_{\Omega} g(x) u^{\beta} \varphi \, dx.
		\end{aligned}
		\end{equation*}
	\end{itemize}
	Additionally, \( u \) belongs to \( L^{\frac{N r (\alpha + q - 1)}{N - sqr}}(\Omega) \).
\end{theorem}

\noindent $ \bullet $ \textbf{For the class of weight functions satisfying}  \textbf{(F2)}:
\begin{theorem}\label{theorem1}
	Let \( \delta + \alpha < 1 + \frac{1}{p'} \) and \( g \in L^{\left(\frac{p^{*}}{\beta + 1}\right)'}(\Omega) \). Then, there exists a positive weak solution \( u \) of the problem \eqref{P2} in the following sense:
	\begin{itemize}
		\item[\textbf{(1)}] \( u \in W^{1, p}_{0}(\Omega) \).
		\item[\textbf{(2)}] For every \( \omega \Subset \Omega \), there exists a constant \( C = C(\omega) > 0 \) such that \( u \geq C \) in \( \omega \).
		\item[\textbf{(3)}] For all \( \varphi \in C_{c}^{\infty}(\Omega) \), the solution satisfies the following equation:
		\begin{equation*}
		\begin{aligned}
		&\int_{\Omega} \left[ \nabla u\right] ^{p-1} \nabla \varphi \, dx + \iint_{\mathbb{R}^{2N}} \frac{\left[ u(x) - u(y) \right] ^{q-1} \left( \varphi(x) - \varphi(y) \right)}{|x - y|^{N + sq}} \, dx \, dy  = \int_{\Omega} f(x) u^{-\alpha} \varphi \, dx + \int_{\Omega} g(x) u^{\beta} \varphi \, dx.
		\end{aligned}
		\end{equation*}
	\end{itemize}
\end{theorem}
\begin{theorem}\label{theorem3}
	Assume that \( \theta_{0} \) satisfies
\begin{equation}\label{equ121}
	\theta_{0} > \max\left\lbrace 1, \frac{(\alpha + p - 1)(p - 1)}{p(p - \delta)}, \frac{\alpha + p - 1}{p}\right\rbrace.
\end{equation}
	If \( \delta + \alpha \geq 1 + \frac{1}{p'} \) and \( g \in L^{\left(\frac{\theta_{0} p^{*}}{p(\theta_{0} - 1) + \beta + 1}\right) '}(\Omega) \), then there exists a positive weak solution \( u \) of the problem \eqref{P2} in the following sense:
	\begin{itemize}
		\item[\textbf{(1)}] \( u  \in W^{1, p}_{\text{loc}}(\Omega) \).
		\item[\textbf{(2)}] \( u^{\theta_{0}} \in W^{1, p}_{0}(\Omega) \).
		\item[\textbf{(3)}] For every \( \omega \Subset \Omega \), there exists a constant \( C = C(\omega) > 0 \) such that \( u \geq C \) in \( \omega \).
		\item[\textbf{(4)}] For all \( \varphi \in C_{c}^{\infty}(\Omega) \), the solution satisfies the following equation:
		\begin{equation*}
		\begin{aligned}
		&\int_{\Omega} \left[ \nabla u \right]^{p-1} \nabla \varphi \, dx + \iint_{\mathbb{R}^{2N}} \frac{\left[ u(x) - u(y)\right]^{q-1} \left( \varphi(x) - \varphi(y) \right)}{\vert x - y \vert^{N + sq}} \, dx \, dy  = \int_{\Omega} f(x) u^{-\alpha} \varphi \, dx + \int_{\Omega} g(x) u^{\beta} \varphi \, dx.
		\end{aligned}
		\end{equation*}
	\end{itemize}
Moreover, if \( \delta > p(1 - s) + s(1 - \alpha) \), the solution achieves optimal Sobolev regularity, specifically:
\[
u \in W^{1, p}_{0}(\Omega) \text{ if and only if } \delta < 1 + \dfrac{1-\alpha}{p'}.
\]
\end{theorem}
\begin{theorem}\label{theorem4}
	Let \( \delta \geq p \). Then, no weak solution exists for the problem \eqref{P2} in the sense described in Theorems \ref{theorem1} and \ref{theorem3}.
\end{theorem}
\noindent \textbf{This article is organized as follows:} In \textbf{Section 2}, we establish the uniqueness result derived from the comparison principle, specifically proving Theorem \ref{Theorem1} along with additional related results. \textbf{Section 3} presents preliminary results for the approximated problem, which will be utilized throughout the remainder of the paper. Finally, we provide proofs of our second results concerning existence, non-existence, and other qualitative properties.
\section{Uniqueness results}
Our goal in this section is to establish the uniqueness of weak solutions to problem \eqref{P2} (if solutions exist). In fact, this result will follow from a more general weak comparison principle for weak sub and super-solutions of problem \eqref{P2}, taking into account its source term, in accordance with Definition \ref{definition1},  namely Theorem \ref{Theorem1}. The aim now is to prove this theorem by employing a variational approach as outlined in \cite{ref51, ref57}.
\begin{proof}[\textbf{Proof of Theorem} \ref{Theorem1}] 
	First, we define the following convex and closed set:
	$$ \mathcal{K} :=\left\lbrace  \phi \in W^{1, p}_{0}(\Omega)\, : \, 0 \leq \phi \leq \overline{u}\, \text{ almost everywhere in } \Omega \right\rbrace.$$
	Next, for each $ \epsilon \in (0, 1) $ fixed,  we define the functional $ \mathcal{J}_{\epsilon} :  \mathcal{K}  \to \mathbb{R} ,$ as follows 
	\begin{align*}
		\mathcal{J}_{\epsilon}(w) & := \dfrac{1}{p} \int_{\Omega} \left| \nabla w \right| ^{p}dx +\dfrac{1}{q}  \iint_{\mathbb{R}^{2N}}\dfrac{\left| w(x) - w(y)\right| ^{q}}{\left| x - y\right| ^{N + s q}} dxdy  - \int_{\Omega} G_{\epsilon}(x, w) dx,
	\end{align*}
	where 
	\begin{equation}\label{comequ13}
		G_{\epsilon}(x, w) :=
		\displaystyle\int_{0}^{w}\left(f(x) (s + \epsilon)^{- \alpha} + g(x) (s + \epsilon)^{\beta}\right)  ds.
	\end{equation}
At this stage, we need to establish three claims. Firstly, we have
\begin{claim}
For any \( \epsilon > 0 \), there exists a minimizer \( w_{0} \) of \( \mathcal{J}_{\epsilon} \) in \( \mathcal{K} \) such that, for every \( \psi \in w_{0} + (W^{1, p}_{0}(\Omega) \cap L_{c}^{\infty}(\Omega)) \) with \( \psi \in \mathcal{K} \), we have
\begin{equation} \label{comequ8}
			\begin{gathered}
				\begin{aligned}
					\int_{\Omega} &\left[ \nabla w_{0}\right] ^{p-1}  \nabla (\psi - w_{0})\, dx +  \displaystyle\iint_{\mathbb{R}^{2 N}} \dfrac{\left[w_{0}(x) - w_{0}(y) \right] ^{q-1}\left((\psi - w_{0}) (x)- (\psi - w_{0})(y)\right)}{\vert x-y\vert ^{N+ s q}}dx dy \\[4pt]
					& \geq {\displaystyle\int_{\Omega}}\left(f(x)\,(w_{0} + \epsilon)^{-\alpha} + g(x) (w_{0} + \epsilon)^{\beta} \right) \,(\psi - w_{0})\,dx.
				\end{aligned}
			\end{gathered}
		\end{equation}
	\end{claim}
	\noindent In fact, we consider two distinct cases:\\[4pt]
If \textbf{(F1)} holds, it follows from Lemma \ref{Lemma1} and the Sobolev embedding theorem (see Theorem \ref{thm0}) that \( \mathcal{J}_{\epsilon} \) is well-defined on \( \mathcal{K} \). Moreover, \( \mathcal{J}_{\epsilon} \) is coercive on \( \mathcal{K} \). Specifically, let \( w \in \mathcal{K} \) and fix \( r > 1 \). Then, there exists a constant \( C(\epsilon) > 0 \) such that
\begin{equation} \label{alpha1}
\ln(w + \epsilon) \leq C (\epsilon) (w + \epsilon)^{\min\left\lbrace \frac{p^{*}}{r'}, \, q-1\right\rbrace }.
		\end{equation}
Taking into account this fact, in conjunction with H\"{o}lder's inequality and Sobolev embeddings (Theorem \ref{thm0}), enables us to derive the following based on the cases of \( \alpha \): \\[4pt]
$ \bullet $ When $0 <  \alpha < 1,$ we have
\begin{equation*} 
\begin{gathered}
\begin{aligned}
\mathcal{J}_{\epsilon}(w)  &\geq \dfrac{1}{p} \left\| w\right\| ^{p}_{W^{1, p}_{0}(\Omega)} - C\left(\left\| f\right\| _{L^{\left(\frac{p^{*}}{1 - \alpha}\right)'}(\Omega)} \left\| w \right\| ^{1 - \alpha}_{L^{p^{*}}(\Omega)}+ \left\| g\right\| _{L^{\left(\frac{p^{*}}{1+ \beta }\right)'}(\Omega)}\left\| w\right\| ^{1 + \beta}_{L^{p^{*}}(\Omega)} + 1 \right)\\[4pt]
& \geq \left\| w\right\| ^{p}_{W^{1, p}_{0}(\Omega)}  \left( \dfrac{1}{p} - C\left(\left\| f\right\| _{L^{\left(\frac{p^{*}}{1 - \alpha}\right)'}(\Omega)} \left\| w\right\| ^{1- \alpha - p}_{W^{1, p}_{0}(\Omega)}  + \left\| g\right\| _{L^{\left(\frac{p^{*}}{1+ \beta }\right)'}(\Omega)}\left\| w\right\| ^{1+ \beta - p}_{W^{1, p}_{0}(\Omega)}  +  \left\| w\right\| ^{-p}_{W^{1, p}_{0}(\Omega)}  \right)\right)  .
\end{aligned}
\end{gathered}
\end{equation*}
$\bullet$ For \( \alpha = 1 \), it follows from \eqref{alpha1} that
\begin{equation*} 
	\begin{gathered}
	\begin{aligned}
	\mathcal{J}_{\epsilon}(w)  &\geq \dfrac{1}{p} \left\| w\right\| ^{p}_{W^{1, p}_{0}(\Omega)} - C\left(\left\| f\right\| _{L^{r}(\Omega)} \left\| w \right\| ^{\min\left\lbrace \frac{p^{*}}{r'}, q-1\right\rbrace }_{L^{p^{*}}(\Omega)}+ \left\| g\right\| _{L^{\left(\frac{p^{*}}{1+ \beta }\right)'}(\Omega)}\left\| w\right\| ^{1 + \beta}_{L^{p^{*}}(\Omega)} + 1 \right)\\[4pt]
	& \geq \left\| w\right\| ^{p}_{W^{1, p}_{0}(\Omega)}  \left( \dfrac{1}{p} - C\left(\left\| f\right\| _{L^{r}(\Omega)} \left\| w\right\| ^{\min\left\lbrace \frac{p^{*}}{r'}, q-1\right\rbrace  - p}_{W^{1, p}_{0}(\Omega)}  + \left\| g\right\| _{L^{\left(\frac{p^{*}}{1+ \beta }\right)'}(\Omega)}\left\| w\right\| ^{1+ \beta - p}_{W^{1, p}_{0}(\Omega)}  +  \left\| w\right\| ^{-p}_{W^{1, p}_{0}(\Omega)}  \right)\right) .
	\end{aligned}
	\end{gathered}
	\end{equation*}
$ \bullet $ When $ \alpha > 1, $ we have
\begin{equation*} 
	\begin{gathered}
	\begin{aligned}
	\mathcal{J}_{\epsilon}(w)  &\geq \dfrac{1}{p} \left\| w\right\| ^{p}_{W^{1, p}_{0}(\Omega)} - C\left(\left\| f\right\| _{L^{1}(\Omega)}  + \left\| g\right\| _{L^{\left(\frac{p^{*}}{1+ \beta }\right)'}(\Omega)}\left\| w\right\| ^{1 + \beta}_{L^{p^{*}}(\Omega)} + 1 \right)\\[4pt]
	& \geq \left\| w\right\| ^{p}_{W^{1, p}_{0}(\Omega)}  \left( \dfrac{1}{p} - C\left(\left\| f\right\| _{L^{1}(\Omega)} \left\| w\right\| ^{- p}_{W^{1, p}_{0}(\Omega)}   + \left\| g\right\| _{L^{\left(\frac{p^{*}}{1+ \beta }\right)'}(\Omega)}\left\| w\right\| ^{1+ \beta - p}_{W^{1, p}_{0}(\Omega)}  +  \left\| w\right\| ^{-p}_{W^{1, p}_{0}(\Omega)}  \right)\right)  .
	\end{aligned}
	\end{gathered}
	\end{equation*}
Thus, we conclude that $\mathcal{J}_{\epsilon}(w) \to \infty$ as $\left\| w\right\|_{W^{1, p}_{0}(\Omega)} \to \infty$ in all the cases discussed above. Moreover, the energy functional \( \mathcal{J}_{\epsilon} \) is weakly lower semi-continuous on \( \mathcal{K} \). In particular, consider a sequence \( (w_{n}) \subset \mathcal{K} \) that converges weakly to some \( w \in \mathcal{K} \) as \( n \longrightarrow \infty \). Hence, 
\begin{equation} \label{equ1}
\begin{gathered}
\begin{aligned}
\int_{\Omega} \left| \nabla w \right| ^{p}dx  &\leq \liminf_{n \to \infty} \int_{\Omega} \left| \nabla w_{n} \right| ^{p}dx, \\[4pt]
\iint_{\mathbb{R}^{2N}}\dfrac{\left| w(x) - w(y)\right| ^{q}}{\left| x - y\right| ^{N + sq}} dxdy &\leq \liminf_{n \to \infty} \iint_{\mathbb{R}^{2N}}\dfrac{\left| w_{n}(x) - w_{n}(y)\right| ^{q}}{\left| x - y\right| ^{N + s q}} dx dy.
	\end{aligned}
\end{gathered}
\end{equation}
We also have, based on the different cases of \( \alpha \), that
\begin{equation*} 
\begin{aligned}
&\int_{\Omega} f(x) (w_{n} + \epsilon)^{1 - \alpha} \, dx \leq C(\epsilon) \left\| f \right\|_{L^{\left(\frac{p^{*}}{1 - \alpha}\right)'}(\Omega)} \left( \left\| w_{n} \right\|_{W^{1, p}_{0}(\Omega)}^{1 - \alpha} + 1 \right) \leq C(\epsilon), \quad &\text{if } \alpha < 1, \\[4pt]
&\int_{\Omega} f(x) \ln(w_{n} + \epsilon) \, dx \leq C(\epsilon) \left\| f \right\|_{L^{r}(\Omega)} \left( \left\| w_{n} \right\|_{W^{1, p}_{0}(\Omega)}^{\min\left\{ \frac{p^{*}}{r'}, q-1 \right\}} + 1 \right) \leq C(\epsilon), \quad &\text{if } \alpha = 1, \\[4pt]
&\int_{\Omega} f(x) (w_{n} + \epsilon)^{1 - \alpha} \, dx \leq C(\epsilon) \left\| f \right\|_{L^{1}(\Omega)} \leq C(\epsilon), \quad &\text{if } \alpha > 1.
\end{aligned}
\end{equation*}
Additionally, we find
\begin{equation} \label{equ2}
\begin{aligned}
\int_{\Omega} g(x) (w_{n} + \epsilon)^{\beta+1} \, dx &\leq C(\epsilon) \left\| g \right\|_{L^{\left(\frac{p^{*}}{\beta + 1}\right)'}(\Omega)} \left( \left\| w_{n} \right\|_{W^{1, p}_{0}(\Omega)}^{\beta + 1} + 1 \right) \\
&\leq C(\epsilon),
\end{aligned}
\end{equation}
where \( C(\epsilon) > 0 \) is independent of \( n \). Then, by applying Vitali's Convergence Theorem, we obtain
$$ \int_{\Omega}\int_{0}^{w_{n}} f(x) (s + \epsilon)^{- \alpha} ds dx \to \int_{\Omega}\int_{0}^{w} f(x) (s + \epsilon)^{- \alpha} ds dx \text{ as } n \to +\infty, $$
and
$$ \int_{\Omega}\int_{0}^{w_{n}} g(x) (s + \epsilon)^{\beta} ds dx \to \int_{\Omega}\int_{0}^{w} g(x) (s + \epsilon)^{\beta} ds dx \text{ as } n \to +\infty. $$
Combining all the aforementioned results, we conclude that \( \mathcal{J}_{\epsilon} \) is weakly lower semi-continuous in class \textbf{(F1)}.\\[4pt]
If \textbf{(F2)} holds, we first observe that if \(0 \leq \delta < 1 + \frac{1}{p'}\), then \(d^{-\delta} \in W^{-1, p'}(\Omega)\). More precisely, by applying H\"{o}lder's inequality and Hardy's inequality, for any \( u \in W^{1, p}_{0}(\Omega) \), we obtain
\begin{equation}\label{equ3}
\begin{aligned}
\int_{\Omega} d^{- \delta}(x)  u(x)  \,dx &\leq  \left( \int_{\Omega} d^{p' (1 - \delta)}(x) \, dx \right) ^{\frac{1}{p'}} \left( \int_{\Omega}\dfrac{\left| u(x)\right|^{p} }{d^{p}(x)}  \, dx\right) ^{\frac{1}{p}} \\[4pt]
&\leq C \left( \int_{\Omega} d^{p' (1 - \delta)}(x) \, dx \right) ^{\frac{1}{p'}} \left\| u\right\| _{W^{1, p}_{0}(\Omega)} < \infty,
\end{aligned}
\end{equation}
since 
$$ \delta < 1 + \frac{1}{p'} \Longrightarrow p' (1 - \delta) > -1  \Longrightarrow  \int_{\Omega} d^{p' (1 - \delta)}(x) \, dx  < \infty.$$ 
	\noindent From this fact, we also observe that \( \mathcal{J}_{\epsilon} \) is well-defined on \( \mathcal{K} \). Furthermore, by using \eqref{equ3}, we can infer that
\begin{equation*} 
\begin{gathered}
\begin{aligned}
 \mathcal{J}_{\epsilon}(w) &\geq \left\| w\right\|^{p}_{W^{1, p}_{0}(\Omega)}\left( \dfrac{1}{p} - C(\epsilon) \left( \left\| w\right\|^{1 - p}_{W^{1, p}_{0}(\Omega)} +  \left\| g\right\| _{L^{\left(\frac{p^{*}}{1+ \beta }\right)'}(\Omega)}\left\| w\right\| ^{1+ \beta - p}_{W^{1, p}_{0}(\Omega)} + \left\| w\right\| ^{-p}_{W^{1, p}_{0}(\Omega)} \right) \right) \\[4pt]
 & \quad \to + \infty \text{ as } \left\| w\right\|_{W^{1, p}_{0}(\Omega)} \to + \infty. 
 \end{aligned}
 \end{gathered}
 \end{equation*}
	This implies that \( \mathcal{J}_{\epsilon} \) is coercive on \( \mathcal{K} \). On the other hand, we have that \( \int_{0}^{w_{n}} f(x) (s + \epsilon)^{-\alpha} \, ds \) is uniformly integrable in \( L^{1}(\Omega) \). Indeed, using \eqref{equ3}, we have that
	\begin{equation*}
	\int_{\Omega} \int_{0}^{w_{n}} f(x) (s + \epsilon)^{-\alpha} \, ds \, dx \leq C(\epsilon) \left\| w_{n} \right\|_{W^{1, p}_{0}(\Omega)} < C(\epsilon),
	\end{equation*}
	where \( C(\epsilon) > 0 \) is independent of \( n \). Then, by using Vitali Convergence Theorem, we obtain
	\begin{equation} \label{equ4}
	\int_{\Omega} \int_{0}^{w_{n}} f(x) (s + \epsilon)^{-\alpha} \, ds \, dx \to \int_{\Omega} \int_{0}^{w} f(x) (s + \epsilon)^{-\alpha} \, ds \, dx, \quad \text{as } n \to +\infty.
	\end{equation}
By combining \eqref{equ1}--\eqref{equ4}, we deduce that \( \mathcal{J}_{\epsilon} \) is weakly lower semi-continuous on \( \mathcal{K} \). Hence, based on the aforementioned properties, \( \mathcal{J}_{\epsilon} \) possesses a minimizer \( w_{0} \) within the convex and closed set \( \mathcal{K} \). Now, let us consider \( \psi \in w_{0} + (W^{1, p}_{0}(\Omega) \cap L_{c}^{\infty}(\Omega)) \) with \( \psi \in \mathcal{K} \). Setting
	\begin{center}
		$ \xi : \left[0, 1\right] \to \mathbb{R} $\quad by\quad $ \xi(s) = \mathcal{J}_{\epsilon}(s \psi + (1- s) w_{0}). $
	\end{center}
	Hence, we have
	\begin{align*}
		& 0 \leq \xi(s) - \xi(0) = \mathcal{J}_{\epsilon}(w_{0} + s (\psi - w_{0})) -\mathcal{J}_{\epsilon}(w_{0})\\[4pt]
		& = \dfrac{1}{p} \int_{\Omega} \left(\left| \nabla(w_{0} + s (\psi  - w_{0}))\right| ^{p} - \left| \nabla w_{0}\right| ^{p}\right) dx\\[4pt]
		&\quad + \dfrac{1}{q}\iint_{\mathbb{R}^{2 N}} \dfrac{\left|(w_{0} + s (\psi - w_{0}))(x) - (w_{0} + s (\psi - w_{0}))(y)\right| ^{q} - \left|w_{0}(x) - w_{0}(y)\right| ^{q}}{\left| x-y\right| ^{N +sq}}  dx dy \\[4pt]
		& 
		\quad- \int_{\Omega} \left( G_{\epsilon}(x, w_{0} + s (\psi - w_{0})) - G_{\epsilon}(x, w_{0})\right)dx. 
	\end{align*}
Then, dividing by \( s \) and passing to the limit as \( s \to 0^{+} \), we obtain that
\begin{align*}
		0 \leq & \int_{\Omega} \left[ \nabla w_{0}\right] ^{p-1}  \nabla (\psi - w_{0})\, dx  +\displaystyle\iint_{\mathbb{R}^{2 N}}  \dfrac{\left[w_{0}(x) - w_{0}(y)\right] ^{q- 1}\left((\psi - w_{0}) (x)- (\psi - w_{0})(y)\right)}{\vert x-y\vert ^{N+ s q}}dx dy \\[4pt]
		& - \int_{\Omega} \left( f(x)(w_{0} + \epsilon)^{- \alpha} + g(x)\,(w_{0} + \epsilon)^{\beta}\right) \,(\psi - w_{0})\,dx.
	\end{align*}
	\begin{claim}
		For every $ \psi \in W^{1, p}_{0}(\Omega)\cap L^{\infty}(\Omega),$ with $  \psi \geq 0 ,$ a.e. in \( \Omega \), it holds that:
		\begin{equation}\label{comequ10}
			\begin{gathered}
				\begin{aligned}
					&\int_{\Omega} \left[ \nabla w_{0}\right] ^{p-1}  \nabla \psi\, dx + \displaystyle\iint_{\mathbb{R}^{2 N}} \dfrac{\left[w_{0}(x) - w_{0}(y)\right] ^{q-1}\left(\psi(x)- \psi(y)\right)}{\vert x-y\vert ^{N+ s q}}dx dy \\[4pt]
					&\quad \geq \int_{\Omega} \left( f(x) (w_{0} + \epsilon)^{- \alpha} + g(x)\,(w_{0} + \epsilon)^{\beta}\right) \,\psi \,dx.
				\end{aligned}
			\end{gathered}
		\end{equation}
	\end{claim} 
	\noindent In fact, initially, for any non-negative \( \psi \in C^{\infty}_{c}(\Omega) \), and for \( t \in (0, 1) \), we define	
	\begin{center}
		$ \psi_{t} := \min \left\lbrace w_{0} + t \psi, \overline{u}\right\rbrace$ \, and \, $ w_{t} =  (w_{0} + t \psi - \overline{u})^{+}. $
	\end{center}
First, considering the expression for \( w_{t} \), we observe that \( w_{t} \leq t \psi \). Moreover, \( w_{t} \in W^{1, p}_{0}(\Omega) \) and \( w_{t} \in \mathcal{K} \) for sufficiently small \( t \). In fact, it is evident that \( w_{t} = 0 \) in \( \mathbb{R}^{N} \setminus \text{supp}(\psi) \), since \( w_{0} \in \mathcal{K} \). Furthermore, we have
\[
\int_{\Omega} \left| \nabla w_{t} \right| ^{p} \, dx = \int_{\text{supp} (\psi)} \left| \nabla \left(w_{0} + t \psi - \overline{u}\right) \right| ^{p} \, dx < \infty, \quad \text{since } \overline{u} \in W^{1, p}_{\text{loc}}(\Omega).
\]
Additionally, we note that
\begin{equation}\label{equ32}
\psi_{t} - w_{0} = t \psi - w_{t}.
\end{equation}
Hence, we conclude that \( \psi_{t} \in w_{0} + \left( W^{1, p}_{0}(\Omega) \cap L_{c}^{\infty}(\Omega) \right) \), with \( \psi_{t} \in \mathcal{K} \). Therefore, from \eqref{comequ8}, we deduce that:
\begin{equation*}
\begin{gathered}
\begin{aligned}
\int_{\Omega} &\left[ \nabla w_{0}\right] ^{p-1}  \nabla (\psi_{t} - w_{0})\, dx  +  \displaystyle\iint_{\mathbb{R}^{2 N}} \dfrac{\left[w_{0}(x) - w_{0}(y) \right] ^{q-1}\left((\psi_{t} - w_{0})(x)- (\psi_{t} - w_{0})(y)\right)}{\vert x-y\vert ^{N+ s q}}dx dy \\[4pt]
& \geq {\displaystyle\int_{\Omega}}\left(f(x)\,(w_{0} + \epsilon)^{-\alpha} + g(x) (w_{0} + \epsilon)^{\beta} \right) \,(\psi_{t} - w_{0})\,dx,
\end{aligned}
\end{gathered}
\end{equation*}
which along with \eqref{equ32} implies
	\begin{equation*}
		\begin{gathered}
			\begin{aligned}
				\int_{\Omega} &\left[ \nabla w_{0}\right] ^{p-1}  \nabla \psi\, dx  +  \displaystyle\iint_{\mathbb{R}^{2 N}} \dfrac{\left[w_{0}(x) - w_{0}(y) \right] ^{q-1}\left(\psi(x)- \psi(y)\right)}{\vert x-y\vert ^{N+ s q}}dx dy \\[4pt]
				& - {\displaystyle\int_{\Omega}}\left(f(x)\,(w_{0} + \epsilon)^{-\alpha} + g(x) (w_{0} + \epsilon)^{\beta} \right) \,\psi\,dx\\[4pt]
		       &\geq \dfrac{1}{t}\left( \int_{\Omega} \left[ \nabla w_{0}\right] ^{p-1}  \nabla w_{t}\, dx +  \displaystyle\iint_{\mathbb{R}^{2 N}} \dfrac{\left[w_{0}(x) - w_{0}(y) \right] ^{q-1}\left(w_{t}(x)- w_{t}(y)\right)}{\vert x-y\vert ^{N+ s q}}dx dy \right. \\[4pt]
		       &\left.  - {\displaystyle\int_{\Omega}}\left(f(x)\,(w_{0} + \epsilon)^{-\alpha} + g(x) (w_{0} + \epsilon)^{\beta} \right) \,w_{t}\,dx \right) .
			\end{aligned}
		\end{gathered}
	\end{equation*}
Since $ \overline{u} $ is a weak super-solution to problem \eqref{P2}, we can now write
	\begin{equation}\label{equ5} 
	\begin{gathered}
	\begin{aligned}
	\int_{\Omega} &\left[ \nabla w_{0}\right] ^{p-1}  \nabla \psi\, dx  +  \displaystyle\iint_{\mathbb{R}^{2 N}} \dfrac{\left[w_{0}(x) - w_{0}(y) \right] ^{q-1}\left(\psi(x)- \psi(y)\right)}{\vert x-y\vert ^{N+ s q}}dx dy \\[4pt]
	& - {\displaystyle\int_{\Omega}}\left(f(x)\,(w_{0} + \epsilon)^{-\alpha} + g(x) (w_{0} + \epsilon)^{\beta} \right) \,\psi\,dx\\[4pt]
	&\geq  \underbrace{ \dfrac{1}{t}\left( \int_{\Omega} \left( \left[ \nabla w_{0}\right] ^{p-1} -\left[ \nabla \overline{u}\right] ^{p-1} \right)  \nabla w_{t}\, dx\right) }_{\textbf{I}_{1}}  \\[4pt]
	& +  \underbrace{ \dfrac{1}{t}\displaystyle\iint_{\mathbb{R}^{2 N}} \dfrac{\left( \left[w_{0}(x) - w_{0}(y) \right] ^{q-1} - \left[ \overline{u}(x) - \overline{u}(y) \right] ^{q-1} \right) \left(w_{t}(x)- w_{t}(y)\right)}{\vert x-y\vert ^{N+ s q}}dx dy }_{\textbf{I}_{2}} \\[4pt]
	& + \underbrace{ \dfrac{1}{t} \left( {\displaystyle\int_{\Omega}}\left( f(x)\,\left(\overline{u}^{-\alpha} - (w_{0} + \epsilon)^{-\alpha} \right) + g(x) \left(\overline{u}^{\beta} - (w_{0} + \epsilon)^{\beta} \right) \right) \,w_{t}\,dx\right) }_{\textbf{I}_{3}} .
	\end{aligned}
	\end{gathered}
	\end{equation}
	\textbf{Estimate of $ \textbf{I}_{1} .$}  By applying Lemma \ref{lemma1} and \eqref{equ32}, we obtain
\begin{equation}\label{equ21}
\begin{gathered}
\begin{aligned}
 \textbf{I}_{1}  =  \dfrac{1}{t}\left[\int_{\text{supp}(w_{t})} \left( \left[ \nabla w_{0}\right] ^{p-1} -\left[ \nabla \psi_{t} \right] ^{p-1} \right)  \nabla w_{t}\, dx\right] &\geq  \underbrace{ \int_{\text{supp}(w_{t})} \left( \left[ \nabla w_{0}\right] ^{p-1} -\left[ \nabla \overline{u}\right] ^{p-1} \right)  \nabla \psi\, dx }_{\longrightarrow  0 \text{ as } t \to 0}.
 	\end{aligned}
 \end{gathered}
 \end{equation}
\textbf{Estimate of $ \textbf{I}_{2} .$}  First, we set
	\begin{center}
		$ \omega_{1} = \left\lbrace x \in \Omega,\, \overline{u}(x) \leq  (w_{0} + t \psi)(x) \right\rbrace $ and $ \omega_{2} = \left\lbrace x \in \Omega,\, \overline{u}(x) \geq (w_{0} + t \psi)(x) \right\rbrace .$ 
	\end{center} 
	Then by using the monotonicity of the map 
	\begin{center}
		$ \tau \mapsto \left[ \tau - \tau_{0}\right] ^{q - 1} , $
	\end{center} 
	we obtain that	
\begin{align*}
		&\displaystyle\iint_{\mathbb{R}^{2 N}} \dfrac{\left[ \psi_{t}(x) - \psi_{t}(y)\right] ^{q-1}\left( w_{t}(x) - w_{t}(y)\right) }{\vert x-y\vert ^{N+ sq}} dx dy =  \displaystyle\iint_{\omega_{1}\times \omega_{1}}  \dfrac{\left[ \psi_{t}(x) - \psi_{t}(y)\right] ^{q-1}\left( w_{t}(x) - w_{t}(y)\right) }{\vert x-y\vert ^{N+ s q}} dx dy \\[4pt]
		&+ \displaystyle\iint_{\omega_{1} \times \omega_{2}}  \dfrac{\left[ \psi_{t}(x) - \psi_{t}(y)\right] ^{q-1}\left( w_{t}(x) - w_{t}(y)\right) }{\vert x-y\vert ^{N+ s q}} dx dy +  \displaystyle\iint_{\omega_{2} \times \omega_{1}}  \dfrac{\left[ \psi_{t}(x) - \psi_{t}(y)\right] ^{q-1}\left( w_{t}(x) - w_{t}(y)\right) }{\vert x-y\vert ^{N+ s q}} dx dy\\[4pt]
		&\geq \displaystyle\iint_{\omega_{1} \times \omega_{1}}  \dfrac{\left[ \overline{u}(x) - \overline{u}(y)\right] ^{q-1}\left( w_{t}(x) - w_{t}(y)\right) }{\vert x-y\vert ^{N+ s q}} dx dy +  \displaystyle\iint_{\omega_{1} \times \omega_{2}}   \dfrac{\left[ \overline{u}(x) - \overline{u}(y)\right] ^{q-1}\left( w_{t}(x) - w_{t}(y)\right) }{\vert x-y\vert ^{N+ s q}} dx dy \\[4pt]
		&+ \displaystyle\iint_{\omega_{2} \times \omega_{1}}  \dfrac{\left[ \overline{u}(x) - \overline{u}(y)\right] ^{q-1}\left( w_{t}(x) - w_{t}(y)\right) }{\vert x-y\vert ^{N+ s q}} dx dy = \displaystyle\iint_{\mathbb{R}^{2 N}}   \dfrac{\left[ \overline{u}(x) -\overline{u}(y)\right] ^{q-1}\left( w_{t}(x) - w_{t}(y)\right) }{\vert x-y\vert ^{N+ s q}} dx dy.
	\end{align*}
Once more, by invoking Lemma \ref{lemma1}, we get 
\begin{align*}
		\vspace{-3cm}&\displaystyle\iint_{\mathbb{R}^{2 N}\setminus (\text{supp}(w_{t}))^{c} \times(\text{supp}(w_{t}))^{c}}\dfrac{\left(\left[ w_{0}(x) - w_{0}(y)\right] ^{q-1} - \left[ \psi_{t}(x) - \psi_{t}(y)\right] ^{q-1}\right) ( w_{0}(x) - w_{0}(y))}{\vert x-y\vert ^{N+ s q}}dx dy \\[4pt]
		&\quad\geq \displaystyle\iint_{\mathbb{R}^{2 N}\setminus (\text{supp}(w_{t}))^{c} \times(\text{supp}(w_{t}))^{c}} \dfrac{\left(\left[ w_{0}(x) - w_{0}(y)\right] ^{q-1} - \left[ \psi_{t}(x) - \psi_{t}(y)\right] ^{q-1}\right)  (\psi_{t}(x) -  \psi_{t}(y))}{\vert x-y\vert ^{N+ s q}}dx dy.
	\end{align*}
By combining these observations with \eqref{equ32}, we arrive at
\begin{equation}\label{equ33}
		\begin{gathered}
			\begin{aligned}
				\textbf{I}_{2} &\geq \displaystyle\iint_{\mathbb{R}^{2 N}\setminus (\text{supp}(w_{t}))^{c} \times(\text{supp}(w_{t}))^{c}} \dfrac{\left(\left[ w_{0}(x) - w_{0}(y)\right] ^{q-1} -  \left[\psi_{t}(x) - \psi_{t}(y)\right] ^{q-1}\right)\left(\psi(x)- \psi(y)\right)}{\vert x-y\vert ^{N+ s q}}dx dy\\[4pt]
				&=\displaystyle\iint_{\text{supp}(w_{t}) \times \text{supp}(w_{t}) } \dfrac{\left(\left[ w_{0}(x) - w_{0}(y)\right] ^{q-1} -  \left[ \overline{u}(x) - \overline{u}(y)\right] ^{q-1}\right)\left(\psi(x)- \psi(y)\right)}{\vert x-y\vert ^{N+ s q}}dx dy\\[4pt]
				&\quad + 2\displaystyle\iint_{\text{supp}(w_{t}) \times (\text{supp}(w_{t}))^{c} } \dfrac{\left( \left[ w_{0}(x) - w_{0}(y)\right] ^{q-1} -  \left[\psi_{t}(x) - \psi_{t}(y)\right] ^{q-1} \right)\left(\psi(x)- \psi(y)\right)}{\vert x-y\vert ^{N+ s q}}dx dy\\[4pt]
				&\qquad \longrightarrow  0  \text{ as } t \to 0^{+}.
			\end{aligned}
		\end{gathered}
	\end{equation}
	Indeed, by noting that $ \text{meas} (\text{supp} (w_{t})) \to 0 $  as $ t \to 0^{+}, $  it is not hard to show that  
\begin{equation*}
\begin{gathered}
\begin{aligned}
		\displaystyle\iint_{\text{supp}(w_{t}) \times \text{supp}(w_{t}) }& \dfrac{\left(\left[w_{0}(x) -w_{0}(y)\right] ^{q-1} -  \left[ \overline{u}(x) - \overline{u}(y)\right] ^{q-1}\right)\left(\psi(x)- \psi(y)\right)}{\vert x-y\vert ^{N+ s q}}dx dy\\[4pt]
		&\quad 
		 \longrightarrow 0\, \text{ as }\, t \to 0^{+}.
			\end{aligned}
\end{gathered}
\end{equation*}
	On the other hand, utilizing  Lemma \ref{lemma1}, together with H\"{o}lder's inequality, we obtain
\begin{equation}\label{equ36}
		\begin{gathered}
			\begin{aligned}
				&\left| \displaystyle\iint_{\text{supp}(w_{t}) \times (\text{supp}(w_{t}))^{c}} \dfrac{\left(\left[ w_{0}(x) - w_{0}(y)\right] ^{q-1} -  \left[ \psi_{t}(x) - \psi_{t}(y)\right] ^{q-1}\right)\left(\psi(x)- \psi(y)\right)}{\vert x-y\vert ^{N+ s q}}dx dy \right| \\[4pt]
				&\leq C \displaystyle\iint_{\text{supp}(w_{t}) \times (\text{supp}(w_{t}))^{c}} \dfrac{\left( \left|w_{0}(x) - w_{0}(y)\right| + \left|\psi_{t}(x) - \psi_{t}(y)\right|\right)^{q-1}  \left| \psi(x)- \psi(y)\right| }{\vert x-y\vert ^{N+ s q}}dx dy \\[4pt]
				&\leq C \left( \displaystyle\iint_{\text{supp}(w_{t}) \times (\text{supp}(w_{t}))^{c}} \dfrac{\left( \left|w_{0}(x) - w_{0}(y)\right| + \left|\psi_{t}(x) - \psi_{t}(y)\right|\right)^{q}}{\vert x-y\vert ^{N+ s q}}dx dy\right) ^{\frac{q-1}{q}}\\[4pt]
				&\qquad  \times\left(\displaystyle\iint_{\text{supp}(w_{t}) \times (\text{supp}(w_{t}))^{c}}\dfrac{\left| \psi(x)- \psi(y)\right|^{q}}{\vert x-y\vert ^{N+ s q}} dx dy\right) ^{\frac{1}{q}}.
			\end{aligned}
		\end{gathered}
	\end{equation}
Given that \( \text{supp}(w_{t}) \Subset \text{supp}(\psi) \) and according to Definition \ref{definition1}, there exists a constant \( C_{\text{supp}(\psi)} > 0 \), independent of \( t \), such that \( \overline{u} > C_{\text{supp}(\psi)} \) within \( \text{supp}(w_{t}) \). Hence, from \cite[Lemma 3.5]{ref01}, we obtain
	\begin{equation}\label{equ37}
		\begin{gathered}
			\begin{aligned}
				&\displaystyle\iint_{\text{supp}(w_{t}) \times (\text{supp}(w_{t}))^{c}} \dfrac{\left( \left|w_{0}(x) - w_{0}(y)\right| + \left|\psi_{t}(x) - \psi_{t}(y)\right|\right)^{q}}{\vert x-y\vert ^{N+ s q}}dx dy\\[4pt]
				& \leq C \left( \displaystyle\iint_{\mathbb{R}^{2N}}\dfrac{\left| w_{0}(x)- w_{0}(y)\right|^{q}}{\vert x-y\vert ^{N+ s q}} dx dy  + \displaystyle\iint_{\mathbb{R}^{2N}}\dfrac{\left| \psi(x)- \psi(y)\right|^{q}}{\vert x-y\vert ^{N+ s q}} dx dy \right. \\[4pt]
				&\left.  +  \displaystyle\iint_{\mathbb{R}^{2N}}\dfrac{\left| \overline{u}^{\theta} (x)- \overline{u} ^{\theta}(y)\right|^{q}}{\vert x-y\vert ^{N+ s q}} dx dy \right) < \infty,
			\end{aligned}
		\end{gathered}
	\end{equation}
	where $ C $ depends on $ q, \theta, C_{\text{supp} (\psi)}. $ By applying the result from \eqref{equ37} to \eqref{equ36}, we obtain	
	
\begin{equation*}
\begin{gathered}
\begin{aligned}
		\displaystyle\iint_{\text{supp}(w_{t}) \times (\text{supp}(w_{t}))^{c}}& \dfrac{\left(\left[w_{0}(x) - w_{0}(y)\right] ^{q-1} -  \left[\psi_{t}(x) - \psi_{t}(y)\right] ^{q-1}\right)\left(\psi(x)- \psi(y)\right)}{\vert x-y\vert ^{N+ s q}}dx dy\\[4pt]
		& \longrightarrow 0\, \text{ as }\, t \to 0^{+}.
		\end{aligned}
\end{gathered}
\end{equation*}
	\textbf{Estimate of $ \textbf{I}_{3}. $} Given that $ w_{t} \leq t \psi $, we can observe that
	\begin{equation}\label{equ30}
		\begin{gathered}
			\begin{aligned}
				\textbf{I}_{3} &\geq -  {\displaystyle\int_{\text{supp}(w_{t})}}\left( f(x)\,\left| \overline{u}^{-\alpha} - (w_{0} + \epsilon)^{-\alpha} \right| + g(x) \left| \overline{u}^{\beta} - (w_{0} + \epsilon)^{\beta} \right| \right) \,\psi\,dx\\[4pt]
				&\quad \longrightarrow 0\, \text{ as }\, t \to 0^{+}.
			\end{aligned}
		\end{gathered}
	\end{equation}  
By gathering \eqref{equ21}, \eqref{equ33}, and \eqref{equ30} and taking the limit in \eqref{equ5} as $ t \to 0^{+} $, we obtain
\begin{equation}\label{equ6} 
\begin{gathered}
\begin{aligned}
&\int_{\Omega} \left[ \nabla w_{0}\right] ^{p-1}  \nabla \psi\, dx + \displaystyle\iint_{\mathbb{R}^{2 N}} \dfrac{\left[w_{0}(x) - w_{0}(y)\right] ^{q-1}\left(\psi(x)- \psi(y)\right)}{\vert x-y\vert ^{N+ s q}}dx dy \\[4pt]
&\quad \geq \int_{\Omega} \left( f(x) (w_{0} + \epsilon)^{- \alpha} + g(x)\,(w_{0} + \epsilon)^{\beta}\right) \,\psi \,dx,
\end{aligned}
\end{gathered}
\end{equation}
for all \( \psi \in C^{\infty}_{c}(\Omega) \) with \( \psi \geq 0 \) almost everywhere in \( \Omega \), we now conclude, using classical density arguments, that \eqref{equ6} holds for any \( \psi \in W^{1, p}_{0}(\Omega) \cap L^{\infty}(\Omega) \), with \( \psi \geq 0 \) almost everywhere in \( \Omega \). Indeed, consider the sequence \( (\psi_{\tau, n}) \subset C^{\infty}_{c}(\Omega) \) such that
\[
\psi_{\tau, n} \xrightarrow{\tau \to 0} \psi_{n} \xrightarrow{n \to \infty} \psi \text{ in } W^{1, p}_{0}(\Omega),
\]
where \( \text{supp}(\psi_{n}) \Subset \Omega \) and \( 0 \leq \psi_{n} \leq \psi \) for all \( n \in \mathbb{N} \). By first passing \( \tau \to 0 \) and then \( n \to \infty \), we obtain the required result.
	\begin{claim}
		$ \underline{u} \leq w_{0}  +  \epsilon, $ for each $ \epsilon > 0.$
	\end{claim}
	\noindent To this end, we adapt the proof methods from \cite{durastanti2022comparison}. Let \( \epsilon > 0 \), \( m > 0 \), and \( k > 0 \) be fixed. We start by observing that
	\[
	\Psi_{m}  = \frac{\textbf{T}_{k}\left( \left(\left(\underline{u} + m\right) ^{q} - \left( w_{0} + m + \epsilon\right) ^{q}\right)^{+} \right) }{(\underline{u} + m)^{q - 1}}, \quad \Phi_{m}  = \frac{\textbf{T}_{k}\left( \left( \left( w_{0} + m + \epsilon\right) ^{q} - \left(\underline{u} + m\right) ^{q}\right)^{-} \right) }{\left( w_{0} + m + \epsilon\right) ^{q-1}},
	\]
	are well-defined and belong to \( W^{1, p}_{0}(\Omega) \cap L^{\infty}(\Omega) \). We provide the explicit proof only for \( \Psi_{m} \), as the treatment of \( \Phi_{m} \) follows similarly. Given that \( w_{0} \geq 0 \) almost everywhere in \( \Omega \), we have
	\[
	\text{supp}(\Psi_{m}) \subset \mathcal{S}_{\epsilon},
	\]
where $ \mathcal{S}_{\epsilon}  : =\text{supp}((\underline{u} - \epsilon)^{+}).$ We now define the following subdomains:
\begin{equation*}
\begin{gathered}
\begin{aligned}
& \Omega_{1} = \left\lbrace x \in \Omega, \quad \left(\underline{u} + m\right) ^{q}(x) - \left( w_{0} + m + \epsilon\right) ^{q}(x) \leq 0\right\rbrace, \\[4pt]
&\Omega_{2} = \left\lbrace x \in \Omega,\quad  0 < \left(\underline{u} + m\right) ^{q}(x) - \left( w_{0} + m + \epsilon\right) ^{q}(x)  < k\right\rbrace, \\[4pt]
&\Omega_{3} = \left\lbrace x \in \Omega, \quad \left(\underline{u} + m\right) ^{q}(x) - \left( w_{0} + m + \epsilon\right) ^{q}(x)  \geq k\right\rbrace.
\end{aligned}
\end{gathered}
\end{equation*}
A direct computation reveals that
\begin{equation*}
\begin{gathered}
\begin{aligned}
\left| \nabla \Psi_{m} \right| &= \left| \left( \nabla \underline{u} - q \left(\dfrac{w_{0} + m + \epsilon}{ \underline{u}  + m}\right)^{q-1} \nabla w_{0} + (q-1) \left( \dfrac{w_{0} + m + \epsilon}{\underline{u} + m}\right) ^{q} \nabla \underline{u}\right) \chi_{\Omega_{2}} - \dfrac{k(q-1)}{(\underline{u} + m)^{q}}\nabla \underline{u} \, \chi_{\Omega_{3}}\right| \\[4pt]
& \leq C(q) \left( \left|  \nabla \underline{u} \right|  + \left|\nabla w_{0} \right| \right).
\end{aligned}
\end{gathered}
\end{equation*}
On the other hand, since \( \underline{u} \) is a weak sub-solution to problem \eqref{P2}, it follows that
\begin{equation*}
\begin{gathered}
\begin{aligned}
\int_{\Omega}\left| \nabla \Psi_{m}  \right| ^{p} dx &\leq C(p, q) \left[\int_{\mathcal{S}_{\epsilon}}\left| \nabla \underline{u}\right| ^{p} dx + \int_{\mathcal{S}_{\epsilon}}\left| \nabla w_{0} \right| ^{p} dx\right] \\[4pt]
&\leq C(p, q, \theta, \epsilon) \left[\int_{\Omega}\left| \nabla (\underline{u}^{\theta})\right| ^{p} dx + \int_{\Omega}\left| \nabla w_{0} \right| ^{p} dx\right] < \infty.
\end{aligned}
\end{gathered}
\end{equation*}
Now, by testing \eqref{comequ10} with \( \Phi_{m} \), we derive that
		\begin{equation}\label{equ7}
	\begin{gathered}
	\begin{aligned}
	&\int_{\Omega} \left[ \nabla w_{0}\right] ^{p-1}  \nabla \Phi_{m} \, dx + \displaystyle\iint_{\mathbb{R}^{2 N}} \dfrac{\left[w_{0}(x) - w_{0}(y)\right] ^{q-1}\left(\Phi_{m} (x)- \Phi_{m} (y)\right)}{\vert x-y\vert ^{N+ s q}}dx dy \\[4pt]
	&\quad \geq \int_{\Omega} \left( f(x) (w_{0} + \epsilon)^{- \alpha} + g(x)\,(w_{0} + \epsilon)^{\beta}\right) \,\Phi_{m} \,dx.
	\end{aligned}
	\end{gathered}
	\end{equation}
Next, we consider a sequence of functions \( (\varphi_{n}) \in C^{\infty}_{c}(\Omega) \) such that \( \varphi_{n} \to \Psi_{m} \) in \( W^{1, p}_{0}(\Omega) \), and define \( \tilde{\varphi}_{n} := \min\left\lbrace \Psi_{m}, \varphi_{n}^{+} \right\rbrace \). Then, noting that \( \tilde{\varphi}_{n} \in W^{1, p}_{0}(\Omega) \cap L_{c}^{\infty}(\Omega) \), we obtain
\begin{equation*}
\begin{gathered}
\begin{aligned}
\int_{\Omega} \left[ \nabla \underline{u}\right] ^{p-1} \nabla \tilde{\varphi}_{n}  dx &+  \displaystyle \displaystyle\iint_{\mathbb{R}^{2N}} \dfrac{\left[ \underline{u}(x) - \underline{u}(y) \right] ^{q-1} \left(\tilde{\varphi}_{n} (x)- \tilde{\varphi}_{n} (y)\right)}{\vert x-y\vert ^{N+ s q}}dxdy \\[4pt]
&\leq {\displaystyle\int_{\Omega}} \left( f(x) \underline{u}^{-\alpha} + g(x) \underline{u}^{\beta}\right) \tilde{\varphi}_{n}  dx.
\end{aligned}
\end{gathered}
\end{equation*}
From Definition \ref{definition1}, we can pass to the limit as \( n \to \infty \) to obtain
\begin{equation}\label{equ8}
\begin{gathered}
\begin{aligned}
\int_{\Omega} \left[ \nabla \underline{u}\right] ^{p-1} \nabla \Psi_{m} dx &+  \displaystyle \displaystyle\iint_{\mathbb{R}^{2N}} \dfrac{\left[ \underline{u}(x) - \underline{u}(y) \right] ^{q-1} \left(\Psi_{m} (x)-\Psi_{m} (y)\right)}{\vert x-y\vert ^{N+ s q}}dxdy \\[4pt]
&\leq {\displaystyle\int_{\Omega}} \left( f(x) \underline{u}^{-\alpha} + g(x) \underline{u}^{\beta}\right) \Psi_{m} dx.
\end{aligned}
\end{gathered}
\end{equation}
By subtracting \eqref{equ8} from \eqref{equ7}, we obtain
\begin{equation}\label{comequ1}
		\begin{gathered}
			\begin{aligned}
				&\underbrace{ \int_{\Omega} \left( \left[ \nabla \underline{u}\right] ^{p-1}\nabla\Psi_{m} -  \left[ \nabla w_{0}\right] ^{p-1}\nabla\Phi_{m}  \right) dx}_{\textbf{E}_{1}} +  \underbrace{ \displaystyle \displaystyle\iint_{\mathbb{R}^{2N}}  \dfrac{\textbf{F}_{\underline{u}, w_{0}, \Psi_{m}, \Phi_{m}}(x, y) }{\vert x-y\vert ^{N+ s q}} dxdy}_{\textbf{E}_{2}}  \\[4pt]
				&\leq \underbrace{  {\displaystyle\int_{\Omega}} \textbf{G}_{\underline{u}, w_{0}, f, g}(x) \, \textbf{T}_{k} \left( \left(\left(\underline{u} + m\right) ^{q} - \left( w_{0} + m + \epsilon\right) ^{q}\right)^{+} \right)  dx}_{\textbf{E}_{3}},
			\end{aligned}
		\end{gathered}
	\end{equation}
where 
$$\textbf{F}_{\underline{u}, w_{0}, \Psi_{m}, \Phi_{m}}(x, y) := \left[ \underline{u}(x)  - \underline{u}(y) \right] ^{q-1} \left(\Psi_{m} (x)-\Psi_{m} (y)\right) - \left[ w_{0}(x) - w_{0}(y) \right] ^{q-1} \left(\Phi_{m}  (x)-\Phi_{m} (y)\right),$$
and
$$ \textbf{G}_{\underline{u}, w_{0}, f, g}(x) := f(x)\left(\dfrac{\underline{u}^{-\alpha}}{\left(\underline{u} + m\right) ^{q-1}} - \dfrac{(w_{0} + \epsilon)^{-\alpha}}{\left(w_{0} + m + \epsilon\right) ^{q-1}} \right) + g(x)\left(\dfrac{\underline{u}^{\beta}}{\left(\underline{u} + m\right) ^{q-1}} - \dfrac{(w_{0} + \epsilon)^{\beta}}{\left(w_{0} + m + \epsilon\right) ^{q-1}} \right). $$

\noindent  \textbf{Estimate of $ \textbf{E}_{1} .$}  By applying Lemma \ref{Lemma3} and since $ \text{meas} (\Omega_{3}) \to 0 $  as $ k \to \infty, $ we obtain
 \begin{equation}\label{equ9}
\begin{gathered}
\begin{aligned}
\textbf{E}_{1} &\geq k(q-1)  \int_{\Omega_{3}} \left(\dfrac{\left| \nabla w_{0}\right| ^{p}}{\left(w_{0} + m + \epsilon \right) ^{q}} - \dfrac{\left| \nabla \underline{u} \right| ^{p}}{\left(\underline{u}  + m  \right) ^{q}} \right) dx  \\[4pt]
&\geq  -(q-1)  \int_{\Omega_{3}}  \left| \nabla \underline{u} \right| ^{p}dx  \longrightarrow  0  \text{ as } k \to \infty,
\end{aligned}
\end{gathered}
\end{equation}
where we have utilized the fact that $ \left| \nabla \underline{u} \right|^{p} $ is integrable over the domain $ \mathcal{S}_{\epsilon} $.\\[4pt]
\noindent  \textbf{Estimate of $\mathbf{E}_{2}$.}  In this case, we can express it as follows:
\begin{equation*}
	\begin{gathered}
	\begin{aligned}
	\textbf{E}_{2}   = \left(\iint_{\Omega_{1}\times\Omega_{1}} + \iint_{\Omega_{2}\times\Omega_{2}}  + \iint_{\Omega_{3}\times\Omega_{3}} + 2\iint_{\Omega_{2}\times \Omega_{1}} + 2\iint_{\Omega_{3}\times \Omega_{1}} + 2\iint_{\Omega_{3}\times \Omega_{2}}\right) \dfrac{\textbf{F}_{\underline{u}, w_{0}, \Psi_{m}, \Phi_{m}}(x, y)}{\vert x-y\vert^{N+ s q}} dx dy.
	\end{aligned}
	\end{gathered}
	\end{equation*}
We shall now estimate each integral individually, starting with the following observation:
\begin{equation*}
\begin{gathered}
\Psi_{m} (x) - \Psi_{m} (y) = 0 \text{ and } \Phi_{m} (x) - \Phi_{m} (y) = 0 \text{ for all }  (x, y) \in \Omega_{1} \times \Omega_{1},
\end{gathered}
\end{equation*}
which implies that the first integral vanishes over $\Omega_{1} \times \Omega_{1}$, i.e.,
\begin{equation}\label{equ11}
\begin{gathered}
\iint_{\Omega_{1}\times \Omega_{1}} \dfrac{\textbf{F}_{\underline{u}, w_{0}, \Psi_{m}, \Phi_{m}}(x, y)}{\vert x-y\vert^{N+ s q}} dx dy = 0.
\end{gathered}
\end{equation}
Next, by applying Lemma \ref{Lemma2}, we derive the following result:

{\small 	\begin{equation}\label{equ12}
	\begin{gathered}
	\begin{aligned}
	&\iint_{\Omega_{2} \times \Omega_{2}} \dfrac{\left[ \underline{u}(x) - \underline{u}(y) \right]^{q-1}}{\vert x-y\vert^{N+ s q}} \left( \dfrac{ (\underline{u}(x) + m)^{q} - (w_{0}(x) + m + \epsilon)^{q}}{(\underline{u}(x) + m)^{q-1}} - \dfrac{ (\underline{u}(y) + m)^{q} - (w_{0}(y) + m + \epsilon)^{q}}{(\underline{u}(y) + m)^{q-1}} \right) \\[4pt]
	& + \dfrac{\left[ w_{0}(x) - w_{0}(y) \right]^{q-1}}{\vert x-y\vert^{N+ s q}} \left( \dfrac{(w_{0}(x) + m + \epsilon)^{q} - (\underline{u}(x) + m)^{q}}{(w_{0}(x) + m + \epsilon)^{q-1}} - \dfrac{(w_{0}(y) + m + \epsilon)^{q} - (\underline{u}(y) + m)^{q}}{(w_{0}(y) + m + \epsilon)^{q-1}} \right)  dx dy \geq 0.
	\end{aligned}
	\end{gathered}
\end{equation}}

\noindent  Moreover, by applying Lagrange's Theorem, we obtain:
	\begin{equation*} 
	\begin{aligned}
	\iint_{\Omega_{3} \times \Omega_{3}} &\dfrac{\textbf{F}_{\underline{u}, w_{0}, \Psi_{m}, \Phi_{m}}(x, y)}{\vert x - y \vert^{N + s q}} \, dx \, dy\\[4pt]
	& \geq -k(q-1)\iint_{\Omega_{3} \times \Omega_{3}} \dfrac{\left| \underline{u}(x) - \underline{u}(y) \right|^{q}}{\vert x - y \vert^{N + s q}} \left( \dfrac{\max\left\lbrace (\underline{u}(y) + m)^{q-2},  (\underline{u}(x) + m)^{q-2}\right\rbrace }{(\underline{u}(y) + m)^{q-1} (\underline{u}(x) + m)^{q-1}}\right)  dx dy.
	\end{aligned}
	\end{equation*}

\noindent  Now, define
$$ \omega_{\underline{u}} : = \left\lbrace (x, y) \in  \Omega_{3} \times \Omega_{3} \mid \underline{u}(x) > \underline{u}(y) \right\rbrace,$$
then, noting that $\text{meas}(\Omega_{3}) \to 0$ as $k \to \infty$, we have
\begin{equation} \label{equ13}
	\begin{aligned}
	\iint_{\Omega_{3} \times \Omega_{3}} &\dfrac{\textbf{F}_{\underline{u}, w_{0}, \Psi_{m}, \Phi_{m}}(x, y)}{\vert x - y \vert^{N + s q}} \, dx \, dy\\[4pt]
	& \geq - (q-1)\iint_{(\Omega_{3} \times \Omega_{3}) \cap \omega_{\underline{u}}} \dfrac{\left| \underline{u}(x) - \underline{u}(y) \right|^{q}}{\vert x - y \vert^{N + s q}} \left( \dfrac{\underline{u}(y) + m }{\underline{u}(x) + m}\right)  dx dy\\[4pt]
	&  - (q-1)\iint_{(\Omega_{3} \times \Omega_{3}) \cap \omega^{c}_{\underline{u}}} \dfrac{\left| \underline{u}(x) - \underline{u}(y) \right|^{q}}{\vert x - y \vert^{N + s q}} \left( \dfrac{\underline{u}(x) + m }{\underline{u}(y) + m}\right)  dx dy\\[4pt]
	& \geq   -(q-1)\iint_{\Omega_{3} \times \Omega_{3}} \dfrac{\left| \underline{u}(x) - \underline{u}(y) \right|^{q}}{\vert x - y \vert^{N + s q}} dx dy \to 0 \, \text{ as } \, k \to \infty,
	\end{aligned}
	\end{equation}

\noindent  where we use \cite[Lemma 3.5]{ref01} and Definition \ref{definition1} to establish that the last integral is finite. Furthermore, by leveraging the monotonicity of the following functions
\[
\tau \mapsto \frac{\left[ \tau - \tau_{0} \right]^{q-1}}{\tau^{q-1}}  \text{ and } \tau \mapsto \left[ \tau_{0} - \tau \right]^{q-1},
\]
we obtain
{\small \begin{equation}\label{equ10}
	\begin{gathered}
	\begin{aligned}
	&\iint_{\Omega_{2} \times \Omega_{1}} \frac{\textbf{F}_{\underline{u}, w_{0}, \Psi_{m}, \Phi_{m}}(x, y)}{\vert x-y \vert^{N+ s q}} \, dx \, dy  \\[4pt]
	&= \iint_{\Omega_{2} \times \Omega_{1}} \frac{(\underline{u}(x) + m)^{q} - (w_{0}(x) + m + \epsilon)^{q}}{\vert x-y \vert^{N+ s q}} \left( \frac{\left[\underline{u}(x) - \underline{u}(y)\right]^{q-1}}{(\underline{u}(x) + m)^{q-1}} - \frac{\left[w_{0}(x) - w_{0}(y)\right]^{q-1}}{(w_{0}(x) + m + \epsilon)^{q-1}} \right) dx dy \geq 0.
	\end{aligned}
	\end{gathered}
	\end{equation}}

\noindent  Similarly, we have
\begin{equation}\label{equ14}
	\begin{gathered}
	\begin{aligned}
	&\iint_{\Omega_{3} \times \Omega_{1}} \frac{\textbf{F}_{\underline{u}, w_{0}, \Psi_{m}, \Phi_{m}}(x, y)}{\vert x-y \vert^{N+ s q}} \, dx \, dy  \\[4pt]
	&= \iint_{\Omega_{3} \times \Omega_{1}} \frac{k}{\vert x-y \vert^{N+ s q}} \left( \frac{\left[\underline{u}(x) - \underline{u}(y)\right]^{q-1}}{(\underline{u}(x) + m)^{q-1}} - \frac{\left[w_{0}(x) - w_{0}(y)\right]^{q-1}}{(w_{0}(x) + m + \epsilon)^{q-1}} \right) dx dy \geq 0.
	\end{aligned}
	\end{gathered}
	\end{equation}

\noindent  For the last integral, we have
\begin{equation*}
		\begin{gathered}
			\begin{aligned}
				&\iint_{\Omega_{3} \times \Omega_{2}} \frac{\textbf{F}_{\underline{u}, w_{0}, \Psi_{m}, \Phi_{m}}(x, y)}{\vert x-y \vert^{N+ s q}} \, dx dy  \\[4pt]
				&= \underbrace{ \iint_{\Omega_{3} \times \Omega_{2}} \dfrac{\left[ \underline{u}(x) - \underline{u}(y)\right] ^{q-1}}{\left| x -y\right| ^{N+sq}} \left(\dfrac{k}{(\underline{u}(x) + m)^{q-1}} - \dfrac{(\underline{u}(y) + m)^{q} - (w_{0}(y) + m + \epsilon)^{q}}{(\underline{u}(y) + m)^{q-1}}\right) dx dy }_{\textbf{(*)}} \\[4pt]
				&\underbrace{ -  \iint_{\Omega_{3} \times \Omega_{2}} \dfrac{\left[ w_{0}(x) - w_{0}(y)\right] ^{q-1}}{\left| x -y\right| ^{N+sq}} \left(\dfrac{k}{(w_{0}(x) + m + \epsilon)^{q-1}} - \dfrac{(\underline{u}(y) + m)^{q} - (w_{0}(y) + m + \epsilon)^{q}}{(w_{0}(y) + m + \epsilon)^{q-1}}\right) dx dy}_{\textbf{(**)}} .
			\end{aligned}
		\end{gathered}
\end{equation*}
\noindent  Now we set 
	$$ \textbf{S}_{\underline{u}} : = \left\lbrace (x, y) \in  \Omega_{3} \times \Omega_{2}, \, \underline{u}(x) > \underline{u}(y) \right\rbrace \text{ and } \textbf{S}_{w_{0}} : = \left\lbrace (x, y) \in  \Omega_{3} \times \Omega_{2},\, w_{0}(x) > w_{0}(y) \right\rbrace,$$
	and by Lagrange's Theorem and farctional Picone inequality \cite[Proposition 4.2]{brasco2014convexity}, we obtain
    \begin{equation*}
	\begin{gathered}
	\begin{aligned}
	\textbf{(*)} &\geq  \iint_{\textbf{S}_{\underline{u}} } \dfrac{\left[ \underline{u}(x) - \underline{u}(y)\right] ^{q-1}}{\left| x -y\right| ^{N+sq}}\left(\dfrac{(\underline{u}(y) + m)^{q} - (w_{0}(y) + m + \epsilon)^{q}}{(\underline{u}(x) + m)^{q-1}} - \dfrac{(\underline{u}(y) + m)^{q} - (w_{0}(y) + m + \epsilon)^{q}}{(\underline{u}(y) + m)^{q-1}}\right) dx dy\\[4pt]
	& + \iint_{\textbf{S}^{\text{c}}_{\underline{u}} } \dfrac{\left[ \underline{u}(x) - \underline{u}(y)\right] ^{q-1}}{\left| x -y\right| ^{N+sq}} \left(\dfrac{(\underline{u}(x) + m)^{q} - (w_{0}(x) + m + \epsilon)^{q}}{(\underline{u}(x) + m)^{q-1}} - \dfrac{(\underline{u}(y) + m)^{q} - (w_{0}(y) + m + \epsilon)^{q}}{(\underline{u}(y) + m)^{q-1}}\right) dx dy\\[4pt]
	&\geq - (q-1) \iint_{\textbf{S}_{\underline{u}} } \dfrac{\left| \underline{u}(x) - \underline{u}(y)\right| ^{q}}{\left| x -y\right| ^{N+sq}} \left( \dfrac{\underline{u}(y) + m}{\underline{u}(x) + m}\right) dx dy + \iint_{\textbf{S}^{\text{c}}_{\underline{u}} } \dfrac{\left| \underline{u}(x) - \underline{u}(y)\right| ^{q}}{\left| x -y\right| ^{N+sq}} dx dy\\[4pt]
	& - \iint_{\textbf{S}^{\text{c}}_{\underline{u}} } \dfrac{\left[ \underline{u}(x) - \underline{u}(y)\right] ^{q-1}}{\left| x -y\right| ^{N+sq}} \left(\dfrac{(w_{0}(x) + m + \epsilon)^{q}}{(\underline{u}(x) + m)^{q-1}} - \dfrac{(w_{0}(y) + m + \epsilon)^{q}}{(\underline{u}(y) + m)^{q-1}}\right) dx dy\\[4pt]
	&\geq - (q-1) \iint_{\textbf{S}_{\underline{u}} } \dfrac{\left| \underline{u}(x) - \underline{u}(y)\right| ^{q}}{\left| x -y\right| ^{N+sq}} dx dy+ \iint_{\textbf{S}^{\text{c}}_{\underline{u}} } \dfrac{\left| \underline{u}(x) - \underline{u}(y)\right| ^{q}}{\left| x -y\right| ^{N+sq}} dx dy  - \iint_{\textbf{S}^{\text{c}}_{\underline{u}} } \dfrac{\left| w_{0}(x) - w_{0}(y)\right| ^{q}}{\left| x -y\right| ^{N+sq}} dx dy,
	\end{aligned}
	\end{gathered}
	\end{equation*}

\noindent and 
{\small 	\begin{equation*}
	\begin{gathered}
	\begin{aligned}
	\textbf{(**)} &\geq  \iint_{\textbf{S}_{w_{0}} } \dfrac{\left[ w_{0}(x) - w_{0}(y)\right] ^{q-1}}{\left| x -y\right| ^{N+sq}} \left(\dfrac{(\underline{u}(y) + m)^{q} - (w_{0}(y) + m + \epsilon)^{q}}{(w_{0}(y) + m + \epsilon)^{q-1}} - \dfrac{(\underline{u}(x) + m)^{q} - (w_{0}(x) + m + \epsilon)^{q}}{(w_{0}(x) + m + \epsilon)^{q-1}} \right) dx dy\\[4pt]
	& -  \iint_{\textbf{S}^{\text{c}}_{w_{0}} } \dfrac{\left[ w_{0}(x) - w_{0}(y)\right] ^{q-1}}{\left| x -y\right| ^{N+sq}} \left(\dfrac{k}{(w_{0}(x) + m + \epsilon)^{q-1}} - \dfrac{k}{(w_{0}(y) + m + \epsilon)^{q-1}}\right) dx dy\\[4pt]
	& \geq \iint_{\textbf{S}_{w_{0}}} \dfrac{\left| \underline{u}(x) - \underline{u}(y)\right| ^{q}}{\left| x -y\right| ^{N+sq}} dx dy + \iint_{\textbf{S}_{w_{0}}} \dfrac{\left| w_{0}(x) - w_{0}(y)\right| ^{q}}{\left| x -y\right| ^{N+sq}} dx dy.
	\end{aligned}
	\end{gathered}
	\end{equation*}}

\noindent Then, we have 
\begin{equation}\label{equ15}
\begin{gathered}
\begin{aligned}
\iint_{\Omega_{3} \times \Omega_{2}} &\frac{\textbf{F}_{\underline{u}, w_{0}, \Psi_{m}, \Phi_{m}}(x, y)}{\vert x-y \vert^{N+ s q}} \, dx dy \\[4pt]
&\geq - (q-1) \iint_{\textbf{S}_{\underline{u}} } \dfrac{\left| \underline{u}(x) - \underline{u}(y)\right| ^{q}}{\left| x -y\right| ^{N+sq}} dx dy+ \iint_{\textbf{S}^{\text{c}}_{\underline{u}} } \dfrac{\left| \underline{u}(x) - \underline{u}(y)\right| ^{q}}{\left| x -y\right| ^{N+sq}} dx dy \\[4pt]
&  - \iint_{\textbf{S}^{\text{c}}_{\underline{u}} } \dfrac{\left| w_{0}(x) - w_{0}(y)\right| ^{q}}{\left| x -y\right| ^{N+sq}} dx dy +\iint_{\textbf{S}_{w_{0}}} \dfrac{\left| \underline{u}(x) - \underline{u}(y)\right|  ^{q}}{\left| x -y\right| ^{N+sq}} dx dy \\[4pt]
& + \iint_{\textbf{S}_{w_{0}}} \dfrac{\left| w_{0}(x) - w_{0}(y)\right| ^{q}}{\left| x -y\right| ^{N+sq}} dx dy \quad \to 0 \, \text{ as } \, k \to \infty.
\end{aligned}
\end{gathered}
\end{equation}

\noindent Hence, using \eqref{equ11}-\eqref{equ15} we obtain
 \begin{equation}\label{equ16}
\begin{gathered}
\begin{aligned}
\textbf{E}_{2} \geq 0 \quad \text{ as } \, k \to \infty.
\end{aligned}
\end{gathered}
\end{equation}
\textbf{Estimate of $\mathbf{E}_{3}$.}  Given the monotonicity, it is apparent that
\begin{equation}\label{equ17}
\begin{gathered}
\begin{aligned}
\textbf{E} _{3} &\leq  {\displaystyle\int_{\mathcal{S}_{\epsilon}}} g(x)\left(\dfrac{\underline{u}^{\beta}}{\left(\underline{u} + m\right) ^{q-1}} - \dfrac{(w_{0} + \epsilon)^{\beta}}{\left(w_{0} + m + \epsilon\right) ^{q-1}} \right) \textbf{T}_{k} \left( \left(\left(\underline{u} + m\right) ^{q} - \left( w_{0} + m + \epsilon\right) ^{q}\right)^{+}\right)  dx.
\end{aligned}
\end{gathered}
\end{equation}

\noindent To pass to the limit on the right-hand side of the inequality, we first note that, on one hand, by using the fact that \( \textbf{T}_{k}(s) \leq s \) for \( s \geq 0 \) and applying Lagrange's Theorem, we deduce that
\begin{equation*}
\begin{gathered}
\begin{aligned}
g(x)&\left(\dfrac{\underline{u}^{\beta}}{\left(\underline{u} + m\right) ^{q-1}} - \dfrac{(w_{0} + \epsilon)^{\beta}}{\left(w_{0} + m + \epsilon\right) ^{q-1}} \right)^{+} \textbf{T}_{k} \left( \left(\left(\underline{u} + m\right) ^{q} - \left( w_{0} + m + \epsilon\right) ^{q}\right) ^{+}\right) \\[4pt]
& \leq g(x)\, \underline{u}^{\beta} \left( \dfrac{\left(\underline{u} + m\right) ^{q} - \left( w_{0} + m + \epsilon\right) ^{q}}{\left(\underline{u} + m\right) ^{q-1}}\right)   \leq g(x)\, \underline{u}^{\beta} \,\left( \dfrac{\left(\underline{u} + m\right) ^{q} - m  ^{q}}{\left(\underline{u} + m\right) ^{q-1}}\right) \\[4pt]
& \leq g(x) \underline{u}^{\beta +1} \in L^{1}(\mathcal{S}_{\epsilon}).
\end{aligned}
\end{gathered}
\end{equation*}
By applying the Dominated Convergence Theorem, we obtain that
\begin{equation}\label{equ18}
\begin{gathered}
\begin{aligned}
\lim_{m \to 0}  &{\displaystyle\int_{\mathcal{S}_{\epsilon}}} g(x) \left(\dfrac{\underline{u}^{\beta}}{\left(\underline{u} + m\right) ^{q-1}} - \dfrac{(w_{0} + \epsilon)^{\beta}}{\left(w_{0} + m + \epsilon\right) ^{q-1}} \right)^{+} \textbf{T}_{k} \left( \left(\left(\underline{u} + m\right) ^{q} - \left( w_{0} + m + \epsilon\right) ^{q}\right) ^{+}\right) dx \\[4pt]
& = {\displaystyle\int_{\mathcal{S}_{\epsilon}}} g(x) \left(\underline{u}^{\beta - q+ 1}- (w_{0} + \epsilon)^{\beta - q + 1}\right)^{+} \textbf{T}_{k} \left( \left(\underline{u} ^{q} - \left( w_{0} +  \epsilon\right) ^{q}\right) ^{+}\right) dx .
\end{aligned}
\end{gathered}
\end{equation} 
On the othere hand,  thanks to  Fatou’s lemma, we conclude
\begin{equation}\label{equ19}
\begin{gathered}
\begin{aligned}
\liminf_{m \to0}  &{\displaystyle\int_{\mathcal{S}_{\epsilon}}} g(x) \left(\dfrac{\underline{u}^{\beta}}{\left(\underline{u} + m\right) ^{q-1}} - \dfrac{(w_{0} + \epsilon)^{\beta}}{\left(w_{0} + m + \epsilon\right) ^{q-1}} \right)^{-} \textbf{T}_{k} \left( \left(\left(\underline{u} + m\right) ^{q} - \left( w_{0} + m + \epsilon\right) ^{q}\right) ^{+}\right) dx \\[4pt]
&\geq {\displaystyle\int_{\mathcal{S}_{\epsilon}}} g(x) \left(\underline{u}^{\beta - q+ 1}- (w_{0} + \epsilon)^{\beta - q + 1}\right)^{-} \textbf{T}_{k} \left( \left(\underline{u}^{q} - \left( w_{0} + \epsilon\right) ^{q}\right) ^{+}\right) dx .
\end{aligned}
\end{gathered}
\end{equation}

\noindent   Taking $ \liminf $ in \eqref{equ17} as $ m \to 0 $ and by \eqref{equ18} and \eqref{equ19}, we obtain
\begin{equation*}
	\begin{gathered}
	\begin{aligned}
	\liminf _{m \to 0}\textbf{E} _{3} &\leq  {\displaystyle\int_{\Omega}} g(x) \left(\underline{u}^{\beta - q+ 1}- (w_{0} + \epsilon)^{\beta - q + 1}\right)\textbf{T}_{k} \left( \left(\underline{u} ^{q} - \left( w_{0} + \epsilon\right) ^{q}\right) ^{+}\right) dx.
	\end{aligned}
	\end{gathered}
	\end{equation*}
Furthermore, due to the monotonicity, we have
\[
U_{k} := - g(x) \left(\underline{u}^{\beta - q + 1} - (w_{0} + \epsilon)^{\beta - q + 1}\right) \textbf{T}_{k} \left( \left(\underline{u}^{q} - \left(w_{0} +  \epsilon\right)^{q}\right)^{+} \right) \geq 0.
\]
Since \((U_{k})\) is an increasing sequence, the Monotone Convergence Theorem implies that
\[
\lim_{k \to \infty} \int_\Omega U_k \, dx  = - \int_\Omega g(x) \left(\underline{u}^{\beta - q+ 1}- (w_{0} + \epsilon)^{\beta - q + 1}\right) \left(\underline{u}^{q} - \left( w_{0}  + \epsilon\right) ^{q}\right) ^{+} dx.
\]
Hence, we conclude 
\begin{equation}\label{equ20}
	\begin{gathered}
	\begin{aligned}
	\lim_{k \to \infty}\liminf_{m \to 0} \textbf{E}_{3} &\leq  \int_\Omega g(x) \left(\underline{u}^{\beta - q+ 1}- (w_{0} + \epsilon)^{\beta - q + 1}\right) \left(\underline{u}^{q} - \left( w_{0}  + \epsilon\right) ^{q}\right) ^{+} dx \leq 0.
	\end{aligned}
	\end{gathered}
	\end{equation}

\noindent    Finally, gathering \eqref{equ9}, \eqref{equ16} and \eqref{equ20}, we have
\begin{equation*}
	\begin{gathered}
	\begin{aligned}
	0 \leq  \int_\Omega g(x) \left(\underline{u}^{\beta - q+ 1}- (w_{0} + \epsilon)^{\beta - q + 1}\right) \left(\underline{u}^{q} - \left( w_{0}  + \epsilon\right) ^{q}\right) ^{+} dx \leq 0.
	\end{aligned}
	\end{gathered}
	\end{equation*}
	Consequently,  $ \underline{u} \leq w_{0} + \epsilon \leq \overline{u} + \epsilon.  $ By taking the limit  $ \epsilon \to 0 $ it follows that $ \underline{u} \leq \overline{u}. $ 
\end{proof}
It is worth noting that when considering \( \underline{u}, \overline{u} \in W^{1, p}_{0}(\Omega) \), the proof of Theorem \ref{Theorem1} can be significantly simplified. Specifically, we present the following problem within a fairly general framework, which implicitly addresses our problem in \eqref{P2}:
\begin{equation}\label{equ23}\tag{GP}
-\Delta_{p} u + (-\Delta)^{s}_{q} u = h(x, u), \quad u > 0 \quad \text{in } \Omega; \quad u = 0 \quad \text{in } \mathbb{R}^{N} \setminus \Omega.
\end{equation}
We assume the following hypotheses:
\begin{enumerate}
	\item[\textbf{(h1)}] \( h : \Omega \times \mathbb{R}^{+} \to \mathbb{R}^{+} \) is a Carathéodory function, i.e., \( h(x, \cdot) \) is continuous on \( \mathbb{R}^{+} \) for a.e. \( x \in \Omega \), and \( h(\cdot, s) \) is measurable for all \( s > 0 \).
	\item[\textbf{(h2)}] For a.e. \( x \in \Omega \), the map \( s \mapsto \dfrac{h(x, s)}{s^{q-1}} \) is non-increasing on \( \mathbb{R}^{+} \setminus \{0\} \).
\end{enumerate}
\begin{definition}\label{definition2}
We define a nonnegative function \( u \in W^{1, p}_{0}(\Omega) \) as a weak super-solution to \eqref{equ23} if it satisfies the following conditions:
\begin{enumerate}
	\item[(i)] \( h(\cdot, u(\cdot)) \varphi(\cdot) \in L^{1}(\Omega) \) for all $ \varphi  \in W^{1, p}_{0}(\Omega) $.
	\item[(ii)] For all \( \varphi \in W^{1, p}_{0}(\Omega) \) with \( \varphi \geq 0 \), we have
	{\small
		\begin{equation*}
		\begin{aligned}
		\int_{\Omega} \left[ \nabla u\right]^{p-1} \nabla \varphi \, dx 
		&+ \iint_{\mathbb{R}^{2N}} \dfrac{\left[ u(x) - u(y) \right]^{q-1} \left(\varphi (x)- \varphi(y)\right)}{\vert x-y\vert^{N+ s q}} \, dx \, dy \geq \int_{\Omega} h(x, u) \varphi \, dx.
		\end{aligned}
		\end{equation*}}
\end{enumerate}
If \( u \) satisfies the reverse inequality, it is called a weak sub-solution of \eqref{equ23}. A function \( u \) that satisfies both the weak sub and super-solution conditions is termed a weak solution of \eqref{equ23}.
\end{definition}
\begin{theorem}\label{Theorem2}
Assume \( h \) satisfies \textbf{(h1)} and \textbf{(h2)}. Let \( \underline{u}, \overline{u} \in W^{1, p}_{0}(\Omega) \) be, respectively, a weak sub-solution and super-solution to the problem \eqref{equ23} in the sense of Definition \ref{definition2}. Then, \( \underline{u} \leq \overline{u} \) a.e. in \( \Omega \).
\end{theorem}
\begin{proof}
For any nonnegative pair  $ \Psi,  \Phi  \in W^{1, p}_{0}(\Omega) $ we have 
\begin{equation*}
\begin{gathered}
\begin{aligned}
\int_{\Omega} \left[ \nabla \underline{u}\right] ^{p-1} \nabla \Psi dx &+  \displaystyle \displaystyle\iint_{\mathbb{R}^{2N}} \dfrac{\left[ \underline{u}(x) - \underline{u}(y) \right] ^{q-1} \left(\Psi (x)-\Psi  (y)\right)}{\vert x-y\vert ^{N+ s q}}dxdy \leq {\displaystyle\int_{\Omega}} h(x, \underline{u}) \Psi dx,
\end{aligned}
\end{gathered}
\end{equation*}
and
\begin{equation*}
\begin{gathered}
\begin{aligned}
\int_{\Omega} \left[ \nabla \overline{u}\right] ^{p-1} \nabla \Phi dx &+  \displaystyle \displaystyle\iint_{\mathbb{R}^{2N}} \dfrac{\left[ \overline{u}(x) - \overline{u}(y) \right] ^{q-1} \left(\Phi (x)-\Phi  (y)\right)}{\vert x-y\vert ^{N+ s q}}dxdy \geq {\displaystyle\int_{\Omega}} h(x, \overline{u})  \Phi dx.
\end{aligned}
\end{gathered}
\end{equation*}
Subtracting the above inequalities with test functions
	\[
	\Psi = \frac{\textbf{T}_{k}\left( \left(\left(\underline{u} + m\right)^{q} - \left( \overline{u}+ m\right)^{q}\right)^{+} \right)}{(\underline{u} + m)^{q - 1}}, \quad \Phi = \frac{\textbf{T}_{k}\left( \left( \left( \overline{u}+ m \right)^{q} - \left(\underline{u} + m\right)^{q}\right)^{-} \right)}{\left( \overline{u}+ m \right)^{q - 1}} \in W^{1,p}_{0}(\Omega),
	\]
with $ m > 0, $ we obtain
\begin{equation}\label{equ22}
\begin{gathered}
\begin{aligned}
&\int_{\Omega} \left( \left[ \nabla \underline{u}\right] ^{p-1}\nabla\Psi -  \left[ \nabla \overline{u} \right] ^{p-1}\nabla\Phi \right) dx\\[4pt]
& + \displaystyle \displaystyle\iint_{\mathbb{R}^{2N}}  \dfrac{\left[ \underline{u}(x)  - \underline{u}(y) \right] ^{q-1} \left(\Psi (x)-\Psi (y)\right) - \left[ \overline{u}(x) - \overline{u}(y) \right] ^{q-1} \left(\Phi (x)-\Phi (y)\right)}{\vert x-y\vert ^{N+ s q}} dxdy\\[4pt]
&\leq {\displaystyle\int_{\left\lbrace\underline{u} > \overline{u}\right\rbrace }}  \left( \dfrac{h(x, \underline{u})}{(\underline{u} + m)^{q-1}} - \dfrac{h(x, \overline{u})}{(\overline{u} + m)^{q-1}}\right) \textbf{T}_{k} \left(\left(\underline{u} + m\right) ^{q} - \left(\overline{u} + m\right) ^{q} \right)  dx.
\end{aligned}
\end{gathered}
\end{equation}
By following the same methods used in the proof of Claim 3 within the proof of Theorem \ref{Theorem1}, we find that the left-hand side of \eqref{equ22} is nonnegative. Using Lagrange's Theorem, we obtain
\begin{equation*}
\begin{aligned}
&\left( \dfrac{h(x, \underline{u})}{(\underline{u} + m)^{q-1}} - \dfrac{h(x, \overline{u})}{(\overline{u} + m)^{q-1}} \right)^{+} \textbf{T}_{k} \left( (\underline{u} + m)^{q} - (\overline{u} + m)^{q} \right) \chi_{\left\lbrace \underline{u} > \overline{u} \right\rbrace } \\[4pt]
&  \leq h(x, \underline{u}) \left( \dfrac{(\underline{u} + m)^{q} - (\overline{u} + m)^{q}}{(\underline{u} + m)^{q-1}} \right) \chi_{\left\lbrace \underline{u} > \overline{u} \right\rbrace } \\[4pt]
&  \leq h(x, \underline{u}) \left( \dfrac{(\underline{u} + m)^{q} - m^{q}}{(\underline{u} + m)^{q-1}} \right) \chi_{\left\lbrace \underline{u} > \overline{u} \right\rbrace } \\[4pt]
& \leq h(x, \underline{u}) \, \underline{u} \, \chi_{\left\lbrace \underline{u} > \overline{u} \right\rbrace } \in L^{1}(\Omega),
\end{aligned}
\end{equation*}
where we use Definition \ref{definition2} to ensure the last integral is well-defined. By combining this fact with the proof of Theorem \ref{Theorem1}, we infer that
\begin{equation*}
\begin{aligned}
0 \leq  \int_{\left\lbrace \underline{u} > \overline{u} \right\rbrace} \left( \dfrac{h(x, \underline{u})}{\underline{u}^{q-1}} - \dfrac{h(x, \overline{u})}{\overline{u}^{q-1}} \right)(\underline{u}^{q} - \overline{u}^{q})\, dx \leq 0.
\end{aligned}
\end{equation*}
Hence, we deduce that \( \underline{u} \leq \overline{u} \).
\end{proof}
\begin{remark}
It is important to highlight that the function $h$ may potentially be unbounded both at the origin, with a weak singularity, and at infinity. Additionally, both the sub-solution and super-solution to \eqref{equ23}, in the sense of Definition \ref{definition2}, could also be unbounded.
\end{remark}
Based on these findings, the uniqueness of solutions to problem \eqref{P2} follows in a straightforward manner, as outlined below:
\begin{proof}[\textbf{Proof of Corollary} \ref{corollary}]
	Assume that there exist two weak solutions, \( u \) and \( v \), of \eqref{P2} in \( W^{1, p}_{\text{loc}}(\Omega) \). By considering \( u \) as a sub-solution and \( v \) as a super-solution to \eqref{P2}, the weak comparison principle (Theorem \ref{Theorem1}, and see also Theorem \ref{Theorem2} for the case where \( u \) and \( v \) are in \( W^{1, p}_{0}(\Omega) \)) yields \( u \leq v \) a.e. in \( \Omega \). By interchanging the roles of \( u \) and \( v \), we conclude that \( u = v \) a.e. in \( \Omega \).
\end{proof}
\begin{proof}[\textbf{Proof of Corollary} \ref{corollary2}]
	Without loss of generality, and owing to the properties of rotation and translation invariance, we assume that \( \Omega \) is symmetric with respect to the \( x_{1} \) direction. This symmetry implies that \( f(x_{1}, x') = f(-x_{1}, x') \) and \( g(x_{1}, x') = g(-x_{1}, x') \) for all \( x' \in \mathbb{R}^{N-1} \). We then define \( v(x_{1}, x') = u(-x_{1}, x') \). It can be readily verified that \( v \) serves as a weak solution to the problem \eqref{P2}. Consequently, by utilizing the uniqueness results presented in Corollary \ref{corollary}, we conclude that \( u(x_{1}, x') = v(x_{1}, x') = u(-x_{1}, x') \).
\end{proof}
\section{Study of approximating problems}\label{section3}
We note that the energy functional associated with the problem \eqref{P2} is not \( C^{1} \) due to the presence of a singular term. Consequently, the minimax results of critical point theory cannot be applied directly. Therefore, we need to find methods to address this challenge and work with \( C^{1} \)-functionals. To this end, we introduce the following approximation problems for $ n \in \mathbb{N} $
\begin{equation}\label{P}\tag{En}
-\Delta_{p} u + (-\Delta)^{s}_{q} u = f_{n}(x) \left( u + n^{-1}\right)  ^{-\alpha}+ g_{n}(x) u^{\beta},\, u > 0\quad \text{in }\, \Omega;\quad u = 0\, \text{ in }\,\mathbb{R}^{N}\setminus \Omega, 
\end{equation}
where $ g_{n}(x) = \textbf{T}_{n}(g(x)) $ and $ f_{n} $ represents an approximation of $ f ,$ defined as follows: \\[4pt]
$ \bullet $ In the case of the class \textbf{(F1)}, we define 
\begin{equation}\label{equ39} 
f_{n}(x) = 
\begin{cases}
 \textbf{T}_{n}(f(x)), \quad &\text{if} \quad r \in \left[ 1, \infty\right( ,\\[4pt]
f(x), \quad &\text{if} \quad r = \infty.
\end{cases}
\end{equation}
$ \bullet $ For the class \textbf{(F2)}, we define 
\begin{equation}\label{equ49} 
f_{n}(x) = 
\begin{cases}
\left(f^{-\frac{1}{\delta}}(x) + n^{-\frac{\alpha + p -1}{p - \delta}}\right) ^{-\delta}, \quad &\text{if} \quad x \in \Omega,\\[4pt]
0, \quad &\text{otherwise}.
\end{cases}
\end{equation}
From \eqref{equ24} there exist positive constants $ \mathcal{C}_{1} $ and $ \mathcal{C}_{2}  $ such that: for any $ x\in \Omega $
\begin{equation}\label{equ44}	
\mathcal{C}_{1}  \left(d(x) + n^{-\frac{\alpha + p -1}{p - \delta}}\right) ^{-\delta} \leq f_{n}(x) \leq \mathcal{C}_{2} \left(d(x) + n^{-\frac{\alpha + p -1}{p - \delta}}\right) ^{-\delta}.
\end{equation}
It is straightforward to observe that, since \( \delta < p \), the function \( f_{n} \) is increasing sequence. Now, the existence, uniqueness, and other qualitative properties of weak solutions to the auxiliary problem \eqref{P} are confirmed, as outlined in the following Lemma.
\begin{lemma} \label{lemma2}
	For each $n \in \mathbb{N}$, there exists a unique non-negative weak solution $u_n$ to \eqref{P} in the following sense:
	\begin{itemize}
		\item[$ \bullet $] $ u_{n} \in W^{1, p}_{0}(\Omega).$
		\item[$ \bullet $] For all $ \varphi \in W^{1, p}_{0}(\Omega): $
		\begin{equation}\label{equ45}
		\begin{gathered}
		\begin{aligned}
		&\int_{\Omega} \left[ \nabla u_{n}\right] ^{p-1}  \nabla \varphi\, dx + \displaystyle\iint_{\mathbb{R}^{2 N}} \dfrac{\left[u_{n}(x) - u_{n}(y)\right] ^{q-1}\left(\varphi(x)- \varphi(y)\right)}{\vert x-y\vert ^{N+ s q}}dx dy \\[4pt]
		&\quad = \int_{\Omega} f_{n}(x) \left(u_{n} + n^{-1}\right) ^{- \alpha} \varphi dx + \int_{\Omega}  g_{n}(x)\,u_{n} ^{\beta} \,\varphi \,dx.
		\end{aligned}
		\end{gathered}
		\end{equation}
	\end{itemize}
	Furthermore, we have
	\begin{itemize}
		\item[\textbf{(1)}]  For every $ n \in \mathbb{N}, $ $ u_{n} > 0 $ in $ \Omega, $ and $ u_{n} \in C^{1, \zeta}(\overline{\Omega}) \cap L^{\infty}(\Omega), $ for some $ \zeta \in \left( 0, 1\right) . $
		\item[\textbf{(2)}] The sequence $ (u_{n}) $ is increasing with respect to $ n. $
		\item[\textbf{(3)}] There exist positive constants $ c_{1} $ and $ c_{2} $ independent on $ n $, such that 
		\begin{equation}\label{equ47}
		u_{n}(x) \geq c_{1}\, d(x) \text{ in } \Omega, \, \text{ for every }  n \in \mathbb{N}.
		\end{equation}
In addition, within the class \textbf{(F2)}, assuming $ \delta  > p(1-s) + s(1-\alpha) $, we can obtain a better lower bound for $ u_n $ as follows:
		\begin{equation}\label{equ90}
		u_{n}(x) \geq c_{2}\left( \left( d(x) + n^{- \frac{\alpha + p -1}{p - \delta}}\right) ^{\frac{p - \delta}{\alpha +p -1} } - n^{-1}\right)  \text{ in } \Omega, \, \text{ for every }  n \in \mathbb{N}.
		\end{equation}
	\end{itemize}
\end{lemma}
\noindent To prove this lemma, we analyze the following problem for a fixed $n \in \mathbb{N}$:
\begin{equation}\label{P3}
-\Delta_{p} u + (-\Delta)^{s}_{q} u = k + g_{n}(x) u^{\beta},\quad u > 0 \quad \text{in } \Omega,\quad u = 0 \quad \text{in } \mathbb{R}^{N} \setminus \Omega,
\end{equation}
where $ k \in L^{\infty}(\Omega),$ $ k \geq 0 $ and is not identically zero. More precisely, we have the following:
\begin{lemma}\label{lemma3}
The problem \eqref{P3} possesses a unique positive weak solution $u \in W^{1, p}_{0}(\Omega).$ Moreover, $u$ belongs to $L^{\infty}(\Omega) \cap C^{1, \zeta}(\overline{\Omega})$ for some $\zeta \in (0, 1).$
\end{lemma}
\begin{proof}
	To establish the existence result, we employ a variational approach. Specifically, we consider the energy functional $\mathfrak{J}$ associated with \eqref{P3}, which is defined on $W^{1, p}_{0}(\Omega)$ as follows:
$$ \mathfrak{J}(u) =\dfrac{1}{p} \int_{\Omega} \left| \nabla u\right| ^{p} dx +  \dfrac{1}{q} \iint_{\mathbb{R}^{2N}}\dfrac{\left| u(x) - u(y)\right| ^{q}}{\left| x - y\right| ^{N + sq }} dx dy  - \int_{\Omega} k u dx - \dfrac{1}{\beta+ 1} \int_{\Omega} g_{n}(x) (u^{+})^{\beta + 1} dx . $$
Therefore, we conclude the following:\\[4pt]
$ \bullet $ The functional $\mathfrak{J}$ is well-defined and weakly lower semi-continuous on $W^{1, p}_{0}(\Omega).$\\[4pt]
$ \bullet $ By applying H\"{o}lder's inequality, Theorem \ref{thm3}, and Lemma \ref{Lemma1}, we obtain
$$ \mathfrak{J}(u)  \geq \left\|u\right\| ^{p}_{W^{1, p}_{0}(\Omega)} \left(\dfrac{1}{p} - c_{1} \left\|u\right\| ^{1 -p}_{W^{1, p}_{0}(\Omega)} - c_{2} \left\|g\right\| _{L^{\left( \frac{q^{*}_{s}}{\beta + 1}\right)'}(\Omega)} \left\|u\right\| ^{\beta + 1- p}_{W^{1, p}_{0}(\Omega)}\right), $$
where the constants $ c_{1} ,$ $ c_{2} $ are independent of $ u $. Consequently, since $ \beta < p - 1 $, we deduce that $ \mathfrak{J}(u) $ is coercive on $W^{1, p}_{0}(\Omega)$. Therefore, $ \mathfrak{J} $ admits a global minimizer in $W^{1, p}_{0}(\Omega)$, denoted by $ u_{0}, $ which satisfies the following weak formulation:
\begin{equation*}
\begin{gathered}
\begin{aligned}
&\int_{\Omega} \left[ \nabla u_{0}\right] ^{p-1}  \nabla \varphi\, dx + \displaystyle\iint_{\mathbb{R}^{2N}} \dfrac{\left[u_{0}(x) - u_{0}(y)\right] ^{q-1}\left(\varphi(x)- \varphi(y)\right)}{\vert x-y\vert ^{N+ sq}}\,dx\,dy \\[4pt]
&\quad = \int_{\Omega} k(x) \varphi \,dx + \int_{\Omega}  g_{n}(x)\,\left(u_{0}^{+}\right)^{\beta} \,\varphi \,dx.
\end{aligned}
\end{gathered}
\end{equation*}
By testing the weak formulation with the function $ - u_{0}^{-} $ and applying the following inequality
$$ - \left|x^{-} -y^{-}\right| ^{q}  \geq \left[ x -y\right] ^{q-1} (x^{-} - y^{-}) \text{ for any }  x, y \in \mathbb{R}, $$
we obtain
$$ \int_{\Omega} \left|  \nabla (-u_{0}^{-})\right|  ^{p} dx = 0. $$
As a result, we conclude that $ u_{0} \geq 0 $ a.e. in $ \Omega. $ To demonstrate that $u_0$ is not identically zero, we consider a nonnegative, nontrivial function $\phi \in C^{1}_{c}(\Omega)$. For any $t > 0$, it holds that
\[
\mathfrak{J}(t \phi) \leq t \left( c_1 t^{p  - 1} + c_2 t^{q  - 1} - c_3 \right),
\]
where the constants $c_1$, $c_2$, and $c_3$ are independent of $t$, with $c_3 > 0$ due to $ k \not\equiv 0$. Hence, for sufficiently small $t > 0$, we have $\mathfrak{J}(t \phi) < 0$. Since $\mathfrak{J}(0) = 0$, we conclude that $u_0 \not\equiv 0.$ Next, since $\beta < q - 1$, we observe that the function $s \mapsto \dfrac{k + g_n s^{\beta}}{s^{q-1}}$ is non-increasing. This allows us to apply Theorem \ref{Theorem2}, leading to the conclusion that problem \eqref{P3} has a unique solution. On the other hand, we claim that all weak solutions to the problem \eqref{P3} belongs to $ L^{\infty}(\Omega)$. To this aim, we follow the approach of \cite[Theorem 3.2]{ref05}. Precisely, let $u_0 \in W^{1, p}_{0}(\Omega) $ be a
weak solution to  \eqref{P3}. 
Setting
$$ v_{0} = \dfrac{u_{0}}{\rho \left\| u_{0}\right\| _{L^{q}(\Omega)} }\quad \text{where} \quad \rho = \max \left\lbrace 1, \left\| u_{0} \right\| ^{-1}_{L^{q}(\Omega)}  \right\rbrace.  $$
Noting that $ v_{0} \in W^{1, p}_{0}(\Omega) $ and $ \left\| v_{0} \right\|_{L^{q}(\Omega)} = \rho^{-1},$ we define the function $ w_k $ as:
\begin{equation*}
\begin{cases}
w_k(x) &:= (v_0(x) - 1 + 2^{-k})^{+}\quad \text{ for }\, k \in \mathbb{N},\\[4pt]
w_0(x) &= (v_0(x))^{+}.
\end{cases}
\end{equation*}
We first state the following straightforward observations about $ w_k(x)$:
$$
w_k = 0
\text{ a.e. in } \, \mathbb{R}^{N}\setminus \Omega \quad \text{and} \quad w_k \in W^{1, p}_{0}(\Omega),
$$
and
\begin{equation}\label{equ53}
\begin{cases}
0 \leq 	w_{k+1}(x) \leq w_k(x)\quad\text{a.e. in } \mathbb{R}^N,\\[4pt]
v_0(x) < (2^{k+1} - 1) w_k (x )\quad\text{for } x \in \{ w_{k+1} > 0\}.
\end{cases}
\end{equation}
Also the inclusion
\begin{equation}\label{equ54}
\{ w_{k+1} > 0\}\subseteq \{ w_k > 2^{-(k+1)}\}\quad \text{ holds for all } \, k \in \mathbb{N}.
\end{equation}
Now, we define $ V_k := \| w_k\|_{L^{q}(\Omega)}^{q}$. Next, we present the following claim:
\begin{claim}
	$  V_k  \to 0 $ as $  k \to \infty. $
\end{claim}
\noindent Indeed, since $ 1 < q \leq p,$ $ \beta < q -1, $  $ \rho \, \left\| u_{0} \right\| _{L^{q}(\Omega)} \geq 1 $ and by using the following inequality
\begin{equation}\label{equ63}
\left| 	x ^{+} - y ^{+} \right| ^{q} \leq \left[ x - y \right] ^{q-1} (x ^{+} - y ^{+}), \quad \text{ for any } \, x, y \in \mathbb{R},
\end{equation}
we obtain
\begin{align*}
\| w_{k+1}\| ^{q}_{W^{s, q}_0(\Omega)}  & =  \iint _{\mathbb{R}^{2N}}
\frac{| w_{k+1}(x)-w_{k+1}(y)| ^{q}}{| x-y|^{N+s q} }\,dx dy \\[4pt]
& \leq  \int _{\Omega} \left| \nabla w_{k+1}\right| ^{p} dx + \iint _{\mathbb{R}^{2N}}
\frac{| w_{k+1}(x)-w_{k+1}(y)| ^{q}}{| x-y|^{N+s q} }\,dx dy \\[4pt]
&\leq  (\rho \, \left\| u_{0} \right\| _{L^{q}(\Omega)}) ^{1-p}  \int _{\Omega} \left[ \nabla u_{0}\right] ^{p - 1} \nabla w_{k+1} dx \\[4pt]
&+  (\rho \, \left\| u_{0} \right\| _{L^{q}(\Omega)})^{1-q} \iint _{\mathbb{R}^{2N}}\frac{\left[ u_0(x) - u_0(y)\right] ^{q-1}(w_{k+1}(x) -w_{k+1}(y))}{| x-y|^{N+s q} }\,dx\,dy \\[4pt]
&\leq \int_{\Omega} k(x) w_{k+1} dx  + (\rho \, \left\| u_{0} \right\| _{L^{q}(\Omega)})^{1-q + \beta}  \int_{\Omega} g_{n}(x) v_{0}^{\beta} w_{k+1} dx\\[4pt]
&\leq C \left( \left|\left\lbrace w_{k+1} > 0\right\rbrace  \right| ^{1- \frac{1}{q}} V_{k}^{\frac{1}{q}} + \left|\left\lbrace w_{k+1} > 0\right\rbrace  \right| ^{1- \frac{\beta + 1}{q}} V_{k}^{\frac{\beta + 1}{q}}\right)  .
\end{align*}
Now, from \eqref{equ54}, we have
\begin{align*}
\left|\left\lbrace w_{k+1} > 0 \right\rbrace \right| \leq 2^{q(k+1)} V_{k}.
\end{align*}
Combining this with \eqref{equ53}, we obtain
\begin{equation}\label{equ27}
\| w_{k+1} \|^{q}_{W^{s, q}_0(\Omega)} \leq C \left( 2^{k+1} + 1 \right)^{q-1} V_{k}.
\end{equation}
Moreover, by H\"{o}lder's inequality and Theorem \ref{thm3}, we have
\begin{equation}\label{equ25}
V_{k+1} = \int_{\{ w_{k+1} > 0 \}} w_{k+1}^{q} \, dx
\leq C \| w_{k+1} \|^{q}_{W_0^{s, q}(\Omega)} \left| \left\lbrace w_{k+1} > 0 \right\rbrace \right|^{1 - \frac{q}{q_{s}^{*}}}.
\end{equation}
Thus, we can rewrite inequality \eqref{equ27} as follows:
\begin{equation}\label{equ26}
V_{k+1} \leq C^k V_k^{1 + \vartheta}, \quad \text{for all } k \in \mathbb{N},
\end{equation}
for a suitable constant $C > 1$ and $\vartheta = \dfrac{s q}{N}$. This implies that
\begin{equation}\label{equ28}
V_k \leq \dfrac{\eta^k}{\rho^q}, \quad \text{for all } k \in \mathbb{N},
\end{equation}
where $\eta = C^{-\frac{1}{\vartheta}}$ and $\rho = \max \left\lbrace 1, \left\| u_0 \right\|^{-1}_{L^q(\Omega)}, C^{\frac{1}{q \Gamma^2}} \right\rbrace$. Indeed, by induction, we have: \\[4pt]
$ \bullet $ Clearly $ V_{0}  = \left\| v_{0}^{+} \right\|^{q} _{L^{q}(\Omega)} \leq \left\| v_{0} \right\|^{q} _{L^{q}(\Omega)}  = \dfrac{1}{\rho ^{q}}.$\\[4pt]
$ \bullet $  Assume now that \eqref{equ28} holds for some $k \in \mathbb{N}$. By applying \eqref{equ26}, we obtain
\[
V_{k+1} \leq C^{k} V_k^{1+ \vartheta} \leq \dfrac{\eta^{k+1}}{\rho^q}.
\]
Since $\eta \in (0, 1)$, it follows that
\begin{equation}\label{equ56}
\lim_{k \to \infty} V_k = 0.
\end{equation}
Moreover, since $w_k$ converges to $(v_0 - 1)^{+}$ a.e.\ in $\mathbb{R}^N$, from \eqref{equ56} we conclude that $w_k \to 0$ a.e.\ in $\Omega$. Hence, $v_0 \leq 1$ a.e.\ in $\Omega$, which implies
\[
\left\| u_0 \right\|_{L^{\infty}(\Omega)} \leq \rho \left\| u_0 \right\|_{L^q(\Omega)}.
\]
Thus, we deduce that $u_0 \in L^{\infty}(\Omega)$, and \cite[Theorem 1.1]{antonini2023global} ensures the $C^{1, \zeta}(\overline{\Omega})$-regularity of $u_0$ for some $ \zeta \in (0, 1)$. Moreover, by \cite[Proposition 6.1]{antonini2023global}, we also obtain that $u_0 > 0$ in $\Omega$.
\end{proof}
\begin{proof}[\textbf{Proof of Lemma} \ref{lemma2}] 
For a given \( n \in \mathbb{N} \) and each \( u \in W^{1, p}_{0}(\Omega) \), Lemma \ref{lemma3} ensures the existence of a unique solution \( w \in W^{1, p}_{0}(\Omega) \) to the following problem:
\begin{equation*}\label{P4}
-\Delta_{p} w + (-\Delta)^{s}_{q} w = f_{n}(x) \left(u^{+} + n^{-1}\right)  ^{-\alpha}+ g_{n}(x) w^{\beta},\, w > 0\quad \text{ in }\, \Omega;\quad w = 0\, \text{ in }\,\mathbb{R}^{N}\setminus \Omega.
\end{equation*}
	Thus, the operator \( \mathcal{S} : W^{1, p}_{0}(\Omega) \to W^{1, p}_{0}(\Omega) \), defined by \( \mathcal{S}(u) = w \) is well-defined. Moreover, \( w \) satisfies the following weak formulation:
	\begin{equation*}
	\begin{gathered}
	\begin{aligned}
	&\int_{\Omega} \left[ \nabla w\right] ^{p-1}  \nabla \varphi\, dx + \displaystyle\iint_{\mathbb{R}^{2 N}} \dfrac{\left[w(x) - w(y)\right] ^{q-1}\left(\varphi(x)- \varphi(y)\right)}{\vert x-y\vert ^{N+ s q}}dx dy \\[4pt]
	&\quad = \int_{\Omega} f_{n}(x) \left(u^{+} + n^{-1}\right) ^{- \alpha} \varphi\, dx + \int_{\Omega}  g_{n}(x)\,w^{\beta} \,\varphi \,dx, \quad \forall \varphi \in W^{1, p}_{0}(\Omega).
	\end{aligned}
	\end{gathered}
	\end{equation*}
    By selecting \( w \) as a test function, we establish the following result:
	\begin{equation*}
	\begin{gathered}
	\begin{aligned}
	&\int_{\Omega} \left| \nabla w\right| ^{p} dx \leq \int_{\Omega} f_{n}(x) \left(u^{+} + n^{-1}\right) ^{- \alpha} w\, dx + \int_{\Omega}  g_{n}(x)\,w^{\beta + 1}\,dx.
	\end{aligned}
	\end{gathered}
	\end{equation*}
	Utilizing H\"{o}lder's inequality, Theorem \ref{thm3} and Lemma \ref{Lemma1}, we derive the following:
{\footnotesize \begin{equation}\label{equ31}
	\left\| w\right\| ^{p}_{W^{1, p}_{0}(\Omega)} \leq C \max\left\lbrace\left\| f_{n}\right\| _{L^{\infty}(\Omega)} n^{\alpha} \left| \Omega\right| ^{1-\frac{1}{q^{*}_{s}}},  \left\| g\right\| _{L^{\left(\frac{q^{*}_{s}}{1+ \beta }\right)'}(\Omega)}\right\rbrace \left( \left\| w\right\|_{W^{1, p}_{0}(\Omega)}  +  \left\| w\right\| ^{\beta +1}_{W^{1, p}_{0}(\Omega)}  \right) . 
	\end{equation}}
  
  \noindent Consequently, there exists \( \textbf{R} > 0 \), independent of \( w \), such that the ball \( B(0, \textbf{R}) \) in \( W^{1, p}_{0}(\Omega) \) remains invariant under the operator \( \mathcal{S} \). Moreover, it can be rigorously demonstrated that \( \mathcal{S} \) is both compact and continuous. Specifically, let \( (u_{k}) \) be a bounded sequence in \( W^{1, p}_{0}(\Omega) \). Then, there exists a function \( u \in W^{1, p}_{0}(\Omega) \) such that \( u_{k} \rightharpoonup u \) weakly in \( W^{1, p}_{0}(\Omega) \), \( u_{k} \to u \) strongly in \( L^{p}(\Omega) \), and \( u_{k}(x) \to u(x) \) a.e. in \( \Omega \). Now, let \( w_{k} = \mathcal{S}(u_{k}) \). Our aim is to prove the existence of a subsequence, still denoted by \( (w_{k}) \), that converges strongly in \( W^{1, p}_{0}(\Omega) \) to some limit function \(  w \). To this end, observe that \( w_{k} \) satisfies the following equation:
   \begin{equation}\label{equ29}
   \begin{gathered}
   \begin{aligned}
   &\int_{\Omega} \left[ \nabla w_{k} \right]^{p-1} \nabla \varphi \, dx + \iint_{\mathbb{R}^{2N}} \dfrac{\left[w_{k}(x) - w_{k}(y)\right]^{q-1} \left( \varphi(x) - \varphi(y) \right)}{\vert x - y \vert^{N + sq}} \, dx \, dy \\[4pt]
   &\quad = \int_{\Omega} f_{n}(x) \left(u_{k}^{+} + n^{-1}\right)^{-\alpha} \varphi \, dx + \int_{\Omega} g_{n}(x) w_{k}^{\beta} \varphi \, dx, \quad \forall \varphi \in W^{1, p}_{0}(\Omega).
   \end{aligned}
   \end{gathered}
   \end{equation}
It is straightforward to deduce from \eqref{equ31} that  \( (w_{k}) \) is uniformly bounded in \( W^{1, p}_{0}(\Omega) \). Then, by extracting a subsequence, we have \( w_{k} \rightharpoonup w \) weakly in \( W^{1, p}_{0}(\Omega) \), \( w_{k} \to w \) strongly in \( L^{p}(\Omega) \), and \( w_{k}(x) \to w(x) \) a.e. in \( \Omega \). By exploiting these facts, together with the pointwise convergence of \( u_{k} \) to \( u \), and applying the Dominated Convergence Theorem, we conclude that
\begin{equation}\label{equ34}
\begin{gathered}
\begin{aligned}
&\int_{\Omega} \left[ \nabla w \right]^{p-1} \nabla \varphi \, dx + \iint_{\mathbb{R}^{2N}} \dfrac{\left[w(x) - w(y)\right]^{q-1} \left( \varphi(x) - \varphi(y) \right)}{\vert x - y \vert^{N + sq}} \, dx \, dy \\[4pt]
&\quad = \int_{\Omega} f_{n}(x) \left(u^{+} + n^{-1}\right)^{-\alpha} \varphi \, dx + \int_{\Omega} g_{n}(x) w^{\beta} \varphi \, dx, \quad \forall \varphi \in W^{1, p}_{0}(\Omega).
\end{aligned}
\end{gathered}
\end{equation}
Moreover, by subtracting equations \eqref{equ29} and \eqref{equ34}, and using the test function \( \varphi = w_{k} - w \), we obtain
\begin{equation*}
\begin{gathered}
\begin{aligned}
&\int_{\Omega} \left( \left[ \nabla w_{k} \right]^{p-1} - \left[ \nabla w \right]^{p-1} \right) \nabla (w_{k} - w) \, dx \\[4pt]
&\quad + \iint_{\mathbb{R}^{2N}} \frac{\left( \left[w_{k}(x) - w_{k}(y)\right]^{q-1} - \left[w(x) - w(y)\right]^{q-1} \right) \left( (w_{k} - w)(x) - (w_{k} - w)(y) \right)}{\vert x - y \vert^{N + sq}} \, dx \, dy \\[4pt]
&\quad = \int_{\Omega} f_{n}(x) \left( \left(u_{k}^{+} + n^{-1}\right)^{-\alpha} - \left(u^{+} + n^{-1}\right)^{-\alpha} \right) (w_{k} - w) \, dx  + \int_{\Omega} g_{n}(x) \left( w_{k}^{\beta} - w^{\beta} \right) (w_{k} - w) \, dx.
\end{aligned}
\end{gathered}
\end{equation*}
By applying H\"{o}lder's inequality and Lemma \ref{lemma1} (for \( p \geq 2 \)), we obtain
\begin{equation*}
\begin{gathered}
\begin{aligned}
\int_{\Omega} \left|\nabla (w_{k} - w)\right| ^{p}\, dx  &\leq \left\| f_{n}\right\| _{L^{\infty}(\Omega)} \left\| \left(u_{k}^{+} + n^{-1}\right)^{-\alpha} - \left(u^{+} + n^{-1}\right)^{-\alpha} \right\| _{L^{p'}(\Omega)}\left\|w_{k} - w \right\| _{L^{p}(\Omega)}\\[4pt]
& + n \left\|  w_{k}^{\beta} - w^{\beta} \right\| _{L^{p'}(\Omega)}  \left\| w_{k} - w\right\| _{L^{p}(\Omega)} \\[4pt]
&\leq \left( 2 n^{\alpha} \left\| f_{n}\right\| _{L^{\infty}(\Omega)}  + n \left\|  w_{k}^{\beta} - w^{\beta} \right\| _{L^{p'}(\Omega)} \right)  \left\| w_{k} - w\right\| _{L^{p}(\Omega)}  \to 0 \quad \text{as} \quad k \to \infty,
\end{aligned}
\end{gathered}
\end{equation*}
where $ \left\|  w_{k}^{\beta} - w^{\beta} \right\| _{L^{p'}(\Omega)}  $ is uniformly bounded in $ k, $ since $ \beta p' < p. $ Consequently, the compactness of \( \mathcal{S} \) from \( W^{1, p}_{0}(\Omega) \) to \( W^{1, p}_{0}(\Omega) \) is established. A similar argument demonstrates the compactness of \( \mathcal{S} \) from \( W^{1, p}_{0}(\Omega) \) to \( W^{1, p}_{0}(\Omega) \) for \( 1 < p < 2 \).  Regarding \textbf{continuity}, let \( (u_{k}) \) be an arbitrary sequence such that \( u_{k} \to u_{0} \) in \( W^{1, p}_{0}(\Omega) \). It follows that, up to a subsequence, \( u_{k}(x) \to u_{0}(x) \) a.e. in \( \Omega \). Furthermore, we know that \( w_{k} = \mathcal{S}(u_{k}) \) satisfies

   \begin{equation}\label{equ35}
\begin{gathered}
\begin{aligned}
&\int_{\Omega} \left[ \nabla w_{k} \right]^{p-1} \nabla \varphi \, dx + \iint_{\mathbb{R}^{2N}} \dfrac{\left[w_{k}(x) - w_{k}(y)\right]^{q-1} \left( \varphi(x) - \varphi(y) \right)}{\vert x - y \vert^{N + sq}} \, dx \, dy \\[4pt]
&\quad = \int_{\Omega} f_{n}(x) \left(u_{k}^{+} + n^{-1}\right)^{-\alpha} \varphi \, dx + \int_{\Omega} g_{n}(x) w_{k}^{\beta} \varphi \, dx, \quad \forall \varphi \in W^{1, p}_{0}(\Omega).
\end{aligned}
\end{gathered}
\end{equation}
On the other hand, since \( \mathcal{S} \) is compact, there exists a subsequence, again denoted by \( (w_{k}) \), such that \( w_{k} \rightharpoonup \tilde{w} \) weakly in \( W^{1, p}_{0}(\Omega) \), \( w_{k} \to \tilde{w} \) strongly in \( L^{p}(\Omega) \), and \( w_{k}(x) \to \tilde{w}(x) \) a.e. in \( \Omega \). Combining these facts with the argument employed in the aforementioned proof of the compactness of \( \mathcal{S} \), we can pass to the limit in \eqref{equ35} to obtain
 \begin{equation*}
\begin{gathered}
\begin{aligned}
&\int_{\Omega} \left[ \nabla \tilde{w}  \right]^{p-1} \nabla \varphi \, dx + \iint_{\mathbb{R}^{2N}} \dfrac{\left[\tilde{w} (x) - \tilde{w} (y)\right]^{q-1} \left( \varphi(x) - \varphi(y) \right)}{\vert x - y \vert^{N + sq}} \, dx \, dy \\[4pt]
&\quad = \int_{\Omega} f_{n}(x) \left(u_{0}^{+} + n^{-1}\right)^{-\alpha} \varphi \, dx + \int_{\Omega} g_{n}(x) \tilde{w} ^{\beta} \varphi \, dx, \quad \forall \varphi \in W^{1, p}_{0}(\Omega).
\end{aligned}
\end{gathered}
\end{equation*}
Therefore, applying the uniqueness results, we deduce that \( \tilde{w} = \mathcal{S}(u_{0}) \), which implies the continuity of \( \mathcal{S} \). Moreover, by Schauder’s Fixed Point Theorem, there exists \( u_{n} \in W^{1, p}_{0}(\Omega) \) such that \( \mathcal{S}(u_{n}) = u_{n} \), meaning \( u_{n} \) is a \textbf{solution} to problem \eqref{P}. By using Lemma \ref{lemma3}, we infer that \( u_{n} \in L^{\infty}(\Omega) \cap C^{1, \zeta}(\overline{\Omega}) \) for some \( \zeta \in (0, 1) \).  Next, by defining  $h_{n}(x, s) = \left(s + \frac{1}{n}\right)^{-\alpha} + g_{n}(x) s^{\beta} , $ we observe that for every \( n \in \mathbb{N} \), \( h_{n}(x, s)   \not\equiv 0 \) implies \( u_{n} \not\equiv 0 \).
Moreover, by applying the strong maximum principle from \cite[Proposition 6.1]{antonini2023global}, we deduce that \( u_{n} > 0 \) in \( \Omega \).
Additionally, we have  $h_{n}(x, s) \leq h_{n+1}(x, s), $ for all $ s > 0 \text{ and a.e. in } \Omega.  $ Consequently, \( u_{n} \), the solution of \textbf{(En)}, serves as a sub-solution to the problem \textbf{(En+1)}.  Applying Theorem \ref{Theorem2}, we conclude that \( (u_{n}) \) is an increasing sequence with respect to \( n \). More precisely, we deduce that \( u_{n} \geq u_{1} \) in \( \Omega \) for every \( n \in \mathbb{N} \). Hence, the Hopf's lemma (see \cite[Theorem 1.2]{antonini2023global}) implies that for every \( n \in \mathbb{N} \),
\begin{equation}\label{equ38}
u_{n}(x) \geq c_{1} \, d(x) \text{ in } \Omega,
\end{equation}
where \( c_{1} > 0 \) is a constant independent of \( n \). We point out that, as a result of the monotonicity of \( (u_{n}) \), we can establish the uniqueness of the weak solution to problem \eqref{P}.  Finally, we adapt the proof techniques from the local case in \cite[Lemma 2.2]{giacomoni2021sobolev}. Given that the boundary \( \partial\Omega \) is of class \( C^2 \), it follows from \cite[Lemma 14.6]{gilbarg2015elliptic} that there exists \( \mu > 0 \) such that \( d \in L^{\infty}(\Omega_{\mu}) \), where \( \Omega_{\mu} = \{ x \in \Omega : \text{dist}(x, \partial\Omega) < \mu \} \). Without loss of generality, in Theorem \ref{theorem2}, we assume \( \varrho \leq \min\left\lbrace\frac{\mu}{2}, 1 \right\rbrace \), which guarantees that \( \Delta d \in L^{\infty}(\Omega_{\varrho}) \). Consequently, there exists a constant \( M > 0 \) such that \( |\Delta d| \leq M \) in \( \Omega_{\varrho} \).  Now, let \( \eta \in (0, 1) \) be a constant to be specified later. We define \( \underline{u}_{n}(x) = \eta\, \underline{w}_{\rho}(x) \), where \( \rho > 0 \) is such that
\begin{align*}
\underline{w}_{\rho}(x) =
\begin{cases}
\left( d_{e}(x) + n^{-\frac{\alpha + p - 1}{p - \delta}} \right)^{\frac{p - \delta}{\alpha + p - 1}} - n^{-1} & \text{if } x \in \Omega \cup (\Omega^{\text{c}})_{\rho}, \\
-  n^{-1}& \text{otherwise},
\end{cases}
\end{align*}
where $ (\Omega^{\text{c}})_{\rho} = \left\lbrace x \in \Omega^{\text{c}}\, : \, \text{dist}(x, \partial\Omega) < \rho   \right\rbrace .$ Next, for \( \varphi \in C^{\infty}_{c}(\Omega_{\varrho}) \) with \( \varphi \geq 0 \), and applying Theorem \ref{theorem2} with \( \gamma = \frac{p - \delta}{\alpha + p - 1} \in (0, s) \) (since \( \delta > s(1-\alpha) + p(1 - s) \)), and \( \kappa = \frac{1}{n} \), we obtain
{\footnotesize \begin{equation*}
\begin{gathered}
\begin{aligned}
\iint_{\mathbb{R}^{2N}} \dfrac{\left[\underline{u}_{n}(x) - \underline{u}_{n}(y)\right]^{q-1} \left( \varphi(x) - \varphi(y) \right)}{\vert x - y \vert^{N + sq}} \, dx  dy &\leq C \eta^{q-1} \int_{\Omega_{\varrho}} \left( d(x) + n^{-\frac{\alpha + p -1}{p - \delta}}\right)^{\frac{(p - \delta)(q-1)}{\alpha + p - 1} - qs} \varphi \, dx\\[4pt]
& = C \eta^{q-1} \int_{\Omega_{\varrho}} \left( d(x) + n^{-\frac{\alpha + p -1}{p - \delta}}\right)^{\left( \frac{p - \delta}{\alpha + p - 1} - s\right) (q-1) - s} \varphi \, dx\\[4pt]
&\leq C \eta^{q-1}  \int_{\Omega_{\varrho}} \left( d(x) + n^{-\frac{\alpha + p -1}{p - \delta}}\right)^{\left( \frac{p - \delta}{\alpha + p - 1} - 1\right) (p-1) - 1} \varphi \, dx,
\end{aligned}
\end{gathered}
\end{equation*}}

\noindent where \( C > 0 \) is independent of \( n \).  On the other hand, we have
$$ \nabla\underline{u}_{n} = \eta \left(\frac{p - \delta}{\alpha + p - 1}\right) \left( d(x) + n^{-\frac{\alpha + p - 1}{p - \delta}} \right)^{\frac{p - \delta}{\alpha + p - 1} - 1}\nabla d \text{ in } \Omega_{\varrho}. $$
Moreover, noting that $ |\nabla d| = 1 $, it follows that
{\footnotesize \begin{equation*}
\begin{gathered}
\begin{aligned}
- \int_{\Omega_{\varrho}} \Delta_{p} (\underline{u}_{n})\,  \varphi \, dx &= -C \eta^{p-1} \int_{\Omega_{\varrho}}  \left( d(x) + n^{-\frac{\alpha + p - 1}{p - \delta}} \right)^{\left( \frac{p - \delta}{\alpha + p - 1} - 1\right) (p-1)} \nabla d  \nabla \varphi dx\\[4pt]
& = - C \eta^{p-1} \int_{\Omega_{\varrho}}   \left(\dfrac{p - \delta}{\alpha + p -1} - 1\right)(p-1)  \left( d(x) + n^{-\frac{\alpha + p - 1}{p - \delta}} \right)^{\left( \frac{p - \delta}{\alpha + p - 1} - 1\right) (p-1) - 1} \varphi dx
\\[4pt]
& - C \eta^{p-1} \int_{\Omega_{\varrho}}  \Delta d \left( d(x) + n^{-\frac{\alpha + p - 1}{p - \delta}} \right)^{\left( \frac{p - \delta}{\alpha + p - 1} - 1\right) (p-1)} \varphi dx\\[4pt]
&\leq C \eta^{p-1} \int_{\Omega_{\varrho}} \left( d(x) + n^{-\frac{\alpha + p - 1}{p - \delta}} \right)^{\left( \frac{p - \delta}{\alpha + p - 1} - 1\right) (p-1) - 1} \varphi dx.
\end{aligned}
\end{gathered}
\end{equation*}}

\noindent A some computation shows that
\begin{equation*}
\begin{gathered}
\begin{aligned}
\int_{\Omega_{\varrho}} \left(-\Delta_{p} \underline{u}_{n}  + (-\Delta)^{s}_{q} \underline{u}_{n} \right) \varphi dx \leq C(\eta ^{p-1} + \eta^{q-1})\int_{\Omega_{\varrho}} \left( d(x) + n^{-\frac{\alpha + p - 1}{p - \delta}} \right)^{-\frac{(p - \delta)\alpha}{\alpha + p - 1} - \delta} \varphi dx.
\end{aligned}
\end{gathered}
\end{equation*}
Therefore, using \eqref{equ44}, we deduce that
{\small \begin{equation}\label{equ46}
\begin{gathered}
\begin{aligned}
-\Delta_{p} \underline{u}_{n}  + (-\Delta)^{s}_{q} \underline{u}_{n} \leq C(\eta ^{p-1} + \eta^{q-1}) \left( d(x) + n^{-\frac{\alpha + p - 1}{p - \delta}} \right)^{-\frac{\alpha(p - \delta)}{\alpha + p - 1} } f_{n}(x) + g_{n}(x) \underline{u}_{n} ^{\beta} \text{ in } \Omega_{\varrho}.
\end{aligned}
\end{gathered}
\end{equation}}

\noindent Next, we observe that 
\[
\left( \underline{u}_{n}(x) + n^{-1}\right)^{-\alpha} = \left( \eta \left( d(x) + n^{-\frac{\alpha + p - 1}{p - \delta}} \right)^{\frac{p - \delta}{\alpha + p - 1}} + (1 - \eta) n^{-1} \right)^{-\alpha}, \quad \text{for} \ x \in \Omega_{\varrho}.
\]
We now proceed by distinguishing the following cases:\\[4pt]
\textbf{(1)} If $x \in \Omega_{\varrho}$ such that 
$\eta \left(d(x) + n^{-\frac{\alpha + p - 1}{p - \delta}}\right)^{\frac{p - \delta}{\alpha + p - 1}} \geq (1 - \eta) n^{-1},$
then the following holds:
$$  \left( d(x) + n^{-\frac{\alpha + p -1}{p- \delta}}\right) ^{- \frac{\alpha (p - \delta)}{\alpha + p -1}} \leq  (2 \eta) ^{\alpha} (\underline{u}_{n}(x) + n^{-1})^{ - \alpha}. $$ 
Thus, based on \eqref{equ46} and for sufficiently small \( \eta > 0 \) (independent of \( n \)), we can infer that:
\begin{equation*}
\begin{gathered}
\begin{aligned}
-\Delta_{p} \underline{u}_{n}  + (-\Delta)^{s}_{q} \underline{u}_{n} &\leq  C(\eta ^{p-1} + \eta^{q-1}) (2 \eta) ^{\alpha} \left( \underline{u}_{n}(x) + n^{-1}\right) ^{ - \alpha} f_{n}(x) + g_{n}(x) \underline{u}_{n} ^{\beta}\\[4pt]
&\leq   f_{n}(x) \left( \underline{u}_{n}(x) + n^{-1}\right) ^{ - \alpha} + g_{n}(x) \underline{u}_{n} ^{\beta}.
\end{aligned}
\end{gathered}
\end{equation*}
	\textbf{(2)} If $x \in \Omega_{\varrho}$ such that 
	$\eta \left(d(x) + n^{-\frac{\alpha + p - 1}{p - \delta}}\right)^{\frac{p - \delta}{\alpha + p - 1}} \leq (1 - \eta) n^{-1}. $ In this case, we have 
	$$ n^{\alpha} \leq  \left( 2 (1- \eta)\right)^{\alpha} (\underline{u}_{n}(x) + n^{-1})^{ - \alpha}. $$ 
	Similarly, we can choose \( \eta > 0 \) to be sufficiently small and independent of \( n \) such that
\begin{equation*}
\begin{gathered}
\begin{aligned}
-\Delta_{p} \underline{u}_{n}  + (-\Delta)^{s}_{q} \underline{u}_{n} &\leq  C(\eta ^{p-1} + \eta^{q-1}) n^{\alpha} f_{n}(x) + g_{n}(x) \underline{u}_{n} ^{\beta}\\[4pt]
&\leq C(\eta ^{p-1} + \eta^{q-1})  \left( \underline{u}_{n}(x) + n^{-1}\right) ^{ - \alpha}  f_{n}(x) + g_{n}(x) \underline{u}_{n} ^{\beta}
\\[4pt]
&\leq  f_{n}(x) \left( \underline{u}_{n}(x) + n^{-1}\right) ^{ - \alpha} + g_{n}(x) \underline{u}_{n} ^{\beta}.
\end{aligned}
\end{gathered}
\end{equation*}
In both cases, we choose \( \eta > 0 \) to be sufficiently small and independent of \( n \), such that
\begin{equation*}
-\Delta_{p} \underline{u}_{n} + (-\Delta)^{s}_{q} \underline{u}_{n} \leq f_{n}(x) \left( \underline{u}_{n}(x) + n^{-1}\right) ^{ - \alpha}  + g_{n}(x) \underline{u}_{n}^{\beta}
\quad \text{in } \Omega_{\varrho}.
\end{equation*}
Now, in light of \eqref{equ38}, we choose \( \eta > 0 \) to be sufficiently small and independent of \( n \) such that the following inequality holds:
\begin{equation*}
\underline{u}_{n}(x) \leq \eta\, \text{diam}(\Omega)^{\frac{p - \delta}{\alpha + p - 1}} \leq C_{\varrho} \leq u_{n}(x) \quad \text{in } \Omega \setminus \Omega_{\varrho}.
\end{equation*}
Since \( u_{1} \) is positive and continuous in \( \Omega \), the weak comparison principle (see Lemma \ref{Lemma4}) allows us to conclude that \( \underline{u}_{n}(x) \leq u_{n}(x) \) in \( \Omega_{\varrho} \). Thus, it follows that
\begin{equation*}
\eta \left( \left(d(x) + n^{-\frac{\alpha + p - 1}{p - \delta}}\right)^{\frac{p - \delta}{\alpha + p - 1}} - n^{-1} \right) \leq u_{n}(x) \quad \text{in } \Omega.
\end{equation*}
\end{proof}
\section{Existence, non-existence, and regularity results}
\noindent In this section, we leverage the results from Section \ref{section3} for the regularized problem to establish the existence, non-existence, and regularity of solutions for problem \eqref{P2}, focusing on the class of weight functions \( f \) that satisfy the assumptions \textbf{(F1)} and \textbf{(F2)}. In particular, we present the following:
\subsection{Proofs of the results in the context of Lebesgue weights}
\begin{proof}[\textbf{Proof of Theorem} \ref{Theorem3}] 
Let us assume that \( 0 < \alpha < 1 \), and that \( f \in L^{r}(\Omega) \) such that \( 1 \leq r < \left( \frac{p^{*}}{1 - \alpha}\right) ' \). Let \( u_n \) represent the weak solution to problem \eqref{P}, as established in Lemma \ref{lemma2}. The proof is organized into two distinct parts:\\[4pt]
\textbf{\underline{Part (1)}: A priori estimates.} Our objective is to establish an a priori estimate for \( (u_{n}) \) in \( W^{1, \varrho_{s, p, q, \alpha, r}}_{0}(\Omega) \), where \( \varrho_{s, p, q, \alpha, r} \) is defined as in \eqref{equ50}. For a fixed \( n \in \mathbb{N} \), let \( \epsilon \in \left( 0, \frac{1}{n} \right) \) and \( \frac{\alpha + p - 1}{p} \leq \vartheta < 1 \). By applying Lemma \ref{lemma2} - \textbf{(1)}, it can be readily verified that \( (u_{n} + \epsilon)^{p(\vartheta - 1) +1} - \epsilon^{p(\vartheta - 1) +1} \in W^{1, p}_{0}(\Omega) \), and therefore, qualifies as an admissible test function in \eqref{equ45}, as follows:
\begin{equation*}
\begin{gathered}
\begin{aligned}
&\int_{\Omega} \left[ \nabla u_{n}\right] ^{p-1}  \nabla  (u_{n} + \epsilon)^{p(\vartheta - 1) +1} \, dx \\[4pt]
& + \displaystyle\iint_{\mathbb{R}^{2 N}} \dfrac{\left[u_{n}(x) - u_{n}(y)\right] ^{q-1}\left( (u_{n} + \epsilon)^{p(\vartheta - 1) +1} (x)-  (u_{n} + \epsilon)^{p(\vartheta - 1) +1} (y)\right)}{\vert x-y\vert ^{N+ s q}}dx dy \\[4pt]
&\quad\leq\int_{\Omega} f_{n}(x) \left(u_{n} + n^{-1}\right) ^{- \alpha} (u_{n} + \epsilon)^{p(\vartheta - 1) +1} dx + \int_{\Omega}  g_{n}(x)\,u_{n} ^{\beta}  (u_{n} + \epsilon)^{p(\vartheta - 1) +1} \,dx \\[4pt]
&\quad \leq \int_{\Omega} f_{n}(x)  (u_{n} + \epsilon)^{p(\vartheta - 1) + 1- \alpha}  dx + \int_{\Omega}  g_{n}(x)\,(u_{n} + \epsilon)^{\beta + p(\vartheta - 1) + 1} \,dx.
\end{aligned}
\end{gathered}
\end{equation*}
By applying Fatou’s Lemma and letting \( \epsilon \to 0 \) in the above estimate, we obtain:
\begin{equation*}
\begin{gathered}
\begin{aligned}
&\int_{\Omega} \left[ \nabla u_{n}\right] ^{p-1}  \nabla  u_{n}  ^{p(\vartheta - 1) +1} \, dx + \displaystyle\iint_{\mathbb{R}^{2 N}} \dfrac{\left[u_{n}(x) - u_{n}(y)\right] ^{q-1}\left( u_{n}  ^{p(\vartheta - 1) +1} (x)- u_{n} ^{p(\vartheta - 1) +1} (y)\right)}{\vert x-y\vert ^{N+ s q}}dx dy \\[4pt]
&\qquad \leq \int_{\Omega} f_{n}(x)  u_{n}  ^{p(\vartheta - 1) + 1- \alpha}  dx + \int_{\Omega}  g_{n}(x)\,u_{n} ^{\beta + p(\vartheta - 1) + 1} \,dx.
\end{aligned}
\end{gathered}
\end{equation*}
By performing some computations and applying the inequality from \cite[Lemma A.2]{ref07}, we obtain:
\begin{equation}\label{equ40}
\begin{aligned}
&  \dfrac{p(\vartheta - 1) + 1}{\vartheta^{p}}\int_{\Omega} \left| \nabla u^{\vartheta}_{n}\right| ^{p} \, dx + \dfrac{(p(\vartheta - 1) + 1) q^{q}}{(p(\vartheta - 1) + q)^{q}} \iint_{\mathbb{R}^{2 N}} \dfrac{\left| u^{\frac{p(\vartheta - 1) + q}{q}}_{n}(x) - u^{\frac{p(\vartheta - 1) + q}{q}}_{n}(y)\right| ^{q}}{\vert x-y\vert ^{N+ s q}} \, dx \, dy \\[4pt]
&\qquad \leq \int_{\Omega} f_{n}(x)  u_{n} ^{p(\vartheta - 1) + 1 - \alpha} \, dx + \int_{\Omega} g_{n}(x)\, u_{n} ^{\beta + p(\vartheta - 1) + 1} \, dx,
\end{aligned}
\end{equation}
and noting that
\begin{equation}\label{equ51}
\int_{\Omega} \left| \nabla u^{\vartheta}_{n}\right| ^{p} \, dx  < \infty,
\end{equation}
and
\[
\iint_{\mathbb{R}^{2 N}} \dfrac{\left| u^{\frac{p(\vartheta - 1) + q}{q}}_{n}(x) - u^{\frac{p(\vartheta - 1) + q}{q}}_{n}(y)\right| ^{q}}{\vert x-y\vert ^{N+ s q}} \, dx \, dy < \infty,
\]
the latter fact, together with fractional Sobolev embeddings (see Theorem \ref{thm3}), implies that
\begin{equation*}
\begin{aligned}
\left( \int_{\Omega} u_{n}^{\frac{q^{*}_{s} (p(\vartheta - 1) + q)}{q}} dx\right) ^{\frac{q}{q^{*}_{s}}} \leq C \left( \int_{\Omega} f_{n}(x)  u_{n}  ^{p(\vartheta - 1) + 1- \alpha}  dx + \int_{\Omega}  g_{n}(x)\,u_{n} ^{\beta + p(\vartheta - 1) + 1} \,dx\right) .
\end{aligned}
\end{equation*}
At this point, we distinguish between two cases. When \( r = 1 \), we set \( \vartheta = \vartheta_{1} = \frac{\alpha + p - 1}{p} \) and observe that \( \beta + \alpha < q - 1 + \alpha < \frac{q^{*}_{s} (p(\vartheta_{1} - 1) + q)}{q} \), which leads to the following estimate:
\begin{equation}\label{equ42}
\begin{aligned}
&\left( \int_{\Omega} u_{n}^{\frac{q^{*}_{s} (p(\vartheta_{1} - 1) + q)}{q}} \, dx \right)^{\frac{q}{q^{*}_{s}}} \\[4pt]
&\leq C \max \left\lbrace \| f \|_{L^{1}(\Omega)}, \| g \|_{L^{m_{s, p, q, \alpha, \beta, 1}}(\Omega)} \right\rbrace \left( 1 + \left( \int_{\Omega} u_{n}^{\frac{q^{*}_{s} (p(\vartheta_{1} - 1) + q)}{q}}  \, dx \right)^{\frac{q (\alpha + \beta)}{q^{*}_{s} (p(\vartheta_{1} - 1) + q)}} \right),
\end{aligned}
\end{equation}
where \( m_{s, p, q, \alpha, \beta, 1} \) is defined by \eqref{equ41}, and \( C > 0 \) is a constant independent of \( n \). Consequently,
\begin{center}
	\( (u_{n}) \) is uniformly bounded in \( L^{\frac{N(\alpha + q - 1)}{N - sq}}(\Omega) \).
\end{center}
Moreover, for \( 1 < r < r_{p, \alpha} \), applying H\"{o}lder's inequality yields:
\begin{equation*}
	\left( \int_{\Omega}u_{n}^{\frac{q^{*}_{s} (p(\vartheta - 1) + q)}{q}} \, dx \right)^{\frac{q}{q^{*}_{s}}} \leq C \left( \| f \|_{L^{r}(\Omega)} \left( \int_{\Omega} u_{n}^{(p(\vartheta - 1) + 1 - \alpha) r'} \, dx \right)^{\frac{1}{r'}} + \int_{\Omega} g_{n}(x) \, u_{n}^{\beta + p (\vartheta - 1) + 1} \, dx \right).
	\end{equation*}
Next, we choose \( \vartheta \) such that
\begin{equation*}
\frac{q^{*}_{s} (p(\vartheta - 1) + q)}{q}  = (p(\vartheta - 1) + 1 - \alpha) r',
\end{equation*}
and noting that \( r_{p, \alpha} < \frac{N}{p} < \frac{N}{s q} \), this simplifies to
\begin{equation}\label{equ68}
\vartheta = \vartheta_{r} =  \underbrace{ \frac{r(N - sq)( \alpha + p - 1) - N(r-1)(p-q)}{p(N -rsq)} < 1}_{r <  r_{p, \alpha} < \left( \frac{q^{*}_{s}}{1- \alpha}\right) '}.
\end{equation}
Thus, we obtain the following estimate:
{\footnotesize \begin{equation}\label{equ43}
\begin{aligned}
& C \left( \int_{\Omega} u_{n}^{\frac{q^{*}_{s} (p(\vartheta_{r} - 1) + q)}{q}} \, dx \right)^{\frac{q}{q^{*}_{s}}} \\[4pt]
&\leq \max \left\lbrace \| f \|_{L^{r}(\Omega)}, \| g \|_{L^{m_{s, p, q, \alpha, \beta, r}}(\Omega)} \right\rbrace \left( \left( \int_{\Omega} u_{n}^{\frac{q^{*}_{s} (p(\vartheta_{r} - 1) + q)}{q}} \, dx \right)^{\frac{1}{r'}} + \left( \int_{\Omega} u_{n}^{\frac{q^{*}_{s} (p(\vartheta_{r} - 1) + q)}{q}}  \, dx \right)^{\frac{q(\beta + p(\vartheta_{r} - 1) + 1)}{q^{*}_{s} (p(\vartheta_{r} - 1) +q)}} \right),
\end{aligned}
\end{equation}}

\noindent where \( m_{s, p, q, \alpha, \beta, r} \) is defined by \eqref{equ41}. We previously utilized the following condition 	\( \beta + p(\vartheta - 1) + 1 < p (\vartheta - 1) + q \). By applying this condition again, together with the inequality \( \frac{q}{q^{*}_{s}} > \frac{1}{r'} \), it follows that 
\begin{equation}\label{equ64}
 (u_{n})  \text{ remains uniformly bounded in }  L^{\frac{N r (\alpha + q - 1)}{N - sqr}}(\Omega).
\end{equation}
Referring back to \eqref{equ40}, and utilizing the estimates \eqref{equ42} and \eqref{equ43}, along with the appropriate choices of \( \vartheta_{r} \) and Sobolev’s embeddings (noting \eqref{equ51}), we obtain
\begin{equation*}
\left( \int_{\Omega}  u^{\vartheta_{r} p^{*}}_{n} \, dx\right) ^{\frac{p}{p^{*}}}  \leq C \int_{\Omega} \left| \nabla u^{\vartheta_{r}}_{n}\right| ^{p}  \, dx  \leq C,
\end{equation*}
where \( C \) is a constant independent of \( n \). This result, in combination with H\"{o}lder's inequality (noting that \( \varrho_{s, p, q, \alpha, r} < p \)), leads to:
\begin{equation*}
\begin{aligned}
\int_{\Omega}\left|\nabla u_{n} \right|^{\varrho_{s, p, q, \alpha, r} } dx &= \int_{\Omega}\dfrac{\left|\nabla u_{n} \right|^{\varrho_{s, p, q, \alpha, r} } }{u_{n}^{(1- \vartheta_{r}) \varrho_{s, p, q, \alpha, r} }} \, u_{n}^{(1- \vartheta_{r}) \varrho_{s, p, q, \alpha, r} }dx\\[4pt]
& \leq \left(\int_{\Omega}\dfrac{\left|\nabla u_{n} \right|^{p} }{u_{n}^{(1- \vartheta_{r}) p}} dx\right) ^{\frac{\varrho_{s, p, q, \alpha, r} }{p}} \left(\int_{\Omega}  u_{n}^{\frac{p (1- \vartheta_{r}) \varrho_{s, p, q, \alpha, r}  }{p-  \varrho_{s, p, q, \alpha, r} }} dx\right) ^{\frac{p- \varrho_{s, p, q, \alpha, r}  }{p}}\\[4pt]
& = \vartheta_{r}^{\varrho_{s, p, q, \alpha, r} } \left(\int_{\Omega} \left|\nabla u_{n}^{\vartheta_{r}}\right| ^{p} dx\right) ^{\frac{\varrho_{s, p, q, \alpha, r} }{p}}  \left(\int_{\Omega}   u^{\vartheta_{r} p^{*}}_{n}  dx\right) ^{\frac{p- \varrho_{s, p, q, \alpha, r} }{p}} \\[4pt]
&  < C \text{  (independent of $ n $) }.
\end{aligned}
\end{equation*}
It follows, up to a subsequence, that
\begin{equation}\label{equ52}
u_{n} \rightharpoonup u  \text{ in } W^{1, \varrho_{s, p, q, \alpha, r} }_{0}(\Omega),\,  u_{n} \to u  \text{ in } L^{\varrho_{s, p, q, \alpha, r} }(\Omega)  \text{ and }   u_{n}(x) \to u(x)  \text{ a.e. in }  \Omega. 
\end{equation} 
\textbf{\underline{Part (2)} : Passing to the limit in the weak formulation \eqref{equ45} and the regularity of weak solutions.} First, we observe that the convergence properties in \eqref{equ52} do not allow for an immediate passage to the limit in \eqref{equ45}. To address this, we note that the properties of $ \vartheta_{r} $ ensures $ \varrho_{s, p, q, \alpha, r} > p - 1 $. Now, let $ \varphi \in C^{\infty}_{c}(\Omega)$.\\[4pt]
$ \bullet $ For the local term on the left-hand side of \eqref{equ45}, we obtain the following:
\begin{equation*} 
\begin{aligned}
\left| \int_{\Omega} \left( \left[ \nabla u_{n} \right]^{p-1} - \left[ \nabla u \right]^{p-1} \right) \nabla \varphi \, dx \right| 
& \leq \left\| \nabla \varphi \right\|_{L^{\infty}(\Omega)} \left( \int_{\Omega} \left| \nabla u_{n} \right|^{p-1} dx + \int_{\Omega} \left| \nabla u \right|^{p-1} dx \right) \\[4pt]
& \leq C \left\| \nabla \varphi \right\|_{L^{\infty}(\Omega)} \left( \left\| u_{n} \right\|_{W^{1, \varrho_{s, p, q, \alpha, r}}_{0}(\Omega)}^{p-1} + \left\| u \right\|_{W^{1, \varrho_{s, p, q, \alpha, r}}_{0}(\Omega)}^{p-1} \right) \leq C,
\end{aligned}
\end{equation*}
for some constant $ C > 0 $ independent of $ n $.  On the other hand, from \eqref{equ47}, it is evident that \( (u_{n}) \) remains bounded away from zero in the interior of the domain $ \Omega $. Thus, we can apply the results from \cite[Theorem 2.1 \& Remark 2.2]{dal1998almost}  to conclude that \( \nabla u_{n}(x) \to \nabla u(x) \) a.e. in \( \Omega \) as \( n \to \infty \). Consequently, we have
$$  \left( \left[ \nabla u_{n} \right]^{p-1} - \left[ \nabla u \right]^{p-1} \right) \nabla \varphi  \to 0 \quad \text{a.e. in } \Omega \quad \text{as} \quad n \to \infty. $$
By Vitali's Convergence Theorem, we obtain
\begin{equation}\label{equ58}
\int_{\Omega} \left[ \nabla u_{n} \right]^{p-1} \nabla \varphi \, dx \to \int_{\Omega}  \left[ \nabla u \right]^{p-1} \nabla \varphi \, dx \quad \text{as} \quad n \to \infty.
\end{equation}
$ \bullet $ For the non-local term, we adapt the proof techniques from \cite[Theorem 3.6]{ref01}. First, we have the following:
{\small \begin{equation} \label{equ55}
\begin{aligned}
&\left|\displaystyle\iint_{\mathbb{R}^{2 N}} \dfrac{\left( \left[u_{n}(x) - u_{n}(y)\right] ^{q-1} - \left[u(x) - u(y)\right] ^{q-1}\right) \left(\varphi(x)- \varphi(y)\right)}{\vert x-y\vert ^{N+ s q}}dx dy \right| \\[4pt]
&\leq \displaystyle\iint_{\mathbb{R}^{2 N}} \dfrac{\left( \left| u_{n}(x) - u_{n}(y)\right| + \left| u(x) - u(y)\right|\right) ^{q-2} \left| (u_{n}(x) - u(x)) -  (u_{n}(y) - u(y))\right| \left| \varphi(x)- \varphi(y)\right| }{\vert x-y\vert ^{N+ s q}}dx dy,
\end{aligned}
\end{equation}}
where we applied Lemma \ref{lemma1}. At this point, we establish the following claim:
\begin{claim}
Let $ \epsilon > 0 $. There exists a compact set $ \mathcal{K} = \mathcal{K}(\epsilon) \subset \mathbb{R}^{2N} $ such that, for all $ n \in \mathbb{N} $:
{\small 	\begin{equation} \label{equ48}
	\iint_{\mathbb{R}^{2 N} \setminus \mathcal{K}} \dfrac{\left( \left| u_{n}(x) - u_{n}(y)\right| + \left| u(x) - u(y)\right|\right) ^{q-2} \left|  (u_{n}(x) - u(x)) -  (u_{n}(y) - u(y))\right|  \left| \varphi(x)- \varphi(y)\right| }{\vert x-y\vert ^{N+ s q}}dx dy \leq \dfrac{\epsilon}{2}.
	\end{equation}}
\end{claim}
\noindent Indeed, by applying H\"{o}lder's inequality, we obtain
\begin{equation*}
\begin{aligned}
&\iint_{\mathbb{R}^{2 N} \setminus \mathcal{K}} \dfrac{\left( \left| u_{n}(x) - u_{n}(y)\right| + \left| u(x) - u(y)\right|\right) ^{q-2} \left| (u_{n}(x) - u(x)) - (u_{n}(y) - u(y))\right|  \left| \varphi(x)- \varphi(y)\right| }{\vert x-y\vert ^{N+ s q}}dx dy\\[4pt]
& \leq \iint_{\mathbb{R}^{2 N} \setminus \mathcal{K}} \dfrac{\left( \left| u_{n}(x) - u_{n}(y)\right| + \left| u(x) - u(y)\right|\right) ^{q-1} \left| \varphi(x)- \varphi(y)\right| }{\vert x-y\vert ^{N+ s q}}dx dy\\[4pt]
&\leq C \left( \displaystyle\iint_{\mathbb{R}^{2 N}} \dfrac{\left| u_{n}(x) - u_{n}(y)\right|^{\varrho_{s, p, q, \alpha, r}}}{\vert x - y\vert ^{N + s \varrho_{s, p, q, \alpha, r}}} dx \, dy + \displaystyle\iint_{\mathbb{R}^{2 N}} \dfrac{\left| u(x) - u(y)\right|^{\varrho_{s, p, q, \alpha, r}}}{\vert x - y\vert ^{N + s \varrho_{s, p, q, \alpha, r}}} dx \, dy \right)^{\frac{q-1}{\varrho_{s, p, q, \alpha, r}}}  \\[4pt]
&\quad\times \left( \displaystyle\iint_{\mathbb{R}^{2 N} \setminus \mathcal{K}} \dfrac{\left| \varphi(x) - \varphi(y)\right|^{\frac{\varrho_{s, p, q, \alpha, r}}{\varrho_{s, p, q, \alpha, r} - q + 1}}}{\vert x - y\vert ^{N + s \frac{\varrho_{s, p, q, \alpha, r}}{\varrho_{s, p, q, \alpha, r} - q + 1}}} dx \, dy \right)^{\frac{\varrho_{s, p, q, \alpha, r} - q + 1}{\varrho_{s, p, q, \alpha, r}}}.
\end{aligned}
\end{equation*}
Using Lemma \ref{Lemma1} and \eqref{equ52}, it follows that
\begin{equation*}
\begin{aligned}
&\iint_{\mathbb{R}^{2 N} \setminus \mathcal{K}} \dfrac{\left( \left| u_{n}(x) - u_{n}(y)\right| + \left| u(x) - u(y)\right|\right) ^{q-2} \left| (u_{n}(x) - u(x)) - (u_{n}(y) - u(y))\right|  \left| \varphi(x)- \varphi(y)\right| }{\vert x-y\vert ^{N+ s q}}dx dy\\[4pt]
&\quad\leq C \left( \displaystyle\iint_{\mathbb{R}^{2 N} \setminus \mathcal{K}} \dfrac{\left| \varphi(x) - \varphi(y)\right|^{\frac{\varrho_{s, p, q, \alpha, r}}{\varrho_{s, p, q, \alpha, r} - q + 1}}}{\vert x - y\vert ^{N + s \frac{\varrho_{s, p, q, \alpha, r}}{\varrho_{s, p, q, \alpha, r} - q + 1}}} dx \, dy \right)^{\frac{\varrho_{s, p, q, \alpha, r} - q + 1}{\varrho_{s, p, q, \alpha, r}}},
\end{aligned}
\end{equation*}
for some constant $ C > 0 $ independent of $ \mathcal{K} $. Hence, since $ \varphi \in C^{\infty}_{c}(\Omega), $ there exists a compact set $ \mathcal{K} $, dependent on $ \epsilon $, such that \eqref{equ48} holds. Similarly, using the same arguments, we obtain for an arbitrary measurable subset $ \mathbf{E} \subset \mathcal{K} $ the following limits, uniformly in $ n $:
\begin{equation}
\begin{aligned}
&\displaystyle\iint_{\mathbf{E}} \dfrac{\left( \left| u_{n}(x) - u_{n}(y)\right| + \left| u(x) - u(y)\right|\right) ^{q-2} \left| (u_{n}(x) - u(x)) - (u_{n}(y) - u(y))\right|  \left| \varphi(x)- \varphi(y)\right| }{\vert x-y\vert ^{N+ s q}}dx dy  \\[4pt]
&\quad\leq C \left( \displaystyle\iint_{\mathbf{E}} \dfrac{\left| \varphi(x) - \varphi(y)\right|^{\frac{\varrho_{s, p, q, \alpha, r}}{\varrho_{s, p, q, \alpha, r} - q + 1}}}{\vert x - y\vert ^{N + s \frac{\varrho_{s, p, q, \alpha, r}}{\varrho_{s, p, q, \alpha, r} - q + 1}}} dx \, dy \right)^{\frac{\varrho_{s, p, q, \alpha, r} - q + 1}{\varrho_{s, p, q, \alpha, r}}} \to 0 \quad \text{as } \text{meas}(\mathbf{E}) \to 0.
\end{aligned}
\end{equation}
Moreover, as $n \to \infty$,
\[
\dfrac{\left( \left| u_{n}(x) - u_{n}(y)\right| + \left| u(x) - u(y)\right|\right) ^{q-2} \left| (u_{n}(x) - u(x)) - (u_{n}(y) - u(y))\right|  \left| \varphi(x)- \varphi(y)\right| }{\vert x-y\vert ^{N+ s q}} \to 0 \quad \text{a.e. in } \mathbb{R}^{2 N}.
\]
Thus, by Vitali’s Theorem, for all $ \epsilon > 0 $, there exists $ n_{0} > 0 $ such that, if $ n \geq n_{0} $, we have
{\small \begin{equation} \label{equ57}
\displaystyle\iint_{\mathbf{E}} \dfrac{\left( \left| u_{n}(x) - u_{n}(y)\right| + \left| u(x) - u(y)\right|\right) ^{q-2}\left| (u_{n}(x) - u(x)) - (u_{n}(y) - u(y))\right| \left| \varphi(x)- \varphi(y)\right| }{\vert x-y\vert ^{N+ s q}}dx dy  \leq \dfrac{\epsilon}{2}.
\end{equation}}

\noindent  Gathering \eqref{equ55}, \eqref{equ48}, and \eqref{equ57}, we obtain as $ n \to \infty $:
\begin{equation}\label{equ59}
\displaystyle\iint_{\mathbb{R}^{2 N}} \dfrac{\left[ u_{n}(x) - u_{n}(y)\right]^{q-1}\left(\varphi(x) - \varphi(y)\right)}{\vert x-y\vert ^{N+ s q}} dx \, dy \to \displaystyle\iint_{\mathbb{R}^{2 N}} \dfrac{\left[ u(x) - u(y)\right]^{q-1} \left(\varphi(x) - \varphi(y)\right)}{\vert x - y\vert ^{N + s q}} dx \, dy.
\end{equation}
Finally, for the right-hand side of \eqref{equ45}, by using the facts that \( f_{n} \leq f \) and \eqref{equ47}, we have
\[
\left| f_{n} \left( u_{n} + n^{-1} \right)^{-\alpha} \varphi \right| \leq c_{\text{supp}(\varphi)}  \left\| \varphi \right\|_{L^{\infty}(\Omega)} f \in L^{1}(\Omega),
\]
which allows us, by applying the Lebesgue Dominated Convergence Theorem, to obtain
\begin{equation}\label{equ60}
\int_{\Omega} f_{n}(x) \left( u_{n} + n^{-1} \right)^{-\alpha} \varphi \, dx  \to  \int_{\Omega} f(x) u^{-\alpha} \varphi \, dx  \quad \text{as} \quad n \to \infty.
\end{equation}
Next, by employing \( g_{n} \leq g \), \( \beta < p(\vartheta_{r} - 1) + q \),   \eqref{equ64} and the embedding \( L^{m_{s, p, q, \alpha, \beta, r}}(\Omega) \hookrightarrow L^{\left( \frac{q^{*}_{s}(p(\vartheta_{r} - 1) + q)}{\beta q} \right)'}(\Omega), \) we obtain
\begin{equation*}
\int_{\Omega} g_{n}(x) u_{n}^{\beta} \varphi \, dx \leq \left\| \varphi \right\|_{L^{\infty}(\Omega)} \left\| g \right\|_{L^{\left( \frac{q^{*}_{s}(p(\vartheta_{r} - 1) + q)}{\beta q} \right)'}(\Omega)} \left\| u_{n} \right\|^{\beta}_{L^{\frac{q^{*}_{s}(p(\vartheta_{r} -1) + q)}{q}}(\Omega)} < C.
\end{equation*}
Thus, by Vitali's Convergence Theorem, we obtain
\begin{equation}\label{equ61}
\int_{\Omega} g_{n}(x)\,u_{n} ^{\beta} \,\varphi \,dx  \to  \int_{\Omega} g(x)\,u^{\beta} \,\varphi \,dx  \quad \text{as} \quad n \to \infty.
\end{equation}
Finally, gathering \eqref{equ58}, \eqref{equ59}, \eqref{equ60}, and \eqref{equ61}, and passing to the limit in \eqref{equ45}, as $ n \to \infty $, we infer that
\begin{equation*}
\begin{aligned}
&\int_{\Omega} \left[ \nabla u \right]^{p-1} \nabla \varphi \, dx + \iint_{\mathbb{R}^{2N}} \frac{\left[ u(x) - u(y)\right]^{q-1} \left( \varphi(x) - \varphi(y) \right)}{\vert x - y \vert^{N + sq}} \, dx \, dy  = \int_{\Omega} f(x) u^{-\alpha} \varphi \, dx + \int_{\Omega} g(x) u^{\beta} \varphi \, dx.
\end{aligned}
\end{equation*}
On the other hand, from \textbf{Part (1)}, we know that the sequence \((u^{\vartheta_{r}}_{n})\) is uniformly bounded in \(W^{1, p}_{0}(\Omega)\). Combining this fact with \eqref{equ52}, we deduce, up to a subsequence, that
\begin{equation*}
u^{\vartheta_{r}}_{n} \rightharpoonup u^{\vartheta_{r}} \quad \text{in } W^{1, p}_{0}(\Omega), \quad u^{\vartheta_{r}}_{n} \to u^{\vartheta_{r}} \quad \text{in } L^{l}(\Omega) \text{ for } 1 \leq l < p^{*}, \quad \text{and} \quad u^{\vartheta_{r}}_{n}(x) \to u^{\vartheta_{r}}(x) \text{ a.e. in } \Omega.
\end{equation*}
Thus, by the reflexivity of \(W^{1, p}_{0}(\Omega)\), it follows that \( u^{\vartheta_{r}} \in W^{1, p}_{0}(\Omega) \). Furthermore, since $ (u_{n}) $ is uniformly bounded in \( L^{\frac{N r (\alpha + q - 1)}{N - sqr}}(\Omega) \), using a similar argument, we conclude that \( u \in L^{\frac{N r (\alpha + q - 1)}{N - sqr}}(\Omega) \).
\end{proof}
\begin{proof}[\textbf{Proof of Theorem} \ref{Theorem6}] In this proof, we adapt the techniques employed in \cite[Proof of Theorem 2.5]{arora2023combined}. Indeed, let \( u_{1} \) and \( u_{2} \) be two weak solutions to problem \eqref{P2} with \( f_{1}, f_{2} \in L^{r}(\Omega) \) for \( 1 \leq r < r_{p, \alpha} \), where \( r_{p, \alpha} \) is defined as in \eqref{equ77} (as established in Theorem \ref{Theorem3}). We now consider the sequences \( (u_{1,n}) \) and \( (u_{2,n}) \) of solutions to the approximating problem \eqref{P}, associated with \( f_{1,n} \) and \( f_{2,n} \) as defined in \eqref{equ39}, respectively. Namely, for any \( \Phi,  \Psi \in W^{1, p}_{0}(\Omega) \), we have
	\begin{equation*}
	\begin{gathered}
	\begin{aligned}
	&\int_{\Omega} \left[ \nabla u_{1, n}\right] ^{p-1}  \nabla \Phi\, dx + \iint_{\mathbb{R}^{2N}} \dfrac{\left[ u_{1, n}(x) - u_{1, n}(y)\right] ^{q-1} \left(\Phi(x)- \Phi(y)\right)}{\vert x-y\vert ^{N+ s q}} \, dx \, dy \\[4pt]
	&\quad = \int_{\Omega} f_{1, n}(x) \left( u_{1, n} + n^{-1} \right) ^{-\alpha} \Phi\, dx,
	\end{aligned}
	\end{gathered}
	\end{equation*}
	and
	\begin{equation*}
\begin{gathered}
\begin{aligned}
&\int_{\Omega} \left[ \nabla u_{2, n}\right] ^{p-1}  \nabla\Psi  \, dx + \iint_{\mathbb{R}^{2N}} \dfrac{\left[ u_{2, n}(x) - u_{2, n}(y)\right] ^{q-1} \left( \Psi (x)- \Psi(y)\right)}{\vert x-y\vert ^{N+ s q}} \, dx \, dy \\[4pt]
&\quad = \int_{\Omega} f_{2, n}(x) \left( u_{2, n} + n^{-1} \right) ^{-\alpha} \Psi \, dx.
\end{aligned}
\end{gathered}
\end{equation*}	
Additionally, we have
	\begin{equation}\label{equ66}
	u_{1,n} \rightharpoonup u_{1} \text{ in } W^{1, \varrho_{s, p, q, \alpha, r}}_{0}(\Omega),\, u_{1,n} \to u_{1} \text{ in } L^{\varrho_{s, p, q, \alpha, r}}(\Omega), \text{ and } u_{1,n}(x) \to u_{1}(x) \text{ a.e. in } \Omega,
	\end{equation} 
	\begin{equation}\label{equ67}
	u_{2,n} \rightharpoonup u_{2} \text{ in } W^{1, \varrho_{s, p, q, \alpha, r}}_{0}(\Omega),\, u_{2,n} \to u_{2} \text{ in } L^{\varrho_{s, p, q, \alpha, r}}(\Omega), \text{ and } u_{2,n}(x) \to u_{2}(x) \text{ a.e. in } \Omega. 
	\end{equation}
Now, for a fixed \( n \in \mathbb{N} \), let \( \epsilon \in \left( 0, \frac{1}{n} \right) \) and \( \vartheta_{r} \) be defined as in \eqref{equ68}. By applying Lemma \ref{lemma2} - \textbf{(1)}, it can be readily verified that \( ((u_{1, n} - u_{2, n})^{+} + \epsilon)^{p(\vartheta_{r} - 1) + 1} - \epsilon^{p(\vartheta_{r} - 1) + 1} \in W^{1, p}_{0}(\Omega) \), thus qualifying as an admissible test function in the two expressions above. Then, after subtracting, we obtain
{\small \begin{equation}\label{equ78}
	\begin{gathered}
	\begin{aligned}
	&  \int_{\Omega} \left( \left[ \nabla u_{1, n}\right]^{p-1} - \left[ \nabla u_{2, n}\right]^{p-1} \right) \nabla \Xi_{u_{1, n}, u_{2, n}, \epsilon}\, dx \\[4pt]
	&+   \iint_{\mathbb{R}^{2N}} \dfrac{\left( \left[ u_{1, n}(x) - u_{1, n}(y) \right]^{q-1} - \left[ u_{2, n}(x) - u_{2, n}(y) \right]^{q-1}\right)\left( \Xi_{u_{1, n}, u_{2, n}, \epsilon}(x) - \Xi_{u_{1, n}, u_{2, n}, \epsilon} (y) \right)}{\vert x - y \vert^{N + s q}} dx dy \\[4pt]
	& \leq   \int_{\Omega} f_{1, n}(x) \left( u_{1, n} + n^{-1} \right) ^{-\alpha}\Xi_{u_{1, n}, u_{2, n}, \epsilon}  dx -    \int_{\Omega} f_{2, n}(x) \left( u_{2, n} + n^{-1}\right) ^{-\alpha}  \left( \Xi_{u_{1, n}, u_{2, n}, \epsilon} - \epsilon^{p(\vartheta_{r} - 1) + 1}\right)  dx,
	\end{aligned}
	\end{gathered}
	\end{equation}}

\noindent where \( \Xi_{u_{1, n}, u_{2, n}, \epsilon} := ((u_{1, n} - u_{2, n})^{+}  + \epsilon)^{p(\vartheta_{r} - 1) + 1} \). Now, by applying Fatou’s Lemma, we can pass to the limit on the left-hand side of the previous inequality. Moreover, we have
\begin{equation*} 
\begin{gathered}
\begin{aligned}
&- \liminf_{\epsilon \to 0} \int_{\Omega}  f_{2, n}(x) \left( u_{2, n} + n^{-1} \right)^{-\alpha}\left( \Xi_{u_{1, n}, u_{2, n}, \epsilon} - \epsilon^{p(\vartheta_{r} - 1) + 1} \right)  dx \\[4pt]
&\leq - \int_{\Omega} f_{2, n}(x) \left( u_{2, n} + n^{-1} \right)^{-\alpha} \left( (u_{1, n} - u_{2, n})^{+} \right)^{p(\vartheta_{r} - 1) + 1} \, dx.
\end{aligned}
\end{gathered}
\end{equation*}
On the other hand, from \eqref{equ64}, we have
\begin{equation*} 
\begin{gathered}
\begin{aligned}
\int_{\Omega} f_{1, n}(x) \left( u_{1, n} + n^{-1} \right) ^{-\alpha} \Xi_{u_{1, n}, u_{2, n}, \epsilon}  dx &\leq \int_{\Omega} f_{1, n}(x) \left( u_{1, n} + n^{-1}\right)^{p(\vartheta_{r} - 1) + 1 - \alpha} dx \\[4pt]
&\leq \left\| f_{1} \right\|_{L^{r}(\Omega)} \left\| u_{1, n} + 1 \right\|_{L^{\frac{N r (\alpha + q - 1)}{N - sqr}}(\Omega)} < C.
\end{aligned}
\end{gathered}
\end{equation*}
Hence, after applying these results and taking the limit in \eqref{equ78} as $ \epsilon \to 0 $, we obtain
{\scriptsize \begin{equation}\label{equ76}
\begin{gathered}
\begin{aligned}
& \underbrace{ \int_{\Omega} \left( \left[ \nabla u_{1, n}\right]^{p-1} - \left[ \nabla u_{2, n}\right]^{p-1} \right) \nabla ((u_{1, n} - u_{2, n})^{+})^{p(\vartheta_{r} - 1) + 1} \, dx}_{\textbf{I}_{1}}  \\[4pt]
&+  \underbrace{ \iint_{\mathbb{R}^{2N}} \dfrac{\left( \left[ u_{1, n}(x) - u_{1, n}(y) \right]^{q-1} - \left[ u_{2, n}(x) - u_{2, n}(y) \right]^{q-1}\right)\left( ((u_{1, n} - u_{2, n})^{+})^{p(\vartheta_{r} - 1) + 1} (x) - ((u_{1, n} - u_{2, n})^{+} )^{p(\vartheta_{r} - 1) + 1} (y) \right)}{\vert x - y \vert^{N + s q}} dx dy }_{\textbf{I}_{2}} \\[4pt]
& \leq   \underbrace{ \int_{\Omega} f_{1, n}(x) \left( u_{1, n} + n^{-1} \right) ^{-\alpha}\left( (u_{1, n} - u_{2, n})^{+}\right)^{p(\vartheta_{r} - 1) + 1}  dx}_{\textbf{I}_{3}}   -   \underbrace{ \int_{\Omega} f_{2, n}(x) \left( u_{2, n} + n^{-1} \right)^{-\alpha} \left( (u_{1, n} - u_{2, n})^{+} \right)^{p(\vartheta_{r} - 1) + 1} \, dx}_{\textbf{I}_{4}} .
\end{aligned}
\end{gathered}
\end{equation}}

\noindent \textbf{Estimate of \( \textbf{I}_{1} \).} Using Lemma \ref{lemma1}, for \( p \geq 2 \) (similar result holds in the case \( 1 < p < 2 \), see \cite[Lemma 4.2]{arora2022multiplicity}), and for all \( \textbf{c} > 0 \), we obtain
\begin{equation}\label{equ69}
\begin{gathered}
\begin{aligned}
\textbf{I}_{1} &\geq C(p, \vartheta_{r}) \int_{\Omega} \dfrac{\left| \nabla (u_{1, n} - u_{2, n})^{+}\right| ^{p}}{\left((u_{1, n} - u_{2, n})^{+}\right) ^{p(1 - \vartheta_{r})}} dx \\[4pt]
&\geq C(p, \vartheta_{r}) \int_{\Omega} \dfrac{\left| \nabla (u_{1, n} - u_{2, n})^{+}\right| ^{p}}{\left((u_{1, n} - u_{2, n})^{+} + \textbf{c}\right) ^{p(1 - \vartheta_{r})}} dx  \\[4pt]
&\geq C(p, \vartheta_{r}) \int_{\Omega} \left| \nabla \left( (u_{1, n} - u_{2, n})^{+} + \textbf{c}\right) ^{\vartheta_{r}}\right| ^{p}dx.
\end{aligned}
\end{gathered}
\end{equation}

\noindent \textbf{Estimate of \( \textbf{I}_{2} \).} After straightforward calculations, we arrive at the following expression for \( (x, y) \in \mathbb{R}^{2N} \):
{\small
	\begin{equation*}
	\begin{gathered}
	\begin{aligned}
	&\left( \left[ u_{1, n}(x) - u_{1, n}(y) \right]^{q-1} - \left[ u_{2, n}(x) - u_{2, n}(y) \right]^{q-1} \right) \left( \left( (u_{1, n} - u_{2, n})^{+} \right)^{p(\vartheta_{r} - 1) + 1}(x) - \left( (u_{1, n} - u_{2, n})^{+} \right)^{p(\vartheta_{r} - 1) + 1}(y) \right) \\[4pt]
	& = \left( \left[ u_{1, n}(x) - u_{1, n}(y) \right]^{q-1} - \left[ u_{2, n}(x) - u_{2, n}(y) \right]^{q-1} \right) \left( (u_{1, n} - u_{2, n})(x) -(u_{1, n} - u_{2, n})(y) \right) \mathcal{H}(x, y) \geq 0,
	\end{aligned}
	\end{gathered}
	\end{equation*}
}

\noindent where
$$ \mathcal{H}(x, y) := \dfrac{ \left( (u_{1, n} - u_{2, n})^{+} \right)^{p(\vartheta_{r} - 1) + 1}(x) - \left( (u_{1, n} - u_{2, n})^{+} \right)^{p(\vartheta_{r} - 1) + 1}(y)}{ (u_{1, n} - u_{2, n})(x) - (u_{1, n} - u_{2, n})(y)},$$
and we have used the fact that \( \mathcal{H}(x, y) \) is non-negative and Lemma \ref{lemma1}, which leads to
\begin{equation}\label{equ70}
\begin{gathered}
\begin{aligned}
\textbf{I}_{2} \geq 0.
\end{aligned}
\end{gathered}
\end{equation}
\noindent \textbf{Estimate of \( \textbf{I}_{3} \).} From \eqref{equ64}, we can obtain the following:
\begin{equation} \label{equ73}
\begin{gathered}
\begin{aligned}
\int_{\Omega} f_{1, n}(x) \left( u_{1, n} + n^{-1} \right) ^{-\alpha}\left( (u_{1, n} - u_{2, n})^{+}\right)^{p(\vartheta_{r} - 1) + 1} dx &\leq \int_{\Omega} f_{1, n}(x)  u_{1, n}^{p(\vartheta_{r} - 1) + 1 -\alpha} dx\\[4pt]
&\leq \left\| f_{1} \right\|_{L^{r}(\Omega)} \left\| u_{1, n}  \right\|_{L^{\frac{N r (\alpha + q - 1)}{N - sqr}}(\Omega)}  < C.
\end{aligned}
\end{gathered}
\end{equation}
Then, by applying Vitali's Convergence Theorem in conjunction with \eqref{equ66} and \eqref{equ67}, we obtain:
\begin{equation} \label{equ71}
\begin{gathered}
\begin{aligned}
\textbf{I}_{3} \to  \int_{\Omega} f_{1}(x) u_{1}^{-\alpha} \left( (u_{1} - u_{2})^{+}\right)^{p(\vartheta_{r} - 1) + 1}  \, dx \quad \text{as } n \to \infty.
\end{aligned}
\end{gathered}
\end{equation}
\noindent \textbf{Estimate of \( \textbf{I}_{4} \).} By the pointwise convergence of \( u_{1, n} \) and \( u_{2, n} \) to \( u_{1} \) and \( u_{2} \) (see \eqref{equ66} and \eqref{equ67}), respectively, and by applying Fatou’s Lemma, we have
\begin{equation}\label{equ75}
\begin{gathered}
\begin{aligned}
- \liminf_{n \to \infty} \textbf{I}_{4} \leq - \int_{\Omega} f_{2}(x) u_{2} ^{-\alpha} \left( (u_{1} - u_{2})^{+} \right)^{p(\vartheta_{r} - 1) + 1} \, dx \leq 0.
\end{aligned}
\end{gathered}
\end{equation}
On the one hand, from \eqref{equ69}, \eqref{equ70}, and \eqref{equ73}, we get that \( \left( \left( (u_{1, n} - u_{2, n})^{+} + \mathbf{c} \right)^{\vartheta_{r}} \right) \) is uniformly bounded in \( W^{1, p}_{0}(\Omega) \) with respect to \( n \). Combining this with \eqref{equ66} and \eqref{equ67}, we deduce, up to a subsequence, that
$$ \left( (u_{1, n} - u_{2, n})^{+} + \mathbf{c} \right)^{\vartheta_{r}} \rightharpoonup \left( (u_{1} - u_{2})^{+} + \mathbf{c} \right)^{\vartheta_{r}} \text{ in } W^{1, p}_{0}(\Omega)\quad \text{as } n \to \infty.$$
Hence, 
\begin{equation*}
\begin{gathered}
\begin{aligned}
\int_{\Omega} \left| \nabla \left( (u_{1} - u_{2})^{+} + \mathbf{c} \right)^{\vartheta_{r}} \right| ^{p} \, dx &\leq  \liminf_{n \to \infty}  \int_{\Omega} \left| \nabla \left( (u_{1, n} - u_{2, n})^{+} + \mathbf{c} \right)^{\vartheta_{r}} \right| ^{p} \, dx.
\end{aligned}
\end{gathered}
\end{equation*}
On the other hand, from the monotonicity of the sequence \( (\nabla \left( (u_{1} - u_{2})^{+} + \mathbf{c} \right)^{\vartheta_{r}}) \) with respect to \( \mathbf{c} \), we get, by the Monotone Convergence Theorem, as \( \mathbf{c} \to 0 \):
  \begin{equation}\label{equ72}
  \begin{gathered}
  \begin{aligned}
   \int_{\Omega} \left| \nabla \left( (u_{1} - u_{2})^{+}\right) ^{\vartheta_{r}}\right| ^{p}dx \leq  \liminf_{\textbf{c} \to 0}  \int_{\Omega} \left| \nabla \left( (u_{1, n} - u_{2, n})^{+} + \textbf{c}\right) ^{\vartheta_{r}}\right| ^{p}dx.
  \end{aligned}
  \end{gathered}
  \end{equation}
Gathering \eqref{equ69}, \eqref{equ70}, \eqref{equ71},  \eqref{equ75}, and \eqref{equ72},   and passing to the limit in \eqref{equ76} as $ n \to \infty $ we obtain
 \begin{equation*}
\begin{gathered}
\begin{aligned}
  C(p, \vartheta_{r}) \int_{\Omega} \left| \nabla \left( (u_{1} - u_{2})^{+}\right) ^{\vartheta_{r}}\right| ^{p}dx \leq  \int_{\Omega} \left( f_{1}(x) u_{1} ^{-\alpha} -  f_{2}(x) u_{2} ^{-\alpha}\right) \left( (u_{1} - u_{2})^{+}\right)^{p(\vartheta_{r} - 1) + 1}  dx.
  \end{aligned}
\end{gathered}
\end{equation*}
Since \( 1 < q \leq p \), by H\"{o}lder's inequality and Sobolev embeddings (see Theorem \ref{thm0}), we get
 \begin{equation*}
\begin{gathered}
\begin{aligned}
C(p, \vartheta_{r}) \int_{\Omega} \left| \nabla \left( (u_{1} - u_{2})^{+}\right) ^{\vartheta_{r}}\right| ^{p}dx &\leq  \int_{\Omega} \left( f_{1}(x) -  f_{2}(x) \right) \left( (u_{1} - u_{2})^{+}\right)^{p(\vartheta_{r} - 1) + 1 - \alpha}  dx\\[4pt]
&\leq  \left\| (f_{1}- f_{2} )^{+} \right\| _{L^{r}(\Omega)} \left(  \int_{\Omega}\left( (u_{1} - u_{2})^{+}\right)^{\frac{q^{*}_{s} (p(\vartheta_{r} - 1) + q)}{q}} \, dx\right) ^{\frac{1}{r'}}\\[4pt]
& \leq C \left\| (f_{1}- f_{2} )^{+} \right\| _{L^{r}(\Omega)} \left(  \int_{\Omega}\left( \left( (u_{1} - u_{2})^{+}\right)^{\vartheta_{r}}\right) ^{p^{*} } \, dx\right) ^{\frac{q^{*}_{s} (p(\vartheta_{r} - 1) + q)}{q p^{*} \vartheta_{r} r'}}\\[4pt]
& \leq C  \left\| (f_{1}- f_{2} )^{+} \right\| _{L^{r}(\Omega)} \left(  \int_{\Omega} \left| \nabla \left( (u_{1} - u_{2})^{+}\right) ^{\vartheta_{r}}\right| ^{p}dx \right) ^{\frac{q^{*}_{s} (p(\vartheta_{r} - 1) + q)}{q p \vartheta_{r} r'}},
\end{aligned}
\end{gathered}
\end{equation*}
where \( C \) is independent of \( u_{1} \) and \( u_{2} \). Finally, we conclude that
 \begin{equation*}
\begin{gathered}
\begin{aligned}
\left\|  \left( (u_{1} - u_{2})^{+}\right) ^{\vartheta_{r}}\right\|  _{W^{1, p}_{0}(\Omega)} \leq C \left\| (f_{1} - f_{2} )^{+} \right\|^{\frac{\vartheta_{r}}{p +\alpha -1}} _{L^{r}(\Omega)} .
\end{aligned}
\end{gathered}
\end{equation*}
\end{proof}
\begin{proof}[\textbf{Proof of Theorem} \ref{Theorem4}]
\textbf{(i)} Consider \( 0 < \alpha < 1 \) and assume that \( f \in L^{r}(\Omega) \) with \( r_{p, \alpha} \leq r \leq \infty \), where \( r_{p, \alpha} \) is defined as in \eqref{equ77}. Let \( u_n \) denote the weak solution to \eqref{P}, as established in Lemma \ref{lemma2}. By choosing \( \varphi = u_n \) in the weak formulation \eqref{equ45} and applying H\"{o}lder inequality along with Sobolev embeddings, we get
	\begin{equation*}
	\left\| u_{n} \right\|^{p}_{W^{1, p}_{0}(\Omega)} \leq C \left( \left\| f \right\|_{L^{r}(\Omega)} \left\| u_{n} \right\|_{W^{1, p}_{0}(\Omega)}^{1 - \alpha} + \left\| g \right\|_{L^{\left( \frac{p^{*}}{\beta + 1}\right)'}(\Omega)} \left\| u_{n} \right\|_{W^{1, p}_{0}(\Omega)}^{\beta + 1} \right).
	\end{equation*}
Thus, from the above estimate, and since \( \beta < q - 1 \), the sequence \( (u_{n}) \) is uniformly bounded in \( W^{1, p}_{0}(\Omega) \). Consequently, there exists \( u \in W^{1, p}_{0}(\Omega) \) such that
	\begin{equation}\label{equ62}
	u_n \rightharpoonup u \quad \text{in } W^{1, p}_{0}(\Omega), \quad u_n \to u \quad \text{in } L^{l}(\Omega) \text{ for } 1 \leq l < p^{*}, \quad \text{and} \quad u_n(x) \to u(x) \text{ a.e. in } \Omega.
	\end{equation}
Now, from \eqref{equ47}, the sequence \( (u_n) \) remains uniformly bounded away from zero in the interior of \( \Omega \). This allows us to apply the results from \cite[Theorem 2.1 \& Remark 2.2]{dal1998almost}, leading to the conclusion that \( \nabla u_n(x) \to \nabla u(x) \) a.e. in \( \Omega \) as \( n \to \infty \). Consequently, by utilizing this fact in conjunction with Vitali's Theorem on one hand and the weak convergence property \eqref{equ62} on the other, together with \eqref{equ60} and \eqref{equ61}, we can pass to the limit in \eqref{equ45}, thereby establishing our claim. 

We assert that all weak solutions belong to \( L^{\infty}(\Omega) \) for \( \frac{N}{p} < r < \frac{p^{*}}{\beta} \). To this end, we employ classical arguments from the seminal work of Stampacchia \cite{Stampacchia1965}. Specifically, let \( k \geq 1 \) and define the function \( \mathbf{G}_{k}(u_{n}) = (u_{n} - k)^{+} \) as a test function in the weak formulation of \eqref{equ45}. This yields the following equation:
\begin{equation*}
\begin{gathered}
\begin{aligned}
&\int_{\Omega} \left[ \nabla u_{n}\right]^{p-1} \nabla \mathbf{G}_{k}(u_{n}) \, dx + \iint_{\mathbb{R}^{2N}} \frac{\left[u_{n}(x) - u_{n}(y)\right]^{q-1} \left(\mathbf{G}_{k}(u_{n})(x) - \mathbf{G}_{k}(u_{n})(y)\right)}{|x-y|^{N+s q}} \, dx \, dy \\[4pt]
&\quad = \int_{\Omega} f_{n}(x) \left(u_{n} + n^{-1}\right)^{-\alpha} \mathbf{G}_{k}(u_{n}) \, dx + \int_{\Omega} g_{n}(x) \, u_{n}^{\beta} \, \mathbf{G}_{k}(u_{n}) \, dx.
\end{aligned}
\end{gathered}
\end{equation*}
Consider the set 
\[
\Omega_{k} := \{ x \in \Omega \mid u_{n}(x) \geq k \}.
\]
Since \(   r < \frac{p^{*}}{\beta}\), by using the Sobolev embedding theorem, inequality \eqref{equ63}, and H\"{o}lder's inequality, we obtain
{\small \begin{equation*}
\begin{aligned}
&\left( \int_{\Omega_{k}} \left| \mathbf{G}_{k}(u_{n}) \right|^{p^{*}} \, dx \right)^{\frac{p}{p^{*}}}\leq C \left( \|f\|_{L^{r}(\Omega)} + \|g\|_{L^{\frac{p^{*} r}{p^{*} - \beta r}}(\Omega)} \left\| u_{n} \right\|_{W^{1, p}_{0}(\Omega)}^{\beta} \right) \times \left( \int_{\Omega_{k}} \left| \mathbf{G}_{k}(u_{n}) \right|^{p^{*}} \, dx \right)^{\frac{1}{p^{*}}} \left| \Omega_{k} \right|^{1 - \frac{1}{r} - \frac{1}{p^{*}}}.
\end{aligned}
\end{equation*}}
Let $ h > 0 $ be such that $ 1 \leq k < h $. Then, it follows that $ \Omega_{h} \subset \Omega_{k} $, and for every $ x \in \Omega_{h} $, we have $ u_{n}(x) \geq h $, implying $ \mathbf{G}_{k}(u_{n}) = u_{n}(x) - k \geq h - k $ in $ \Omega_{h} $. Hence, from \eqref{equ62}, we obtain
\begin{equation*}
\begin{aligned}
&\left( h - k \right)^{p-1} \left| \Omega_{h} \right|^{\frac{p-1}{p^{*}}} \leq  \left( \int_{\Omega_{k}} \left| \mathbf{G}_{k}(u_{n}) \right|^{p^{*}} \, dx \right)^{\frac{p-1}{p^{*}}} \leq C \left( \|f\|_{L^{r}(\Omega)} + \|g\|_{L^{\frac{p^{*} r}{p^{*} - \beta r}}(\Omega)} \right)  \left| \Omega_{k} \right|^{1 - \frac{1}{r} - \frac{1}{p^{*}}},
\end{aligned}
\end{equation*}
which further implies
\begin{equation*}
\begin{aligned}
\left| \Omega_{h} \right| \leq \dfrac{C \left( \|f\|_{L^{r}(\Omega)} + \|g\|_{L^{\frac{p^{*} r}{p^{*} - \beta r}}(\Omega)} \right)^{\frac{p^{*}}{p-1}}}{\left( h - k \right)^{p^{*}}} \left| \Omega_{k} \right|^{\frac{p^{*}}{p-1} \left( 1 - \frac{1}{r} - \frac{1}{p^{*}} \right)}.
\end{aligned}
\end{equation*}
Since \( r > \frac{N}{p} \), it follows that
$$ \frac{p^{*}}{p-1} \left( 1 - \frac{1}{r} - \frac{1}{p^{*}} \right) > 1. $$
Therefore, by applying Lemma \ref{Lemma5}, we deduce that \( u_{n} \in L^{\infty}(\Omega) \), and consequently \( u \in L^{\infty}(\Omega) \). Finally, in case \textbf{(ii)}, by repeating the argument of case \textbf{(i)}, the desired result is established.
\end{proof}
\begin{proof}[\textbf{Proof of Theorem} \ref{Theorem5}]
Let \( \alpha > 1 \) and assume that \( f \in L^{r}(\Omega) \) with \( 1 \leq r < \frac{N}{sq} \). Let \( u_n \) denote the weak solution to \eqref{P}, as established in Lemma \ref{lemma2}. By employing the same arguments as in Theorem \ref{Theorem3} - \textbf{Part(1)}, we deduce that the sequence \((u^{\vartheta_{r}}_{n})\) is uniformly bounded in \( W^{1, p}_{0}(\Omega) \), where \( \vartheta_{r} \) is defined as in \eqref{equ68} and it is clear that \( \vartheta_{r} \geq 1 \). Consequently, we infer that \( (u_{n}) \) is uniformly bounded in \( W^{1, p}_{\text{loc}}(\Omega) \). Indeed, from Lemma \ref{lemma2} - \textbf{(3)}, for every compact subset \( \widetilde{\Omega} \Subset \Omega \), we have
$$ \int_{\widetilde{\Omega}} \left| \nabla u_{n}\right| ^{p} dx\leq C \text{dist}^{p(1 - \vartheta_{r})} (\widetilde{\Omega}, \Omega) \int_{\widetilde{\Omega}} u_{n}^{p(1- \vartheta_{r})} \left| \nabla u_{n}\right| ^{p} dx \leq C \int_{\widetilde{\Omega}} \left| \nabla u_{n}^{\vartheta_{r}}\right| ^{p}dx < C \text{  (independent of $ n $) }. $$
Then, up to a subsequence, we have
\begin{equation*}
u_{n} \rightharpoonup u \text{ in } W^{1, p }_{\text{loc}}(\Omega), \quad u_{n} \to u \text{ in } L^{p}_{\text{loc}}(\Omega), \quad \text{and} \quad u_{n}(x) \to u(x) \text{ a.e. in } \Omega.
\end{equation*}
Now, to pass to the limit in \eqref{equ45}, let \( \varphi \in C^{\infty}_{c}(\Omega) \). \\[4pt]
\noindent $ \bullet $ For the local term on the left-hand side of \eqref{equ45}, we obtain the following:
\begin{equation*} 
\begin{aligned}
\left| \int_{\Omega} \left( \left[ \nabla u_{n} \right]^{p-1} - \left[ \nabla u \right]^{p-1} \right) \nabla \varphi \, dx \right| 
& \leq \left\| \nabla \varphi \right\|_{L^{\infty}(\Omega)} \left( \int_{\text{supp} (\varphi)} \left| \nabla u_{n} \right|^{p-1} dx + \int_{\text{supp} (\varphi)} \left| \nabla u \right|^{p-1} dx \right) \\[4pt]
& \leq C \left\| \nabla \varphi \right\|_{L^{\infty}(\Omega)} \left( \left\| u_{n} \right\|_{W^{1, p}_{\text{loc}}(\Omega)}^{p-1} + \left\| u \right\|_{W^{1, p}_{\text{loc}}(\Omega)}^{p-1} \right) \leq C,
\end{aligned}
\end{equation*}
for some constant $ C > 0 $ independent of $ n $.  On the other hand, from \eqref{equ47}, it is evident that \( (u_{n}) \) remains bounded away from zero in the interior of the domain $ \Omega $. Thus, we can apply the results from \cite[Theorem 2.1 \& Remark 2.2]{dal1998almost} to conclude that \( \nabla u_{n}(x) \to \nabla u(x) \) a.e. in \( \Omega \) as \( n \to \infty \). Consequently, we have
$$  \left( \left[ \nabla u_{n} \right]^{p-1} - \left[ \nabla u \right]^{p-1} \right) \nabla \varphi  \to 0 \quad \text{a.e. in } \Omega  \quad \text{as} \quad n \to \infty. $$
By Vitali's Convergence Theorem, we obtain
\begin{equation}\label{equ80}
\int_{\Omega} \left[ \nabla u_{n} \right]^{p-1} \nabla \varphi \, dx \to \int_{\Omega}  \left[ \nabla u \right]^{p-1} \nabla \varphi \, dx \quad \text{as} \quad n \to \infty.
\end{equation}
\noindent $ \bullet $ For the non-local term, we adapt the proof techniques from \cite[Theorem 3.6]{ref01} (see also the steps in passing to the limits in the proof of Theorem \ref{Theorem3}, \textbf{Part (2)}) to obtain as $ n \to \infty $:
\begin{equation}\label{equ81}
\displaystyle\iint_{\mathbb{R}^{2 N}} \dfrac{\left[ u_{n}(x) - u_{n}(y)\right]^{q-1}\left(\varphi(x) - \varphi(y)\right)}{\vert x-y\vert ^{N+ s q}} dx \, dy \to \displaystyle\iint_{\mathbb{R}^{2 N}} \dfrac{\left[ u(x) - u(y)\right]^{q-1} \left(\varphi(x) - \varphi(y)\right)}{\vert x - y\vert ^{N + s q}} dx \, dy.
\end{equation}
$ \bullet $ For the right-hand side of \eqref{equ45}, using the facts that \( f_{n} \leq f \) and \eqref{equ47}, we have
\[
\left| f_{n} \left( u_{n} + n^{-1} \right)^{-\alpha} \varphi \right| \leq \text{dist}^{-\alpha}(\text{supp}(\varphi), \Omega) \left\| \varphi \right\|_{L^{\infty}(\Omega)} f \in L^{1}(\Omega),
\]
which allows us, by applying the Lebesgue Dominated Convergence Theorem, to obtain
\begin{equation}\label{equ82}
\int_{\Omega} f_{n}(x) \left( u_{n} + n^{-1} \right)^{-\alpha} \varphi \, dx  \to  \int_{\Omega} f(x) u^{-\alpha} \varphi \, dx  \quad \text{as} \quad n \to \infty.
\end{equation}
Next, utilizing \( g_{n} \leq g \), we derive
\begin{equation*}
\int_{\Omega} g_{n}(x) u_{n}^{\beta} \varphi \, dx \leq C \left\| \varphi \right\|_{L^{\infty}(\Omega)} \left\| g \right\|_{L^{\left(\frac{p^{*}}{\beta + 1}\right)'}(\Omega)} \left\| u_{n} \right\|^{\beta}_{W^{1, p}_{\text{loc}}(\Omega)} < C.
\end{equation*}
Thus, by applying Vitali's Convergence Theorem, we obtain
\begin{equation}\label{equ83}
\int_{\Omega} g_{n}(x) u_{n}^{\beta} \varphi \, dx  \to  \int_{\Omega} g(x) u^{\beta} \varphi \, dx  \quad \text{as} \quad n \to \infty.
\end{equation}
Finally, gathering \eqref{equ80}, \eqref{equ81}, \eqref{equ82}, and \eqref{equ83}, and passing to the limit in \eqref{equ45} as \( n \to \infty \), we infer that
\begin{equation*}
\begin{aligned}
&\int_{\Omega} \left[ \nabla u \right]^{p-1} \nabla \varphi \, dx + \iint_{\mathbb{R}^{2N}} \frac{\left[ u(x) - u(y) \right]^{q-1} \left( \varphi(x) - \varphi(y) \right)}{\vert x - y \vert^{N + sq}} \, dx \, dy= \int_{\Omega} f(x) u^{-\alpha} \varphi \, dx + \int_{\Omega} g(x) u^{\beta} \varphi \, dx.
\end{aligned}
\end{equation*}
By utilizing the same arguments as in Theorem \ref{Theorem3} - \textbf{Part (2)}, we deduce that the sequence \( (u_{n}) \) is uniformly bounded in \( L^{\frac{N r (\alpha + q - 1)}{N - sqr}}(\Omega) \). Thus, we conclude that \( u \in L^{\frac{N r (\alpha + q - 1)}{N - sqr}}(\Omega) \).
\end{proof}
\subsection{Proofs of the results in the context of singular weights}
\begin{proof}[\textbf{Proof of Theorem} \ref{theorem1}] 
$ \bullet $	Let \(\delta + \alpha < 1 + \frac{1}{p'}\). By testing the weak formulation \eqref{equ45} with \(u_{n}\), we obtain
	\begin{align*}
	\left\| u_{n} \right\| _{W^{1, p}_{0}(\Omega)}^{p} + \left\| u_{n} \right\| _{W^{s, q}_{0}(\Omega)}^{q} &= \int_{\Omega} f_{n}(x) \left(u_{n} + n^{-1}\right)^{-\alpha} u_{n} \, dx + \int_{\Omega} g_{n}(x) u_{n}^{\beta + 1} \, dx.
	\end{align*}
	By applying \eqref{equ44} and \eqref{equ47}, together with H\"{o}lder's inequality and Hardy's inequality, we obtain
	\begin{align*}
	\left\| u_{n} \right\| _{W^{1, p}_{0}(\Omega)}^{p} &\leq C \left( \int_{\Omega} d^{p'(1 - \alpha - \beta)}(x) \, dx \right)^{\frac{1}{p'}} \left\| u_{n} \right\| _{W^{1, p}_{0}(\Omega)} + \left\| g \right\| _{L^{\left(\frac{p^{*}}{\beta + 1}\right)'}(\Omega)} \left\| u_{n} \right\|^{\beta + 1}_{W^{1, p}_{0}(\Omega)},
	\end{align*}
	where $ C $ is independent of  $ n. $ Thus, from this estimate, and since \( \beta < q - 1 \) and
	$$ \delta + \alpha < 1 + \frac{1}{p'} \Rightarrow p'(1 - \alpha - \delta) > -1 \Rightarrow \int_{\Omega} d^{p'(1 - \alpha - \beta)}(x) \, dx < \infty, $$
	the sequence \( (u_{n}) \) is uniformly bounded in \( W^{1, p}_{0}(\Omega) \). Hence, there exists \( u \in W^{1, p}_{0}(\Omega) \) such that
	\begin{equation}\label{equ84}
	u_n \rightharpoonup u \quad \text{in } W^{1, p}_{0}(\Omega), \quad u_n \to u \quad \text{in } L^{l}(\Omega) \text{ for } 1 \leq l < p^{*}, \quad \text{and} \quad u_n(x) \to u(x) \text{ a.e. in } \Omega.
	\end{equation}
Once again, by using \eqref{equ44} and \eqref{equ47}, in conjunction with H\"{o}lder's and Hardy's inequalities, we obtain
	\[
	f_{n} (u_{n} + n^{-1})^{- \alpha} \varphi \leq C \cdot \frac{1}{d^{\beta + \alpha - 1}} \cdot \left(\frac{\varphi}{d}\right) \in L^{1}(\Omega), \quad \text{for any } \varphi \in C^{\infty}_{c}(\Omega).
	\]
This result, combined with \eqref{equ61} and \eqref{equ84}, allows us to apply the same proof as in Theorem \ref{Theorem4} to pass to the limit in \eqref{equ45}, thereby establishing our claim.
\end{proof}
\begin{proof}[\textbf{Proof of Theorem} \ref{theorem3}]
Let \(\delta + \alpha \geq 1 + \frac{1}{p'}\). First, by using Lemma \ref{lemma2} - \textbf{(1)}, it can be readily verified that 
\[
u_{n}^{p(\theta -1) + 1} \in W^{1, p}_{0}(\Omega) \quad \text{for all} \quad
\theta > \max\left\lbrace 1, \frac{(\alpha + p - 1)(p - 1)}{p(p - \delta)}, \frac{\alpha + p -1}{p}\right\rbrace,
\]
and we will explain the reason for this choice. Thus, it qualifies as an admissible test function in \eqref{equ45}, yielding:
	\begin{equation*}
	\begin{aligned}
	&\int_{\Omega} \left[ \nabla u_{n}\right]^{p-1} \nabla u_{n}^{p(\theta -1) + 1} \, dx + \iint_{\mathbb{R}^{2N}} \dfrac{\left[u_{n}(x) - u_{n}(y)\right]^{q-1} \left(u_{n}^{p(\theta -1) + 1}(x) - u_{n}^{p(\theta -1) + 1}(y)\right)}{\vert x-y\vert^{N + sq}} \, dx \, dy \\[4pt]
	&\quad = \int_{\Omega} f_{n}(x) \left(u_{n} + n^{-1}\right)^{-\alpha} u_{n}^{p(\theta -1) + 1} \, dx + \int_{\Omega} g_{n}(x) \, u_{n}^{p(\theta -1) + 1 + \beta} \, dx.
	\end{aligned}
	\end{equation*}
	By applying \eqref{equ44}, performing some computations, and using the inequality from \cite[Lemma A.2]{ref07}, we obtain:
	\begin{align*}
	&\dfrac{p(\theta - 1) + 1}{\theta^{p}}\left\| u^{\theta}_{n}\right\| _{W^{1, p}_{0}(\Omega)}^{p} \leq C \int_{\Omega} d^{- \delta + \frac{p(\theta - 1) + 1 - \alpha}{\theta} }(x) \left(\dfrac{u_{n}^{\theta}}{d(x)}\right)^{\frac{p(\theta - 1)+ 1 -\alpha}{\theta}} dx \\
	&\qquad + \int_{\Omega} g(x) \, \left( u_{n}^{\theta}\right) ^{\frac{p(\theta -1) + 1 + \beta}{\theta}} \, dx \\[4pt]
	&\leq C \left(\int_{\Omega} d^{\frac{p(\theta(p - \delta) - (\alpha + p - 1))}{\alpha + p -1 }}(x) dx \right) ^{\frac{p + \alpha - 1}{p \theta}} 
	\left( \int_{\Omega}\left(\dfrac{u_{n}^{\theta}}{d(x)}\right)^{p} dx\right) ^{\frac{p (\theta - 1) + 1 - \alpha }{p \theta}}  \\[4pt]
	&\qquad + \left\| g\right\|_{L^{\left(\frac{p^{*} \theta}{p(\theta - 1) + 1 + \beta}\right)'}(\Omega)} \left(\int_{\Omega} (u_{n}^{\theta})^{p^{*}} dx\right) ^{\frac{p(\theta -1) + 1 + \beta}{p^{*} \theta}}.
	\end{align*}
	Since
	\[
	\theta > \frac{(\alpha + p - 1)(p - 1)}{p(p - \delta)} \Longrightarrow \frac{p(\theta(p - \delta) - (\alpha + p - 1))}{\alpha + p -1 } > -1 \Longrightarrow \int_{\Omega} d^{\frac{p(\theta(p - \delta) - (\alpha + p - 1))}{\alpha + p -1 }}(x) dx < \infty,
	\]
	and by Hardy's inequality and the Sobolev embedding theorem (see Theorem \ref{thm0}), we have
	\[
	\left\| u^{\theta}_{n}\right\| _{W^{1, p}_{0}(\Omega)}^{p} \leq C \left(\left\| u^{\theta}_{n}\right\| _{W^{1, p}_{0}(\Omega)}^{\frac{p (\theta - 1) + 1 - \alpha }{\theta}} + \left\| u^{\theta}_{n}\right\| _{W^{1, p}_{0}(\Omega)}^{\frac{p(\theta -1) + 1 + \beta}{\theta}}\right).
	\]
Thus, from the above estimate, and since \( \beta < q - 1 \), the sequence \( (u_{n})^{\theta} \) is uniformly bounded in \( W^{1, p}_{0}(\Omega) \). Consequently, we infer that \( (u_{n}) \) is uniformly bounded in \( W^{1, p}_{\text{loc}}(\Omega) \). Then, up to a subsequence, we have
\[
u_{n} \rightharpoonup u \text{ in } W^{1, p }_{\text{loc}}(\Omega), \quad u_{n} \to u \text{ in } L^{p}_{\text{loc}}(\Omega), \quad \text{and} \quad u_{n}(x) \to u(x) \text{ a.e. in } \Omega.
\]
At this point, the remainder of the proof follows precisely as in the conclusion of the proof of Theorem \ref{Theorem5}.
For \( \delta > p(1 - s) + s(1 - \alpha) \), we use the boundary behavior of the approximating sequence \( (u_{n}) \) from below, as established by Lemma \ref{lemma2} in \eqref{equ90}. By substituting \( \varphi = u_{n} \) into the weak formulation \eqref{equ45}, we obtain:
\begin{align*}
\left\| u_{n} \right\|_{W^{1, p}_{0}(\Omega)}^{p} & \leq C \int_{\Omega} d^{\frac{(1 - p)(\alpha + \delta - 1)}{\alpha + p - 1}}(x) \left( \frac{u_{n}}{d(x)} \right) dx + \int_{\Omega} g(x) u_{n}^{\beta + 1} dx \\[4pt]
& \leq C \left( \left( \int_{\Omega} d^{\frac{p'(1 - p)(\alpha + \delta - 1)}{\alpha + p - 1}}(x) dx \right)^{\frac{1}{p'}} \left\| u_{n} \right\|_{W^{1, p}_{0}(\Omega)} + \left\| g \right\|_{L^{\left(\frac{p^{*} \theta}{p(\theta - 1) + 1 + \beta}\right)'}(\Omega)} \left\| u_{n} \right\|_{W^{1, p}_{0}(\Omega)}^{\beta + 1}\right) \\[4pt]
& \leq C \left(\left\| u_{n} \right\|_{W^{1, p}_{0}(\Omega)} +  \left\| u_{n} \right\|_{W^{1, p}_{0}(\Omega)}^{\beta + 1}\right).
\end{align*}
This follows since
\[
\delta < 1 + \dfrac{1 - \alpha }{p'}\Longrightarrow \frac{p'(1 - p)(\alpha + \delta - 1)}{\alpha + p - 1} > -1 \Longrightarrow \int_{\Omega} d^{\frac{p'(1 - p)(\alpha + \delta - 1)}{\alpha + p - 1}}(x) dx < \infty.
\]
Consequently, the sequence \( (u_{n}) \) is uniformly bounded in \( W^{1, p}_{0}(\Omega) \). Utilizing this fact, we can pass to the limit with respect to \( n \) in \eqref{equ45}, thereby confirming our result. Finally, if \( \delta \geq 1 +  \frac{1 - \alpha}{p'} \), by applying the Hardy inequality and \eqref{equ90}, we obtain
\[
\left\| u \right\|_{W^{1, p}_{0}(\Omega)}^{p} \geq C \int_{\Omega} \left( \frac{u}{d(x)} \right)^{p} \, dx \geq C \int_{\Omega} u^{\frac{p(1 + \delta - \alpha)}{\alpha + p - 1}} \, dx = \infty.
\]
\end{proof}
\begin{proof}[\textbf{Proof of Theorem} \ref{theorem4}] 
Assume that \( \delta \geq p \). We choose \( \delta_0 \) such that \( s(1 - \alpha) + p(1 - s) < \delta_0 < p \) (sufficiently close to \( p \)), and we select \( M \in (0, 1) \), independent of \( \delta_0 \), where \( \delta_0 \geq \delta_0^{*} > 0 \), such that
\begin{center}
	\( M \tilde{f}(x) \leq f(x) \) with \( \textbf{C}_{1} d(x)^{-\delta_{0}} \leq \tilde{f}(x) \leq \textbf{C}_{2} d(x)^{-\delta_{0}} \text{ in } \Omega, \) 
\end{center}
where \( \textbf{C}_{1} \) and \( \textbf{C}_{2} \) are positive constants. By Lemma \ref{lemma2}, for every \( n \in \mathbb{N} \), there exists \( u_{n} \in C^{1, \zeta}(\overline{\Omega}) \cap W^{1, p}_{0}(\Omega) \), where \( \zeta \in (0, 1) \), such that \( u_{n} \) satisfies the following problem:
\begin{equation}\label{3equ1}
\begin{gathered}
\begin{aligned}
&\int_{\Omega} \left[ \nabla u_{n} \right]^{p-1} \nabla \varphi \, dx 
+ \iint_{\mathbb{R}^{2 N}} \dfrac{\left[ u_{n}(x) - u_{n}(y) \right]^{q-1} \left( \varphi(x) - \varphi(y) \right)}{\vert x - y \vert^{N + s q}} \, dx \, dy \\[4pt]
&\quad = M \int_{\Omega} \tilde{f}_{n}(x) \left( u_{n} + n^{-1} \right)^{-\alpha} \varphi \, dx 
+ \int_{\Omega} g_{n}(x) \, u_{n}^{\beta} \varphi \, dx, \quad \forall \varphi \in W^{1, p}_{0}(\Omega),
\end{aligned}
\end{gathered}
\end{equation}
with \( \tilde{f}_{n} \) defined as in \eqref{equ49} and \( g_{n}(x) = \textbf{T}_{n}(g(x)) \). Firstly, assume, by contradiction, that \( u \) is a weak solution of  \eqref{P2} in the sense of Theorem \ref{theorem1}, i.e., \( u \in W^{1, p}_{0}(\Omega) \). Then, by standard density arguments, it follows that
\begin{equation}\label{3equ2}
\begin{aligned}
&\int_{\Omega} \left[ \nabla u \right]^{p-1} \nabla \phi\, dx 
+ \iint_{\mathbb{R}^{2N}} \frac{\left[ u(x) - u(y) \right]^{q-1} \left(\phi(x) - \phi(y) \right)}{|x - y|^{N + sq}} \, dx \, dy \\
&= \int_{\Omega} f(x) u^{-\alpha}\phi\, dx 
+ \int_{\Omega} g(x) u^{\beta} \phi\, dx, \quad \forall \phi \in W^{1, p}_{0}(\Omega) \cap L^{\infty}(\Omega).
\end{aligned}
\end{equation}
By subtracting equations \eqref{3equ1} and \eqref{3equ2}, using the test function
\[
\varphi = \frac{\textbf{T}_{k}\left( \left( (u_{n} + m)^{q} - (u + m)^{q} \right)^{+} \right)}{(u_{n} + m)^{q - 1}}, \quad \phi = \frac{\textbf{T}_{k}\left( \left( (u + m)^{q} - (u_{n} + m)^{q} \right)^{-} \right)}{(u + m)^{q - 1}} \in W^{1, p}_{0}(\Omega) \cap L^{\infty}(\Omega),
\]
with \( m > 0 \), and given the monotonicity, it is apparent that we obtain
	\begin{equation*}
	\begin{gathered}
	\begin{aligned}
	&\int_{\Omega} \left( \left[ \nabla u_{n}\right] ^{p-1}\nabla\varphi -  \left[ \nabla u \right] ^{p-1}\nabla\phi\right) dx\\[4pt]
	& + \displaystyle \displaystyle\iint_{\mathbb{R}^{2N}}  \dfrac{\left[ u_{n}(x)  - u_{n}(y) \right] ^{q-1} \left(\varphi(x)-\varphi(y)\right) - \left[ u(x) - u(y) \right] ^{q-1} \left(\phi(x)-\phi(y)\right)}{\vert x-y\vert ^{N+ s q}} dxdy \\[4pt]
	&\leq M\int_{\left\lbrace u_{n} >u\right\rbrace } \tilde{f}_{n} (x) \left( \dfrac{u_{n}^{-\alpha}}{(u_{n} + m)^{q-1}} - \dfrac{u^{-\alpha}}{(u + m)^{q-1}} \right) \textbf{T}_{k}\left((u_{n} + m)^{q} - (u+m)^{q}\right) dx\\[4pt]
	& + \int_{\left\lbrace u_{n} >u\right\rbrace } g(x) \left( \dfrac{u_{n}^{\beta}}{(u_{n} + m)^{q-1}} - \dfrac{u^{\beta}}{(u + m)^{q-1}} \right) \textbf{T}_{k}\left((u_{n} + m)^{q} - (u+m)^{q}\right)dx \\[4pt]
	& \leq  \int_{\left\lbrace u_{n} >u\right\rbrace } g(x) \left( \dfrac{u_{n}^{\beta}}{(u_{n} + m)^{q-1}} - \dfrac{u^{\beta}}{(u + m)^{q-1}} \right) \textbf{T}_{k}\left((u_{n} + m)^{q} - (u+m)^{q}\right)dx.
	\end{aligned}
	\end{gathered}
	\end{equation*}
Now, by applying the same methods utilized in the proof of Claim 3 within the proof of Theorem \ref{Theorem1} (and also Theorem \ref{Theorem2}), we conclude that \( u_{n} \leq u \) in \( \Omega \). Taking into account the boundary behavior of the approximating sequence \( (u_{n}) \) from below, as provided by Lemma \ref{lemma2} in \eqref{equ90}, we obtain
	\begin{equation*}
	u(x) \geq c_{2}\left( \left( d(x) + n^{-\frac{\alpha + p -1}{p - \delta_{0}}}\right) ^{\frac{p - \delta_{0}}{\alpha +p -1} } - n^{-1}\right)  \text{ in } \Omega, \, \text{ for every }  n \in \mathbb{N}.
	\end{equation*}
Hence, by using Hardy's inequality and considering \( \delta_{0} \) close enough to \( p \), and taking \( n \to \infty \), we obtain
\begin{equation*}
\begin{gathered}
\begin{aligned}
\infty > \left\| u \right\|^{p}_{W^{1, p}_{0}(\Omega)} & \geq C \int_{\Omega} \left( \frac{u}{d(x)} \right)^{p} \, dx  \geq C \int_{\Omega} \left(\left( d(x) + n^{-\frac{\alpha + p - 1}{p - \delta_{0}}} \right)^{\frac{p - \delta_{0}}{\alpha + p - 1}} - n^{-1}\right)^{p}  d(x)^{-p}\, dx = \infty,
\end{aligned}
\end{gathered}
\end{equation*} 
which yields a contradiction. In the remaining case, suppose that \( u \) is a weak solution of the problem \eqref{P2} in the sense of Theorem \ref{theorem3}. Then, by applying standard density arguments, it follows readily that
\begin{equation}\label{3equ3}
\begin{aligned}
&\int_{\Omega} \left[ \nabla u \right]^{p-1} \nabla \phi \, dx 
+ \iint_{\mathbb{R}^{2N}} \frac{\left[ u(x) - u(y) \right]^{q-1} \left( \phi(x) - \phi(y) \right)}{|x - y|^{N + sq}} \, dx \, dy \\
&= \int_{\Omega} f(x) u^{-\alpha} \phi \, dx 
+ \int_{\Omega} g(x) u^{\beta} \phi \, dx, \quad \forall \phi \in W^{1, p}_{0}(\Omega) \cap L_{c}^{\infty}(\Omega).
\end{aligned}
\end{equation}
To select a test function as in the above case, we follow the main ideas outlined in \cite[Proof of Theorem 2.9]{ref48} (see also \cite[Proof of Theorem 1.3]{ref08}). Initially, we observe that due to the continuity of \( u_{n} \), for a given \( \xi > 0 \), there exists \( \eta = \eta(n, \xi) > 0 \) such that \( u_{n} \leq \frac{\xi}{2} \) in \( \Omega_{\eta} \), where \( \Omega_{\eta} := \{ x \in \Omega \mid d(x) < \eta \} \). Furthermore, we note that \( u_{n} - (u + \xi) \leq -\frac{\xi}{2} \) in \( \Omega_{\eta} \), since \( u > 0 \) in \( \Omega \). Consequently, we deduce that
\[
\text{supp}(\varphi) \text{ and } \text{supp}(\phi) \subset \Omega \setminus \Omega_{\eta} \Subset \Omega,
\]
where for \( m > 0 \):
\[
\varphi = \frac{\textbf{T}_{k}\left( \left( (u_{n} + m)^{q} - (u + \xi + m)^{q} \right)^{+} \right)}{(u_{n} + m)^{q - 1}}, \quad \phi = \frac{\textbf{T}_{k}\left( \left( (u + \xi + m)^{q} - (u_{n} + m)^{q} \right)^{-} \right)}{(u + \xi + m)^{q - 1}} \in W^{1, p}_{0}(\Omega) \cap L_{c}^{\infty}(\Omega).
\]
Now, by subtracting equations \eqref{3equ1} and \eqref{3equ3}, and using the test functions \( \varphi \) and \( \phi \), we apply the same methods utilized in the proof of Claim 3 within the proof of Theorem \ref{Theorem1} to infer that \( u_{n} \leq u + \xi \) in \( \Omega \). Then, by letting \( \xi \to 0 \), we obtain \( u_{n} \leq u \) in \( \Omega \). Hence, we deduce that
\begin{equation*}
\begin{gathered}
\begin{aligned}
\left\| u^{\theta_{0}} \right\|^{p}_{W^{1, p}_{0}(\Omega)} & \geq C \int_{\Omega} \left( \dfrac{u^{\theta_{0}}}{d(x)} \right)^{p} dx & \geq C \int_{\Omega} \left( \left( \left( d(x) + n^{-\frac{\alpha + p - 1}{p - \delta_{0}}} \right)^{\frac{p - \delta_{0}}{\alpha + p - 1}} - n^{-1} \right)^{\theta_{0}} d^{-1}(x) \right)^{p} dx,
\end{aligned}
\end{gathered}
\end{equation*}
where \( \theta_{0} \) is defined in \eqref{equ121}, which also leads to a contradiction since the last integral is not finite.
\end{proof}

\bibliographystyle{plain}
\bibliography{interactnlmsample}
\end{document}